\documentclass [12pt] {article}

\usepackage{amsfonts}    
\usepackage{amsmath}
\usepackage{amsthm}
\usepackage{amssymb} 
\usepackage{stackengine}
\usepackage{xcolor}
\usepackage{graphicx}
\usepackage{marvosym}
\usepackage{wasysym}
\usepackage{ dsfont }
\usepackage{tikz-cd} 
\usepackage{verbatim}
\usepackage{float}
\usepackage[toc,page]{appendix}
\theoremstyle{definition}
\newtheorem{defn}{Definition}[section]
\newtheorem{conj}{Conjecture}[section]
\newtheorem{ex}{Example}[section]
\newtheorem{thm}{Theorem}[section]
\newtheorem{prop}{Proposition}[section]
\newtheorem{cor}{Corollary}[section]
\newtheorem{lemma}{Lemma}[section]
\newtheorem{rmk}{Remark}[section]

\newcommand{\zz}{\mathbb{Z}}
\newcommand{\nn}{\mathbb{N}}

\usepackage[labelsep=endash]{caption}

\newcommand{\supp}{\text{supp}}

\newcommand\stacklongleftrightarrow[1]{%
    \mathrel{{\stackon[4pt]{$\longleftrightarrow$}{$\scriptscriptstyle#1$}}}}

\usepackage{color}

\usepackage[top=.75in, bottom=.75in, left=1in, right=1in]{geometry}
\title{Mixed Dimer Configuration Model in Type $D$ Cluster Algebras II: Beyond the Acyclic Case}
\author{Libby Farrell, Gregg Musiker, and Kayla Wright}
\date{\today}

\begin{document}

\maketitle

\begin{abstract}
This is a sequel to the second and third author's Mixed Dimer Configuration Model in Type $D$ Cluster Algebras where we extend our model to work for quivers that contain oriented cycles. Namely, we extend a combinatorial model for $F$-polynomials for type $D_n$ using dimer and double dimer configurations. In particular, we give a graph theoretic recipe that describes which monomials appear in such $F$-polynomials, as well as a graph theoretic way to determine the coefficients of each of these monomials. To prove this formula, we provide an explicit bijection between mixed dimer configurations and dimension vectors of submodules of an indecomposable Jacobian algebra module.
\end{abstract}

\begin{section}{Introduction}
After the positivity conjecture for the coefficients of Laurent polynomials for cluster variables was resolved in \cite{LeeSchiffler, ghkk}, many researchers have worked on trying to provide combinatorial interpretations for these coefficient sequences. Many particular classes of cluster algebras have been studied with this goal in mind. For example, cluster algebras coming from surfaces, first defined by \cite{fst}, have combinatorial interpretations for the cluster variables \cite{ms, mswpos}. In their model, each cluster variable has an associated graph called a snake graph and then the cluster expansion is given by some weighted generating function indexed over matchings on the snake graph. Since defining these snake graphs for cluster algebras from surfaces, many others have studied when such a construction holds in other contexts. For example, snake graphs have been defined for cluster-like algebras called quasi-cluster algebras arising from non-orientable surfaces \cite{w19}, cluster algebras from unpunctured orbifolds \cite{bk20} and super cluster algebras \cite{moz21}. \allowbreak \vspace{1em}

In this paper, we return to a classical case - cluster algebras of type $D_n$ or in other words, cluster algebras arising from a once-punctured $n$-gon. There are a few reasons that this is of interest. Firstly, there are some limitations to the snake graph formulation of cluster algebras from surfaces. For example, oftentimes there are nontrivial coefficients in the Laurent expansions for cluster variables that are given by the Euler characteristic of the quiver Grassmannian that are not recorded in a single snake graph. In addition to this, there are deep representation theoretic connections with the lattice of perfect matchings of snake graphs and the submodule lattice of a fixed indecomposable representation representation of the quiver of a given dimension vector that could use further exploring. For instance, a connection to between these lattices and the weak Bruhat order was recently discovered in \cite{cs21}.\allowbreak
\vspace{1em}

Another point of interest is that shedding light on the $D_n$ case can help us to study cluster algebras associated to coordinate rings of the Grassmannian. In these cluster algebras, a finite subset of cluster variables given by Pl\"ucker coordinates, admit a graphical combinatorial interpretations. More specifically, one can associate a planar bicolored graph embedded in a disk known as a plabic graph and the Laurent expansion is given by (almost) perfect matchings of this plabic graph as in \cite{lam2015dimers, marsh2016twists, muller2017twist, bkm16}.  However, combinatorial interpretations beginning with examples such as cluster algebras associated to the Grassmannian of $3$-planes in $6$-space, which is actually the type $D_4$ cluster algebra, still lack such interpretations.\allowbreak
\vspace{1em}

With the above as our inspiration, we further explore the connection between dimer configurations, representation theory and cluster algebras with the goal of providing a combinatorial interpretation for Laurent expansions of cluster variables that utilizes a mixture of dimer configurations and double dimer configurations. We ramp up previous work on quivers of type $D_n$ to the more complicated representation theoretic setting of allowing cycles in our quiver. More specifically, we focus on a single and double dimer configuration interpretation of the F-polynomial associated to a cluster variable or module over the associated Jacobian algebra. We provide a weighted generating function in terms of dimers and double dimers on a certain graph to give the F-polynomial for a particular cluster variable. We obtain the exact monomials by creating a bijection between these dimers and particular dimension vectors of submodules of a fixed indecomposable Jacobian algebra module and the coefficients are given by an Euler characteristic of the space of possible submodules with this given dimension vector. 

\allowbreak
\vspace{1em}

We begin this paper by reviewing cluster algebras from surfaces and representations of quivers in type $D_n$ in Section \ref{section:prelim}. We then describe our dimer theoretic interpretation of the F-polynomial in Section \ref{section:dimers} and our main result is Theorem \ref{thm:main} found in Section \ref{section:maintheorem}. Our results depend on a classification of the possible crossing vectors that can appear in type $D_n$ cluster algebras (which are notably no longer in bijection with positive roots when the quiver contains an oriented cycle). For a full catalog of such vectors, see Appendix \ref{section:d-vectors}.\allowbreak
\vspace{1em}

{\bf Acknowledgements:} The authors would like to thank Aaron Chan for helpful insights into the representation theoretic side of our paper as well as the support of the NSF, grants DMS-1745638 and DMS-1854162.
\end{section}

\begin{section}{Preliminaries}\label{section:prelim}
This section split into three subsections: Section \ref{subsection:surfacemodel} which discusses the surface model for type $D_n$ cluster algebras i.e. tagged triangulations of once-punctured $n$-gons and their associated type $D_n$ cluster algebras, Section \ref{subsection:principalcoefficients} which defines the $F$-polynomial associated to a cluster variables and Section \ref{subsection:alg} which defines the relevant representation theoretic notions we will need throughout the paper.

\begin{subsection}{Cluster Algebras from Punctured Surfaces}\label{subsection:surfacemodel}
In this section, we give a brief review of the cluster algebra structure on triangulated surfaces as defined by Fomin, Shapiro and Thurston in \cite{fst}. Since our focus is type $D_n$ cluster algebras, our exposition in this section will emphasize the role of punctures following both \cite{fst} and \cite{lf16}. More specifically, we will focus on the once-punctured disk; a helpful exposition of material for the single puncture case is given in \cite{dge2021}.

\begin{defn}\label{defn:markedsurface}
A \textbf{marked surface} is a pair $(S,M)$ where $S$ is a connected oriented Riemann surface and $M$ is a finite set of marked points such that there is at least one marked point on every boundary component of $S$. If a marked point is on the interior of $S$, we call it a puncture. 
\end{defn}

\begin{defn}\label{defn:arc}
An \textbf{arc} $\gamma$ is a curve in $S$, considered up to isotopy, such that
\begin{enumerate}
    \item its endpoints are in $M$;
    \item $\gamma$ is disjoint from $\partial S$, except for possibly its endpoints;
    \item $\gamma$ does not cut an unpunctured monogon or an unpunctured bigon.
\end{enumerate}
\end{defn}

We say that two arcs $\gamma, \delta$ are \textbf{compatible} if, up to isotopy, they do not intersect. More formally, let $e(\gamma, \delta)$ denote the minimal number of crossings between any isotopic representative of $\gamma$ and $\delta$. Then $\gamma$ and $\delta$ are compatible if $e(\gamma, \delta) = 0$. Maximal sets of pairwise compatible arcs are called \textbf{ideal triangulations} of $(S,M)$. For an example, see the right side of Figure \ref{fig:quiverfromtriangulation} for an ideal triangulation of the once-punctured pentagon. The elements of an ideal triangulation $T$, called ideal triangles, may not always have three distinct sides. Namely, we may have the case of triangles called self-folded triangles. We refer to the inner arc of a self-folded triangle as a radius and the outer edge wrapping around the puncture enclosing the radius as a loop. In Figure \ref{fig:quiverfromtriangulation}, the arcs 4 and 5 create a self-folded triangle, where 5 is the radius and 4 is the loop.\allowbreak
\vspace{1em}

Ideal triangulations are connected by ``quadrilateral flips;'' where the quadrilateral flip for any arc $\gamma \in T$ is given by the unique arc $\gamma' \neq \gamma$ that completes $T \setminus \{\gamma\}$ to a triangulation. For example, the resulting triangulation on the right in Figure \ref{fig:mutatedquiverfromtriangulation} is the mutation of arc 3 in Figure \ref{fig:quiverfromtriangulation}. This ``quadrilateral flip'' provides the cluster structure on triangulations of surfaces. In order to define the cluster structure, we associate a directed graph called a quiver to a triangulation as follows.

\begin{figure}
    \centering
    \includegraphics[width=.8\textwidth]{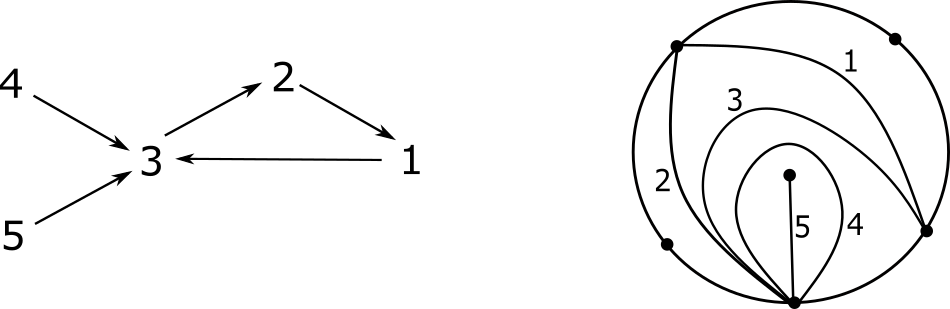}
    \caption{The quiver associated to a triangulation of the once-punctured pentagon.}
    \label{fig:quiverfromtriangulation}
\end{figure}

\begin{defn}
Let $T = \{\tau_1, \tau_2, \dots, \tau_n\}$ be an ideal triangulation of $(S,M)$, we define a \textbf{quiver} $Q_T = (Q_0, Q_1)$ with the vertex set $Q_0 = \{1,2, \dots, n\}$ such that each vertex $i$ is in correspondence with the arcs $\tau_i \in T$ and with arrow set $Q_1$ drawn in clockwise order for each ideal triangle of $T$. Given a self-folded triangle with radius $r$ and loop $\ell$, draw arrows connecting $r$ and $\ell$ to the adjacent arc(s), but not connecting $r$ to $\ell$. 
\end{defn}

To each arc $\tau_i$ in a triangulation $T$, we can associate an indeterminate $x_i$ to arc $\tau_i \in T$. Let $X_T = \{x_1, \dots, x_n\}$ associated to a triangulation $T$, then we call $(X,Q) = (X_T, Q_T)$ a \textbf{seed}.

\begin{figure}
    \centering
    \includegraphics[width=.95\textwidth]{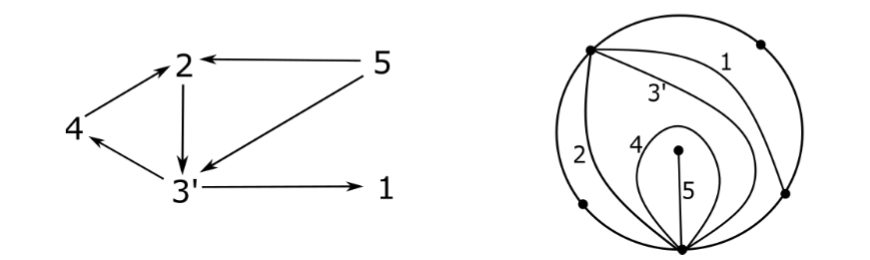}
    \caption{The mutation in direction 3 of quiver and triangulation from Figure \ref{fig:quiverfromtriangulation}.}
    \label{fig:mutatedquiverfromtriangulation}
\end{figure}

\begin{ex} \label{ex:quiverandtriangulation}
Consider the once-punctured pentagon triangulation on the right of Figure \ref{fig:quiverfromtriangulation}. The associated quiver is given by the quiver on the left of Figure \ref{fig:quiverfromtriangulation}.
\end{ex}

\begin{defn} \label{defn:mutation}
For each $1 \leq k \leq n$, define the \textbf{mutation in direction $k$} of the seed $(X,Q)$ by $\mu_k(X,Q) = ((X \setminus \{x_k\}) \cup \{x_{k'}\}, \mu_k(Q))$ where $x_{k'}$ is a rational function in $X \setminus \{x_k\}$ and $\mu_k(Q)$ is another quiver on $n$ vertices. More concretely, 

\begin{itemize}
    \item $x_{k'}x_k = \displaystyle \prod_{i \to k \in Q_1} x_i + \prod_{k \to j \in Q_1} x_j$;  
    \item $\mu_k(Q)$ is obtained from transforming $Q$ by the following three step process:

\begin{itemize}
    \item For any arrow incident to $k$, reverse the orientation;
    \item For any two path $i \to k \to j$, draw an arrow $i \to j$;
    \item Delete any created 2-cycles.
\end{itemize}
\end{itemize}

\end{defn}

\begin{ex}
For the quiver and triangulation from Example \ref{ex:quiverandtriangulation}, if we mutate the quiver in direction 3, we obtain the quiver shown in Figure \ref{fig:mutatedquiverfromtriangulation}. Moreover, $x_{3'} = \dfrac{x_1x_4x_5 + x_2}{x_3}$.
\end{ex}

\begin{defn}
Let $\mathcal{X}$ denote the union of all clusters obtainable by a sequence of mutations starting from a fixed initial seed $(X,Q)$. The \textbf{cluster algebra} $\mathcal{A}(X,Q)$ is the algebra generated by $\mathcal{X}$ over some ground ring $R$ i.e. $\mathcal{A}(X,Q) = R[\mathcal{X}]$. 
\end{defn}

One of the first remarkable properties proven about cluster algebras is the Laurent Phenomenon.

\begin{thm} {\cite[Theorem 3.1]{fzfoundations}} \label{thm:LP}
Any cluster variable $x \in \mathcal{A}$ can be expressed as 
$$x = \frac{N(x_1, \dots, x_n)}{x_1^{d_1} \cdots x_n^{d_n}}$$ 
where $N(x_1, \dots, x_n) \in \zz[x_{n+1}^{\pm 1}, \dots, x_{2n}^{\pm 1}][x_1, \dots, x_n]$ and is not divisible by any $x_i$. The \textbf{denominator vector} of $x$ is the vector $\underline{d} = (d_1, \dots, d_n)$.
\end{thm}

Cluster algebras from surfaces have topological formulations that are helpful for understanding the algebra. Let $(S,M)$ be an unpunctured marked surface and let $T$ be a triangulation of $(S,M)$. Let $Q_T$ be the quiver associated to the triangulation $T$ and let $X_T$ be the initial cluster associated to $T$. In the cluster algebra $\mathcal{A} = \mathcal{A}(X_T,Q_T)$, there are bijections
$$\{\text{cluster variables of }\mathcal{A}\} \longleftrightarrow \{\text{arcs of } (S,M)\}$$
$$\{\text{clusters of }\mathcal{A}\} \longleftrightarrow \{\text{triangulations of } (S,M)\}.$$

\noindent Moreover, let $\gamma \in T$ be an internal arc that is not in a self-folded triangle and let $\gamma'$ be the arc obtained by flipping $\gamma$ in $T$. Then, cluster mutation $\mu_\gamma(X_T)$ in direction $\gamma$ is compatible with flip $\mu_\gamma(T)$ in a quadrilateral of the arc $\gamma$, that is,

$$\mu_{x_\gamma}(X_T) = (X_T \setminus \{x_{\gamma}\}) \cup \{x_{\gamma'}\} $$

\noindent corresponds to 

$$\mu_\gamma(T) = (T \setminus \{\gamma\}) \cup \{\gamma'\}.$$

After observing the above correspondence, when $(S,M)$ is unpunctured, all triangulations are connected by flips and arcs are in bijection with cluster variables of the associated cluster algebra. However, the above correspondence becomes more complicated if $(S,M)$ contains a puncture. For example, the radii of self-folded triangles are not mutable and as a consequence, our set of arcs is not in direct bijection with cluster variables in the associated cluster algebra. To deal with these complications, Fomin, Shapiro, and Thurston introduced a decoration one can place on arcs called a \textbf{tag} that allows us to fix the problem of having immutable arcs. When one of the endpoints of an arc is a puncture, we can choose to decorate it with a tag $\Bowtie$ or leave it plain. We let $\gamma^{\Bowtie}$ denote the tagged version of the arc $\gamma$ which gives the tagged versions of self-folded triangles, see Figure \ref{fig:idealvstagged}. Since we are only considering a once-punctured surface, the only arcs the can be tagged are arcs with one endpoint the unique puncture; that is, an arc whose endpoints are not the puncture are equal to their plain version. \\

\begin{figure}
    \centering
    \includegraphics[scale=.3]{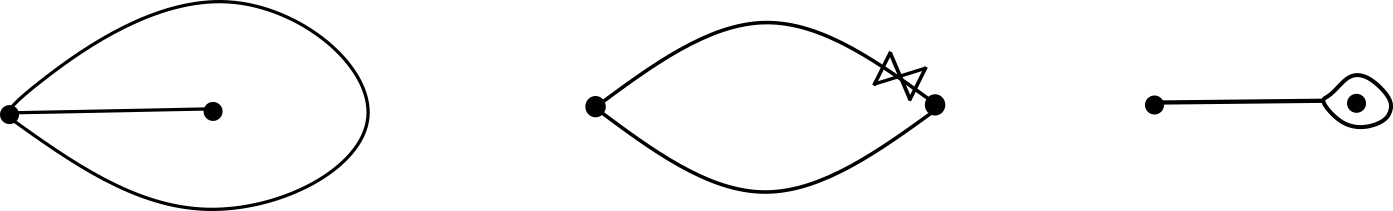}
    \caption{A self-folded triangle (on the left) corresponds to a tagged triangle (in the middle). A tagged arc can be interpreted as a ``lollipop'' as shown on the right.}
    \label{fig:idealvstagged}
\end{figure}

In order to define a tagged triangulation, we must define what it means for two arcs to be compatible with this new decoration. We define this notion specifically for the case of a once-punctured surface. 

\begin{defn}\label{defn:taggedcrossingvector}
Let $\gamma, \delta$ be two (possibly tagged) arcs. Then define the crossing number $e^{\Bowtie}(\gamma, \delta)$ as follows:
\begin{enumerate}
    \item if both arcs are plain, then $e^{\Bowtie}(\gamma, \delta) = e(\gamma, \delta)$;
    \item if both are tagged, then $e^{\Bowtie}(\gamma, \delta) = 0$;
    \item if exactly one is tagged, say $\gamma$, then $e^{\Bowtie}(\gamma, \delta) = e(\ell, \delta)$ where $\ell$ is the loop around the puncture.
\end{enumerate}

With that, we say two (possibly tagged) arcs $\gamma, \delta$ are compatible if $e^{\Bowtie}(\gamma, \delta) = 0$. A maximal set of pairwise compatible tagged arcs form a \textbf{tagged triangulation}. 
\end{defn}

With these definitions, we are now able to state the following result:

\begin{thm}\cite{fst}
Let $(S,M)$ be any marked surface except a once-punctured surface with empty boundary. Let $T$ be a triangulation of $(S,M)$ and $Q_T$ its associated quiver. Let $X_T$ be the initial cluster associated to $T$, in $\mathcal{A} = \mathcal{A}(X_T,Q_T)$, there are bijections
$$\{\text{cluster variables of }\mathcal{A}\} \longleftrightarrow \{\text{tagged arcs of } (S,M)\}$$
$$\{\text{clusters of }\mathcal{A}\} \longleftrightarrow \{\text{tagged triangulations of } (S,M)\}.$$
\end{thm}

Using this theorem, we can define a vector associated to each tagged arc of $(S,M)$ that will index our cluster variables.

\begin{defn}
Let $T = \{\tau_1, \dots, \tau_n\}$ be a triangulation of a marked surface $(S,M)$ and let $\gamma$ be an arc not in $T$. The \textbf{crossing vector associated to $\mathbf{\gamma}$}, denoted $\underline{d}_\gamma \in \zz_{\geq 0}^n$, is the vector defined by $d_i = e^{\Bowtie}(\tau_i, \gamma)$ for $1 \leq i \leq n$. That is, $\underline{d}_\gamma$ is the vector that records the number of times $\gamma$ crosses the arcs of $T$. 
\end{defn}

\begin{rmk}
When $Q_T$ is acyclic i.e. there are no internal triangles in the triangulation $T$, the crossing vector and denominator vector are the same. However, when $Q_T$ contains an oriented cycle, these vectors may not coincide. For example, see Figure 21 in \cite{fst}.
\end{rmk}

In this paper, we will be use initial triangulations that are ideal - as it makes the representation theory simpler. However, in light of the above theorem, we make sure to examine the full set of cluster variables by also considering tagged arcs. We will make sure these cases are always explicitly stated.\allowbreak
\vspace{1em}

In order to properly analyze all type $D_n$ cluster algebras, we heavily rely on Vatne's classification of all type $D_n$ quivers give in \cite{vatne}. Note that we will take these notations for the remainder of the paper.

\begin{thm}\label{thm:vatne} \cite{vatne}
When the underlying quiver is mutation-equivalent to the type $D_n$ Dynkin quiver, there are four forms the quiver can take. Namely, all type $D_n$ quivers must be of types I, II, III and IV as shown in Figure \ref{fig:vatnetypes}, where the subquivers labeled $Q, Q', Q'', \dots, Q^{(k)}$ are type $A_m$ quivers (that need not by acyclic). In type I, the arrows between $a$ and $c$ and between $b$ and $c$ can be oriented in either direction.
\end{thm}

We use this classification and the correspondence between cluster variables, (possibly tagged) arcs and crossing vectors as initial data in our combinatorial model. To this end, we cataloged all families of arcs that can appear on the once-punctured disk in Appendix \ref{section:d-vectors}. The catalog is an exhaustive list of all the crossing vectors that can appear in each type, with respect to Vatne's classification, of any ideal initial triangulation of the once-punctured disk. The Appendix boils down to understanding when crossing vectors are fully supported on oriented cycles on the quiver and which vectors degenerate to the acyclic case. Moreover, we emphasize which crossing vectors are composed of 0's and 1's and which crossing vectors have 2's as this distinction complicates both the combinatorics and representation theory, as we will be associating representations to each crossing vectors by insisting that the dimension vector of the representation and the crossing vector coincide.

\begin{figure}[H]
    \centering
    \includegraphics[width=\textwidth]{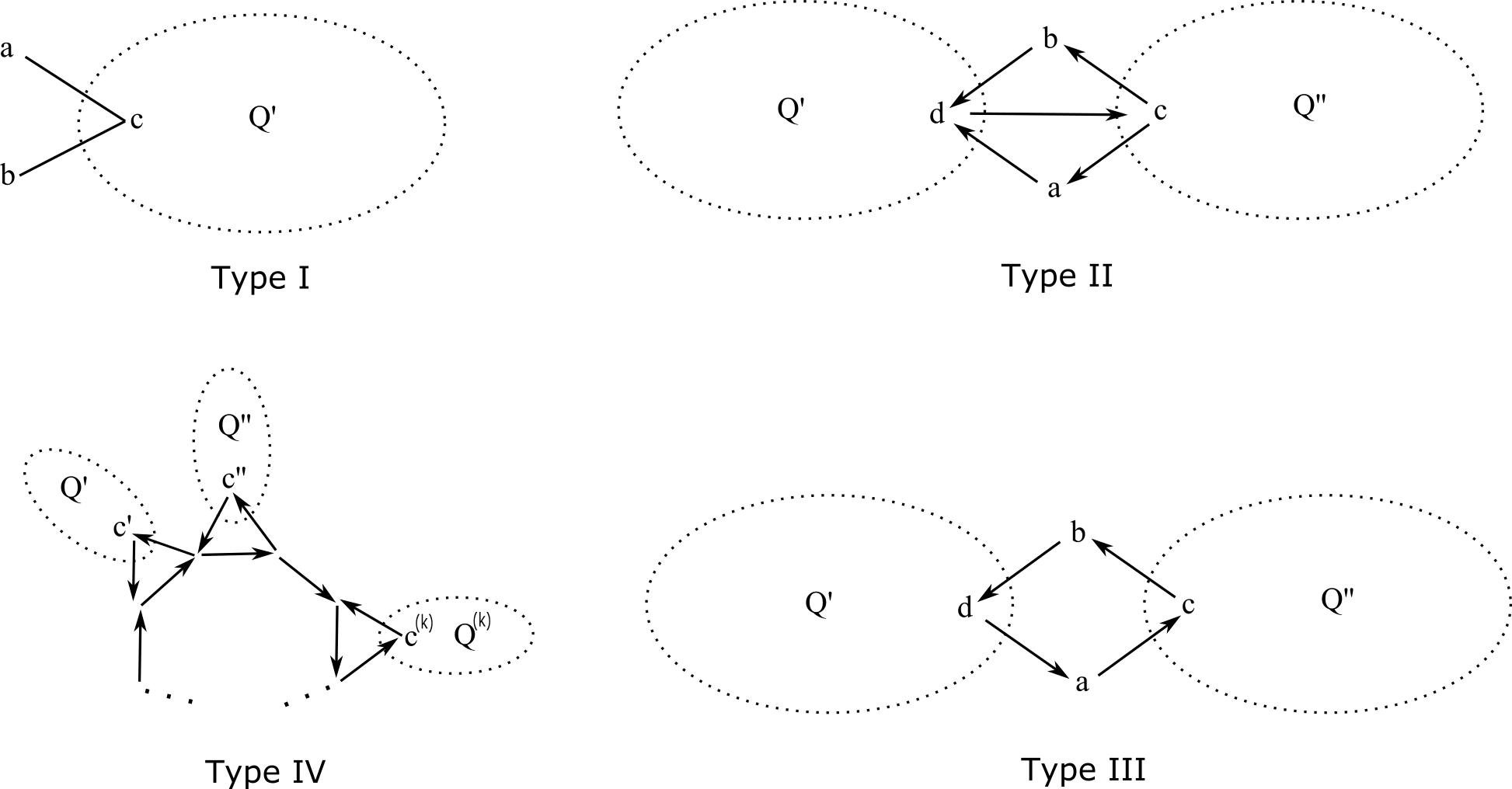}
    \caption{Vatne's classification of the four types of quivers in type $D_n$.}
    \label{fig:vatnetypes}
\end{figure}
\end{subsection}

\begin{subsection}{Principal Coefficients} \label{subsection:principalcoefficients}

Each crossing vector indexes a cluster variable associated to an arc $\gamma$. In this paper, we will give a combinatorial interpretation of the $F$-polynomial associated to a cluster variable $x_\gamma$ i.e. indexed by some vector $\underline{d}_\gamma$. In order to do this, we need to associate coefficients to our cluster algebra by adding what are called frozen vertices to the associated quiver. These are additional vertices of the quiver that record the dynamics of mutation without being mutable themselves. This follows the definitions and results about cluster algebras with coefficients as in \cite{ca4}.

\begin{defn}
A \textbf{framed (principal extension) quiver} $Q$ is a quiver on $2n$ vertices where there are $n$ mutable vertices $\{1, \dots, n\}$ and $n$ frozen vertices $\{1', \dots, n'\}$ such that each $i' \to i$ and no frozen vertices are connected in any other way.  
\end{defn}

Once we have the definition of a framed quiver, we associate an indeterminate $x_i$ to each mutable vertex $i \in Q_0$ and another indeterminate $y_j$ to each frozen vertex $j' \in Q_0$. Each set of $2n$ indeterminates form a \textbf{cluster} $\tilde{X} = \{x_1, \dots, x_n, y_1, \dots, y_n\}$. The pair $(\tilde{X},Q)$ is called an \textbf{extended seed}.

\begin{rmk}
The choice of frozen vertices in the quiver gives rise to cluster algebras with \textbf{principal coefficients.} We take this definition for our purposes in order to most conveniently define the $F$-polynomial to an associated cluster variable. 
\end{rmk}

To define the $F$-polynomial, we need to define the variables the polynomial is in. To this end, fix an initial seed $(X,Q)$, define the new variables for $1 \leq j \leq n$

$$\hat{y_j} = x_{n+j} \prod_{i=1}^n x_i^{\#\{i \to j \in Q_1\}}.$$

Using these new variables, we are ready to define the $F$-polynomial and also $g$-vectors. Although we state it as a definition, the following is also a theorem that cluster variables with principle coefficients have this form.

\begin{defn} \label{defn:fpoly}{\cite[Definition 3.3/Proposition 7.8]{ca4}}
Let $1\leq \ell \leq n$, there exists a unique primitive polynomial 
$F_{\ell} \in \zz[u_1, \dots, u_n]$
and a unique vector $\underline{g_{\ell}} = (g_1, \dots, g_n) \in \zz^n$ such that the cluster variable $x \in A( \tilde{X}, Q )$ is given by 

$$x = F_{\ell}(\hat{y_1}, \dots, \hat{y_n})x_1^{g_1}\cdots x_n^{g_n}.$$
The polynomial $F_{\ell}$ is called an \textbf{$\mathbf{F}$-polynomial} and $\underline{g_{\ell}}$ is called a \textbf{$\mathbf{g}$-vector}.
\end{defn}
\end{subsection}

\begin{subsection}{The Jacobian Algebra and Quivers with Potential Representations}\label{subsection:alg}

In this section, we review some representation theory that will dictate the behavior of our combinatorial model for the $F$-polynomial. Namely, we will define the Jacobian algebra of a triangulation of a punctured surface, first defined by Labardini-Fragoso in \cite{lf09}. For this section, assume that $\Bbbk $ is an (algebraically closed) field. 

\begin{defn}\label{defn:path}
Recall that we denote $Q = (Q_0,Q_1)$ to be a quiver where $Q_0$ is the set of vertices of $Q$ and $Q_1$ is the set of arrows of $Q$. For $i \xrightarrow\alpha$ $j$ in $Q_1$, let $s(\alpha) = i$ denote the source of the arrow $\alpha$ and let $t(\alpha) = j$ denote the target of $\alpha$. 
Let $\alpha_0, \dots, \alpha_p \in Q_1$. A \textbf{path} in $Q$ is a sequence $\alpha_0 \cdots \alpha_p$ where $t(\alpha_{i}) = s(\alpha_{i+1})$ for all $0 \leq i <p$. 
\end{defn}

\begin{defn}\label{defn:concatenation}
Let $\epsilon_i$ denote the path of length 0 at vertex $i$. Let $\alpha = \alpha_0 \cdots \alpha_p$ denote a path of length $p$ and $\beta = \beta_0\cdots\beta_q$ denote a path of length $q$. The concatenation of paths $\alpha \epsilon_i$ is given by 

$$\alpha \epsilon_i = \begin{cases}
\alpha &\text{if } t(\alpha_p) = i\\
0 &\text{otherwise}
\end{cases}$$

Similarly, the concatenation of $\epsilon_i \alpha$ is given by

$$\epsilon_i \alpha = \begin{cases}
\alpha &\text{if } s(\alpha_0) = i\\
0 &\text{otherwise}
\end{cases}$$

In general, the concatenation of $\alpha \beta$ is given by

$$\alpha \beta = \begin{cases}
\alpha_0 \dots \alpha_p \beta_0 \dots \beta_q &\text{if } t(\alpha_p) = s(\beta_0)\\
0 &\text{otherwise}
\end{cases}$$
\end{defn}

\begin{defn}\label{defn:pathalgebra}
The \textbf{path algebra of $Q$}, denoted $\Bbbk Q$ is the $\Bbbk $-algebra with basis paths in $Q$, including paths of length 0, and multiplication given by concatenation of paths. 
\end{defn}

\begin{ex}\label{example:pathalgebra}
Consider the following quiver:

\[\begin{tikzcd}
1 \arrow[r, "\alpha"] & 2 \arrow[d, "\beta"]\\
4 \arrow[u, "\delta"] & 3 \arrow[l, "\gamma"]
\end{tikzcd}\]

Then the path algebra $\Bbbk Q$ will be an infinite-dimensional algebra because of the oriented cycle $\alpha\beta\gamma\delta$. The basis will consist of elements such as

$$\{\epsilon_1, \epsilon_2, \epsilon_3, \epsilon_4, \alpha, \beta, \gamma, \delta, \alpha\beta, \beta\gamma, \gamma\delta, \delta\alpha, \alpha\beta\gamma, \beta\gamma\delta, \gamma\delta\alpha, \delta\alpha\beta,$$
$$\alpha\beta\gamma\delta, \beta\gamma\delta\alpha, \gamma\delta\alpha\beta,\delta\alpha\beta\gamma, \alpha\beta\gamma\delta\alpha, \beta\gamma\delta\alpha\beta, \gamma\delta\alpha\beta\gamma, \delta\alpha\beta\gamma\delta, \dots\}$$

\end{ex}

\begin{defn}
Let $R_Q$ be the ideal of $\Bbbk Q$ generated by the arrows of $Q$. More generally, denote by $R_Q^m$ the ideal of the path algebra $\Bbbk Q$ generated by paths of length $m$ in $Q$. A two-sided ideal $I$ of $A$ is \textbf{admissible} if there exists $m \geq 2$ such that $R_Q^m \subseteq I \subseteq R^2_Q$. If $Q$ is a quiver and $I$ is an admissible ideal of $\Bbbk Q$, then the pair $(Q, I)$ is a \textbf{bound quiver}.
\end{defn}

\begin{ex}\label{example:admissibleideal}
Let $(S,M)$ be an unpunctured marked surface and consider a triangulation $T$ of $(S,M)$. It is possible to associate an admissible ideal to the triangulation $T$, \cite{lf09, ABCJP10}. Let $I_T$ to be the ideal of $\Bbbk Q$ generated by all $2$-paths $\alpha\beta$ such that $s(\alpha) = t(\beta)$, but that $\alpha$ and $\beta$ arise from two different marked points in $(S,M)$; this ideal is admissible. On the level of the quiver $Q_T$, the ideal consists of all $2$-paths in $Q_T$ coming from the same triangle in $T$. 
\end{ex}

The admissible ideal referred to in Example \ref{example:admissibleideal} comes from a potential associated to a triangulation of a surface. We now define a potential of a quiver that arises from a triangulation of a surface.

\begin{defn} \label{defn:cycle}
Let $P = \alpha_0 \cdots \alpha_p$ be a path in $Q$. We say that $P$ is a \textbf{cycle} when $s(\alpha_0) = t(\alpha_p)$. 
\end{defn}

\begin{defn} \label{defn:potential}
A \textbf{potential} $W$ on $Q$ is a possibly infinite linear combination of cycles in $\Bbbk Q$. We call the pair $(Q,W)$ a \textbf{quiver with potential}.
\end{defn}

\begin{defn} \label{defn:cyclicequivalence}
Two potentials $W$ and $W'$ on $Q$ are \textbf{cyclically equivalent} when $W-W'$ lies in the closure of the span of all elements of the form $\alpha_1 \cdots \alpha_p - \alpha_2 \cdots \alpha_p \alpha_1$, where $\alpha_1 \cdots \alpha_p$ is a cycle. 
\end{defn}

\begin{defn} \label{defn:potentialofT}
When a quiver comes from a triangulation $T$ of a once-punctured surface, define the \textbf{potential associated to $\mathbf{T}$} by
$$W(T) = \sum_\Delta W^\Delta + W^P,$$
\noindent where the sum is taken over $\Delta$ - a clockwise oriented triangle in the quiver i.e. an internal triangle in $T$; $W^\Delta$ is the potential given by the cycle created by the triangle $\Delta$; and $W^P$ is the potential given by the counterclockwise cycle around the puncture P as in \cite{lf09}.
\end{defn}

\begin{figure}
    \centering
    \includegraphics[width=\textwidth]{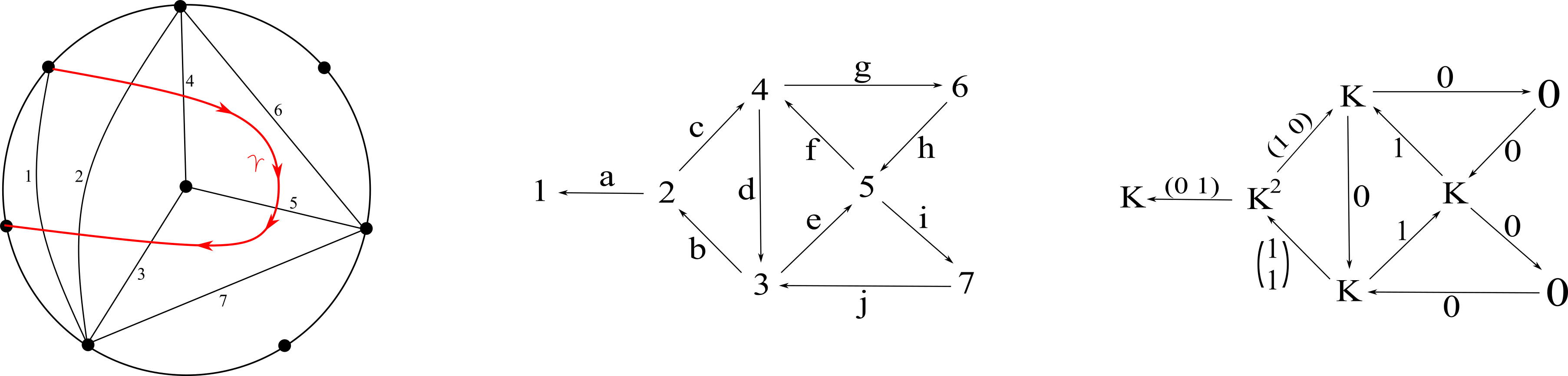}
    \caption{A triangulation with overlayed arc $\gamma$, quiver with potential and an the indecomposable quiver with potential representation $M_\gamma$.}
    \label{fig:puncturedrep}
\end{figure}

\begin{ex}\label{ex:quiverwithpotential}
Consider the ideal triangulation $T$ of the punctured pentagon and its associated quiver $Q_T$ in Figure \ref{fig:puncturedrep}. The potential is given by $W(T) = bcd + fgh+ eij + def$, where $W^\Delta = bcd + fgh+ eij$ which corresponds to the internal triangles formed by the arcs 2,4 and 3; 4,6 and 5; and 3,5 and 7 respectively. Additionally, $W^P = def$ corresponds to the counterclockwise cycle 3,4 and 5 coming from the puncture. 
\end{ex}

\begin{defn} \label{defn:cyclicderivatives}
Let $C = \alpha_0 \cdots \alpha_p$ be a cycle in $Q$. Then the \textbf{formal cyclic derivative} of $C$ with respect to $\alpha_i$ is 
\begin{align*}
\partial_{\alpha_i} (C) = \sum_{\alpha_j = \alpha_i} \alpha_{j+1} \cdots \alpha_{p} \alpha_{1} \cdots \alpha_{j-1}.
\end{align*}
The cyclic derivative extends linearly on linear combinations of cycles in $Q$. Note that the cyclic derivative of two cyclically equivalent potentials $W$ and $W'$ on $Q$ are equal.
\end{defn}

We are now ready to define a representation of a quiver with potential.

\begin{defn}\label{defn:quiverwithpotentialrep}
Let $(Q,W)$ be a quiver with potential. A \textbf{representation $M$ of $(Q,W)$} is a pair $M = (M_a, M_{\varphi_\alpha})$ where

\begin{itemize}
    \item $M_{a}$ is an assignment of $\Bbbk$-vector space $M_a$ to each a vertex $a \in Q_0$;
    \item $M_{\varphi_\alpha}$ is an assignment of a $\Bbbk$-linear map $\varphi_\alpha$ to each arrow $\alpha \in Q_1$ such that for each cyclic derivative $\partial_{a_i}$ of $W$, we insist the composition of the maps in the cyclic derivative is 0.
\end{itemize}

\end{defn}

Now that we have defined representations of quivers with potential, we re-contextualize this setting in terms of the Jacobian algebra.

\begin{defn} \label{defn:Jacobianideal}
Let $W$ be a potential on $Q$. The \textbf{Jacobian ideal} of $W$ is the ideal generated by all cyclic derivatives $\partial_{\alpha}(W)$ for $\alpha \in Q_1$. 
\end{defn}

With that, we define a quotient of the path algebra by this ideal.

\begin{defn}\label{defn:Jacobianalg}
Let $I$ be the Jacobian ideal of a potential $W$ in $Q$. The \textbf{Jacobian algebra} is the quotient $\Bbbk Q/ I$. If $W$ is a potential consisting of a sum of all cycles in $Q$ (up to cyclic equivalence), we refer to the corresponding Jacobian algebra as the Jacobian algebra of $Q$. 
\end{defn}

Both Definition \ref{defn:Jacobianideal} and \ref{defn:Jacobianalg} are well-defined up to cyclic equivalence. This is because the cyclic derivative of two cyclically equivalent potentials $W$ and $W'$ on $Q$ are equal; hence giving that the Jacobian ideals of $W$ and $W'$ are also equal. 
Moreover, since the Jacobian ideals of cyclically equivalent potentials $W$ and $W'$ on $Q$ are equal, the Jacobian algebras of $W$ and $W'$ are also equal. 

\begin{rmk}
A quiver with potential $(Q,W)$ is a bound quiver $(Q,I)$ using the admissible ideal $I$ generated by cyclic derivatives (with respect to each arrow of $Q_1$) of the potential $W$. Note that this ideal is admissible since each cycle in $W$ has positive length, and therefore, the ideal generated by a combination of these cycles must have bounded length. 
\end{rmk}

\begin{ex}\label{example:jacobianalg}
Consider the quiver from Example \ref{example:pathalgebra}. Endow this quiver with the potential $W = \alpha\beta\gamma\delta$. The Jacobian ideal is given by 
$$I = \langle \partial_{\alpha}(\alpha\beta\gamma\delta), \partial_{\beta}(\alpha\beta\gamma\delta), \partial_{\gamma}(\alpha\beta\gamma\delta), \partial_{\delta}(\alpha\beta\gamma\delta) \rangle = \langle \beta\gamma\delta, \gamma\delta\alpha, \delta\alpha\beta, \alpha\beta\gamma \rangle.$$

The Jacobian algebra $\Bbbk Q/I$ is a finite-dimensional algebra with basis 
$$\{\epsilon_1, \epsilon_2, \epsilon_3, \epsilon_4, \alpha, \beta, \gamma, \delta, \alpha\beta, \beta\gamma, \gamma\delta, \delta\alpha\}.$$
\end{ex}

\begin{defn}\label{defn:Jmodule}
Let $\Bbbk Q/I$ be the Jacobian algebra of $Q$. A \textbf{module $M = (M_a, M_{\varphi_\alpha})$ over $\Bbbk Q/I$} is given by two pieces of data

\begin{itemize}
    \item $M_{a}$ is an assignment of $\Bbbk$-vector space $M_a$ to each a vertex $a \in Q_0$;
    \item $M_{\varphi_\alpha}$ is an assignment of a $\Bbbk$-linear map $\varphi_\alpha$ to each arrow $\alpha \in Q_1$ such that for each relation $\alpha_1\cdots \alpha_r \in I$, we insist $\varphi_{\alpha_r} \cdots \varphi_{\alpha_1} = 0$. 
\end{itemize}

\end{defn}

We say such a module is \textbf{indecomposable} if it cannot be expressed as the direct sum of proper submodules. Studying indecomposable Jacobian algebra modules is equivalent to studying representations of the quiver with potential. More precisely, we have the following equivalence of categories:

\begin{thm} \cite{a06}, Theorem 1.6: \label{thm:equivofcats}
Let $A= \Bbbk Q/I$, where $Q$ is a finite connected quiver and $I$ an admissible ideal of $\Bbbk Q$. There exists an equivalence of categories between modules over $A$ and finite-dimensional representations of the bound quiver $(Q,I)$.
\end{thm}

\begin{ex}\label{example:jacobianalgmodule1}
Consider the Jacobian algebra computed in Example \ref{example:jacobianalg}. An example of an indecomposable Jacobian algebra module is given by 

\[\begin{tikzcd}
\Bbbk \arrow[r, "1"] & \Bbbk \arrow[d, "1"]\\
0 \arrow[u, "0"] & \Bbbk \arrow[l, "0"]
\end{tikzcd}\]

Note that the assignment of linear maps gives that a path given by the composition of any three consecutive arrows is 0 which is forced by the relations in the Jacobian ideal.
\end{ex}

\begin{ex}\label{example:jacobianalgmodule2}
Consider another quiver $Q$ shown below:

\[\begin{tikzcd}
 & 3 \arrow[dr, "\delta"] & 5 \arrow[l, swap, "\beta"]\\
1 \arrow[r, swap, "\nu"] & 2 \arrow[u, "\epsilon"] & 4 \arrow[u, swap, "\alpha"] \arrow[l, "\gamma"]
\end{tikzcd}\]

Then define the potential $W = \alpha\beta\delta + \epsilon\delta\gamma$. Taking cyclic derivatives of $W$, we obtain the Jacobian ideal $I = \langle \beta\delta, \delta\alpha, \delta\gamma, \epsilon\delta, \alpha\beta+\gamma\epsilon\rangle$. An example of an indecomposable module over the Jacobian algebra $\Bbbk Q/I$, or equivalently, a representation of $(Q,W)$ is given by

\[\begin{tikzcd}
 & \Bbbk \arrow[dr, "0"] & \Bbbk \arrow[l, swap, "1"]\\
\Bbbk \arrow[r, swap, "(1~0)^T"]  & \Bbbk^2  \arrow[u, "(0~1)"] & \Bbbk \arrow[u, swap, "1"] \arrow[l, "(0~1)^T"]
\end{tikzcd}\]

\end{ex}

Note that in Example \ref{example:jacobianalgmodule2}, we needed to make one of the arrows between non-zero vector spaces the zero map in order for it to satisfy the relations in the Jacobian ideal. Namely the arrow $3 \xrightarrow{\delta} 4$ was 0. We call such an arrow in a Jacobian algebra module \textbf{singular}. Such arrows will become relevant when defining our dimer model in Section \ref{section:dimers}.\allowbreak

\vspace{1em}

We now explain how to uniquely associate a quiver with potential or Jacobian algebra module to an arc in a triangulated surface following \cite{d17}. In particular, we emphasize that an arc overlayed on a triangulation of a surface uniquely determines an indecomposable Jacobian algebra module. \allowbreak
\vspace{1em}

Let $T$ be an ideal initial triangulation of a surface $(S,M)$ and let $(Q_T, W_T)$ be its associate quiver with potential. Overlay a plain oriented arc $\gamma \not\in T$ on $(S,M)$ and we explain how to produce a quiver with potential representation $M_\gamma$. The dimension vector of this representation is simply given by the crossing vector of $\gamma$; that is, 

$$\dim(M_\gamma)_i = \#\{ \text{intersections with }\gamma \text{ with arc }\tau_i \in T\}/\sim$$

\noindent where $\sim$ is up to isotopy. The maps are given by comparing the segments of $\gamma$ between arrows of the quiver. Let $\alpha: i \to k \in Q_1$, and let $r_1, \dots, r_s$ be the intersections of $\gamma$ with $\tau_i \in T$ and let $q_1, \dots, q_t$ be the intersections of $\gamma$ with $\tau_k \in T$. Let $[r_a,q_b]$ denote the segment of $\gamma$ between the intersection points $r_a$ and $q_b$. The map $\varphi_\alpha$ given by the matrix $N_{\varphi_\alpha} = (n_{a,b}) \in \text{Mat}_{s+t,s+t}(\mathbb{Z})$ is defined by

$$n_{b,a} = \begin{cases}
1 &\text{ if the interior of } [r_a,q_b] \text{ is contained in one triangle of }T\\
1 &\text{ if there is a segment } [q_b, q_{b'}] \text{ that surrounds the puncture clockwise }\\
&\text{ and the interior of }[r_a,q_{b'}] \text{ is contained in one triangle of }T\\
0 &\text{ otherwise}.
\end{cases}$$

\begin{ex}\label{ex:puncturerepresentation}
Consider the triangulation, and quiver with potential $W(T) = bcd + fgh+ eij + def$ from Example \ref{ex:quiverwithpotential} in Figure \ref{fig:puncturedrep}. Its Jacobian ideal is given by 

$$I_W = \langle cd, db, bc+ef, fd _ ij, de+gh, hf, hg,je,ei \rangle.$$

Consider the \textcolor{red}{red} arc $\gamma$ overlayed on $T$ with the specified orientation. It crosses arcs, 2,4,5,3,2 of $T$ in order -- giving that the crossing vector of $\gamma$ or the dimension vector of $M_\gamma$ is $\underline{d} = (0,2,1,1,1,0,0)$. The maps between the one-dimensional vector spaces are all the identity except the arrow $4 \to 3$ because the interior of the segment $[r_1,q_1]$ is contained in two triangles. This gives that this arrow is singular. \\

The arrow $2 \to 1$ is given by the matrix $(0~1)$ as the segment $[q_1,r_1]$ is not contained in a single triangle whereas the segment $[q_2,r_1]$ is. Similarly, the arrow $2 \to 4$ is given by the matrix $(1~0)$ as the segment $[q_1,r_1]$ is contained in a single triangle whereas the segment $[q_2,r_1]$ is not. The arrow $3 \to 2$ is given by the matrix $(1~1)^T$ as $[r_1,q_2]$ is contained in a single triangle and $[q_1,q_2]$ surrounds the puncture clockwise and the interior of $[r_1,q_2]$ is contained in a single triangle. The full representation is shown on the right of Figure \ref{fig:puncturedrep}.
\end{ex}

When $\gamma$ is tagged, the process for producing an indecomposable quiver with potential representation from $\gamma^{\Bowtie}$ is quite similar to the above. The dimension vector is given by the crossing vector as before as defined in Definition \ref{defn:taggedcrossingvector}. The maps are slightly modified, imagining replacing the tagged arc with a lollipop as in the rightmost picture in Figure \ref{fig:idealvstagged}. \allowbreak
\vspace{1em}
\end{subsection}

The one last piece of representation theory we will need is the notion of the $F$-polynomial in terms of quiver representations.

\begin{thm}\cite{dwz}
The $F$-polynomial as defined in Definition \ref{defn:fpoly} can be expressed in terms of finite dimensional modules over the Jacobian algebra. Namely, we have that for any finitely generated module $M$ of the Jacobian algebra that 

$$F_M = \sum_{\underline{e}} \chi(Gr_{\underline{e}}(M)) \underline{y}^{\underline{e}} $$

\noindent where the sum is taken over all dimension vectors $\underline{e}$ indexing submodules $N$ of $M$ and $\chi(Gr_{\underline{e}}(M))$ is the Euler characteristic of the quiver Grassmannian of all submodules $N \subset M$ with dimension vector $\underline{e}$.
\end{thm}

\end{section}

\begin{section}{Dimer Configurations} \label{section:dimers}

In this section, we describe a combinatorial model for how to obtain the $F$-polynomial associated to any type $D_n$ cluster algebra. Our model assigns a planar, bipartite graph to a type $D_n$ quiver. Each of the vertices in the graph correspond to a $2k$-gon, we call a tile, where $k$ is the degree of the vertex in the quiver and each arrow in the quiver gives an assignment of how to attach the tiles together. We use the data of the quiver $Q$ and a crossing vector $\underline{d}$ to assign a minimal mixed dimer configuration to this graph. From this assignment, we create a poset of mixed dimer configurations each of which corresponds to a monomial of the $F$-polynomial associated to $\underline{d}$ and whose coefficient is determined by the number of cycles appearing in the mixed dimer configuration. Our methods extend our previous model for the acyclic case \cite{prequel}, where we were able to utilize Thao Tran’s work \cite{tran}, to the non-acyclic case.
\allowbreak
\vspace{1em}

Let $Q= (Q_0,Q_1)$ be any type $D_n$ quiver. Recall there are four types of quivers that are mutation equivalent to the type $D_n$ Dynkin diagram categorized by Vatne in Figure \ref{fig:vatnetypes}. Let $\mathcal{A}(Q)$ be the associated cluster algebra. We first define the \textbf{base graph} associated to $Q$. 

\begin{subsection}{Defining the Base Graph}
\label{base graph}
\begin{defn} \label{defn:uncontractedbasegraph}
Let $i \in Q_0$ be a vertex of degree $k$ in our quiver. Associate a $2k$-gon called tile $i$ to this vertex. Each tile in an even-sided polygon, hence it admits a bipartite coloring. We attach the tiles based on this black and white coloring with the following convention:
\allowbreak  \vspace{1em}

If $i,j \in Q_0$ such that $i \to j \in Q_1$, attach tiles $i$ and $j$ by the convention that we ``see white on the right." Call the resulting graph $\Bar{G}$, the \textbf{uncontracted base graph associated to $Q$}.

\begin{ex}\label{ex:uncontractedbasegraph}
\begin{figure}
    \centering
    \includegraphics[width=\textwidth]{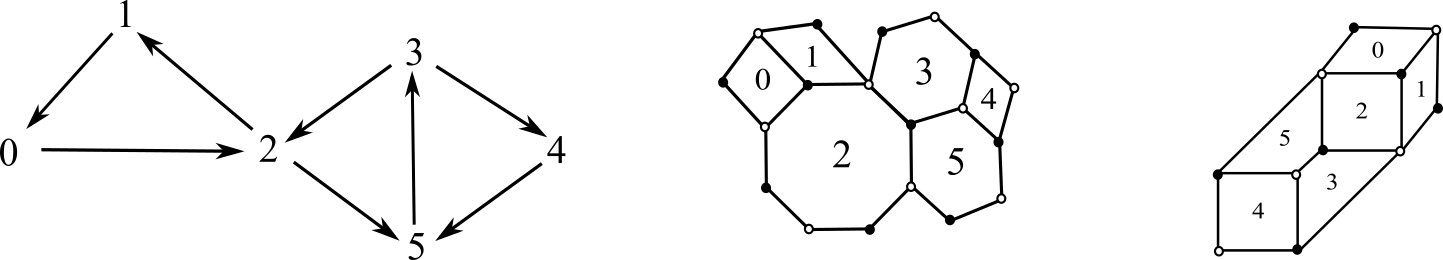}
    \caption{A quiver, its uncontracted and contracted base graphs.}
    \label{fig:basegraphexample}
\end{figure}
Consider the $D_7$ quiver on the left of Figure \ref{fig:basegraphexample}. The associated uncontracted base graph is given by the middle of Figure \ref{fig:basegraphexample}.
\end{ex}

\begin{rmk}\label{rmk:n-star}
For an $n$-cycle in $Q$, there is a more efficient process to obtain the uncontracted base graph. Namely, begin constructing the graph with an ``$n$-star" i.e. $n$ line segments attached at one of their endpoints and use these edges to create the rest of the tiles for vertices in the $n$-cycle in $Q$. If the $n$-cycle is clockwise, the vertex of the $n$-star is white. If the $n$-cycle is counterclockwise, the vertex of the $n$-star is black. 
\end{rmk}

From this graph $\Bar{G}$, we create a refinement of this graph via a local move called \textbf{double-edge contraction} that will result in a graph comprised of only squares and hexagons. 

\begin{defn} \label{defn: basegraph}
\textbf{Double-edge contraction} is a graph transformation that takes two edges in a graph $G$ and contracts them to a point. Locally, the transformation is:
\begin{center}
\begin{tikzpicture}
\node at (-2,0) (0){$\circ$};
\node at (-1,0) (1){$\bullet$};
\node at (0,0) (2){$\circ$};
\node at (3,0) (3){$\circ$};
\draw[-, thick] (0) to (1);
\draw[-, thick] (1) -- (2);
\end{tikzpicture}
\end{center}

\begin{center}
\begin{tikzpicture}
\node at (-2,0) (0){$\bullet$};
\node at (-1,0) (1){$\circ$};
\node at (0,0) (2){$\bullet$};
\node at (3,0) (3){$\bullet$};
\draw[-, thick] (0) to (1);
\draw[-, thick] (1) -- (2);
\end{tikzpicture}
\end{center}
\end{defn}

In $\Bar{G}$, for any tile $i$ that is a $2k$-gon for $k \geq 3$, perform double-edge contraction on any edges that are boundary edges i.e. edges that only belong to tile $i$ and neighbors no other tiles. Repeat this process as many times as possible. Call the resulting graph $G$ the \textbf{base graph associated to $Q$}.
\end{defn}

\begin{ex}\label{example:contractedbasegraph}
Taking the uncontracted base graph from Example \ref{ex:uncontractedbasegraph}, we perform double edge contraction which turns the octagon tile 2 into a square, and the hexagons 3 and 5 into squares. The resultant base graph is illustrated on the right of Figure \ref{fig:basegraphexample}.
\end{ex}

\begin{rmk}
Although this refinement of the uncontracted base graph $\Bar{G}$ to the base graph $G$ does not affect the combinatorics of our model, it yields a graph that is easier to work with in practice. 
\end{rmk}

\end{subsection}

\begin{subsection}{Minimal Mixed Dimer Configuration} \label{subsection: minimalmatching}
We now associate a minimal mixed dimer configuration to the base graph $G$ of a quiver $Q$. In order to define the minimal mixed dimer configurations, we first define single dimer, double dimer, and mixed dimer configurations. We tie this back to cluster algebras by using another piece of global data - the crossing vector $\underline{d}$ associated to a cluster variable $\underline{x}$ in $\mathcal{A}(Q)$. The complete catalog of all crossing vectors in type $D_n$ cluster algebras can be found in Appendix \ref{section:d-vectors}. \allowbreak
\vspace{1em}

For the following definitions, suppose that $G$ is an arbitrary planar bipartite graph.

\begin{defn}
 A \textbf{dimer configuration}, also known as a (perfect) matching, is a subset $D \subset E(G)$ such that every vertex $v \in V(G)$ is contained in exactly one edge $e \in D$.
\end{defn}
\noindent For example, consider the red edges in the following graph:

\begin{center}
    \includegraphics[scale=.25]{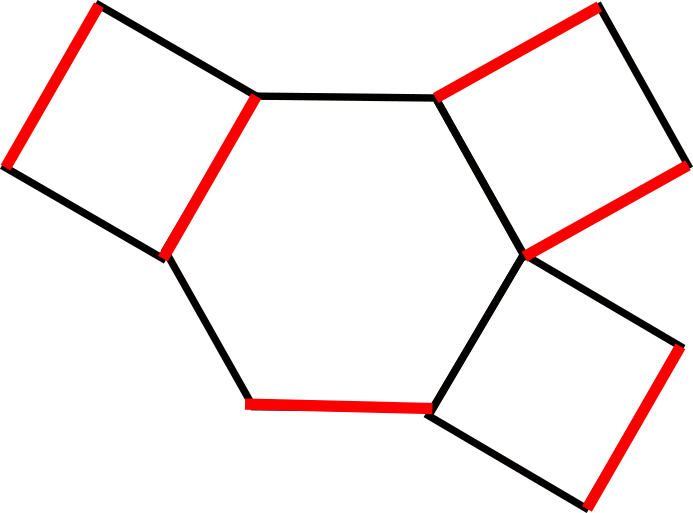}
\end{center}
 
\begin{defn} 
A \textbf{double dimer configuration} $D'$ of $G$ is a multiset of the edges of $G$ such that every vertex $v \in V(G)$ is contained in exactly two edges $e,e' \in D'$.
\end{defn}
\noindent For example, consider the red edges in the following graph:

\begin{center}
    \includegraphics[scale=.25]{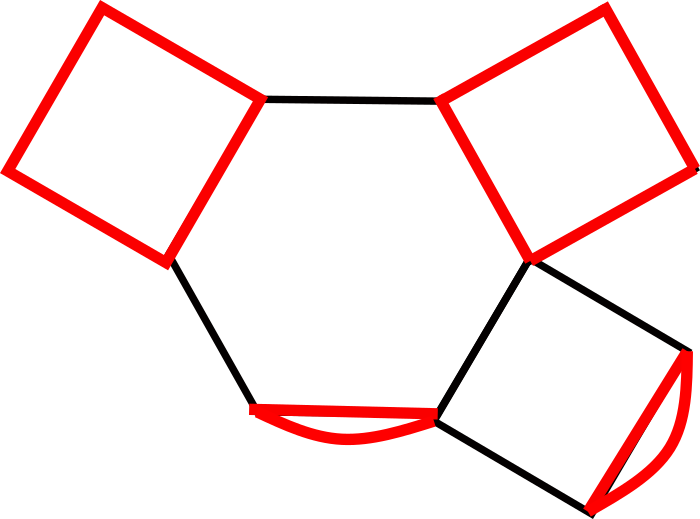}
\end{center}

\begin{defn} \label{def:mixed_dimer}
Let $\underline{d}$ be an $n$-tuple whose entries are each $0$, $1$, or $2$, i.e. $\underline{d} \in \{0,1,2\}^n$. A \textbf{mixed dimer configuration} $M = M(\underline{d)}$ of $G$ is a multiset of the edges of $G$ such that every vertex $v \in V(G)$ is contained in zero, one, or two edges in $M$. Furthermore, we say that $M$ satisfies the {\bf valence condition} with respect to $\underline{d}$ if
\begin{itemize}
    \item Each vertex incident to a tile labeled $i$ with $d_i = 2$ is contained in two edges in $M$.
    \item Each vertex incident to a tile labeled $i$ with $d_i = 1$ is contained in at least one edge in $M$.
\end{itemize}

\begin{ex}\label{ex:mixeddimer}
Let $\underline{d} = (d_0, d_1, d_2, d_3) = (2,2,1,0)$. An example of a mixed dimer configuration can be seen by the red edges highlighted in the following graph:
\begin{center}
\includegraphics[scale=.2]{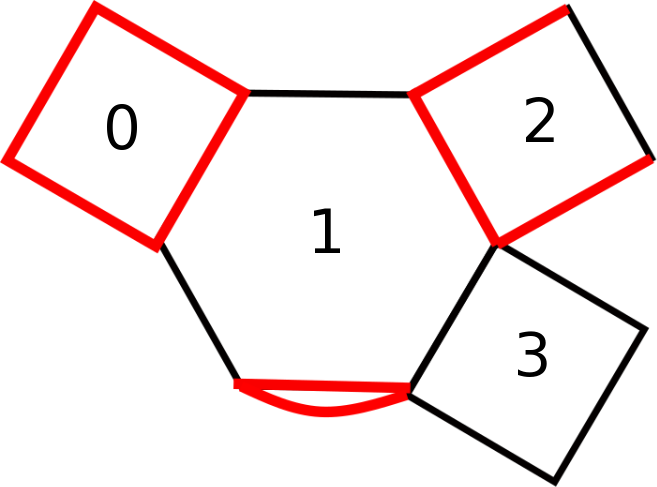}
\end{center}
\end{ex}
\end{defn}

The crossing vector $\underline{d}$ associated to some arc $\gamma$ in the once-punctured $n$-gon will be what determines the mixed dimer configurations in our model. Note that $\underline{d}$ is also the dimension vector of a unique representative of an isomorphism class of an indecomposable quiver with potential representation, $M_\gamma$. \allowbreak
\vspace{1em}

With this, in order to define the minimal matching, we present lemmas about the structure of crossing vectors $\underline{d}$ that will imply the well-definition of our construction. We postpone the proofs of Lemma \ref{lemma:d_vector_supp_connected} and Lemma \ref{lemma:d_vector_2} until Appendix \ref{section:d-vectors}. For the following lemmas, let $T$ be a triangulation of the once-punctured $n$-gon and $Q$ be the corresponding quiver.

\begin{lemma}\label{lemma:d_vector_supp_connected}
Suppose $\gamma$ is some arc not in $T$ and let $\underline{d} = \text{cross}(\gamma)$. $Q^{\supp(\underline{d})}$, the induced subquiver using vertices $i \in Q_0$ with $d_i > 0$, is connected.
\end{lemma}

\begin{lemma}\label{lemma:d_vector_2}
Suppose $\gamma$ is some arc not in $T$ such that there exists some arc $\tau \in T$ that $\gamma$ crosses twice. Let $\underline{d} = \text{cross}(\gamma)$. $Q^{\supp_2(\underline{d})}$, the induced subquiver using vertices $i \in Q_0$ with $d_i = 2$, is a connected tree.
\end{lemma}

Now, we aim to define a minimal mixed dimer configuration associated to $(Q,\underline{d})$. Let $G(Q) = G$ be the base graph obtained by the process described in Definitions \ref{defn:uncontractedbasegraph} and \ref{defn: basegraph}. Let $\gamma$ be an arc such that $cross(\gamma) = \underline{d}$ and let $M_\gamma$ be the corresponding indecomposable representation.

\begin{defn} \label{defn: minimalmatching}
We begin by addressing the nuance in the cyclic case. As we saw in Example \ref{example:jacobianalgmodule2}, some of the arrows in $M_\gamma$ may be 0 between nonzero vector spaces. In order to correctly model the combinatorics, our model must consider these singular arrows in the definition of the minimal matching. For any singular arrow $i \to j$ in $M_\gamma$, enumerate the edge $e_{i,j}$ straddling tiles $i$ and $j$. Define the set of such edges we distinguish $\displaystyle Z = \cup_{i \to j \text{singular}} e_{i,j}$.

Let $G_1$ be the induced subgraph of $G$ using tiles $i$ with $d_i \geq 1$. Traversing clockwise along the boundary of the graph $G_1$, distinguish the edges that go black to white clockwise, call this set of edges $M(G_1)$. If $\underline{d}$ contains no 2's, then define $M(G_1) \cup Z  = M_-(\underline{d})$. 

If $\underline{d}$ contains at least one 2, then let $G_2$ be the induced subgraph of $G$ using tiles $i$ with $d_i = 2$. Traversing clockwise along the boundary of the graph $G_2$, distinguish the edges that go black to white clockwise, call this set of edges $M(G_2)$. Define $M(G_1) \sqcup M(G_2) \sqcup Z = M_-(\underline{d})$. We refer to $M_-(\underline{d})$ as the \textbf{minimal matching associated to $(Q, \underline{d}$)}

\end{defn} 

\begin{figure}[H]
    \centering
    \includegraphics[width=\textwidth]{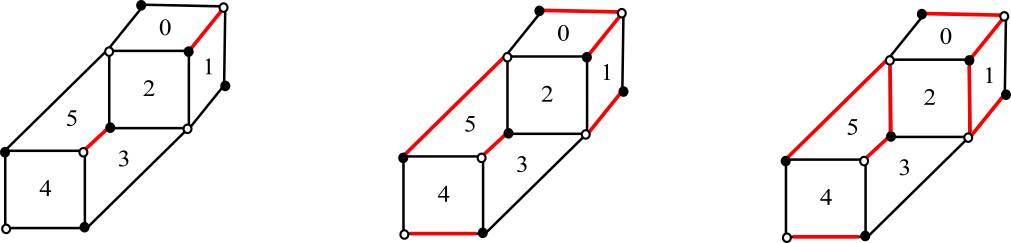}
    \caption{Phases of enumerating the edges in the minimal matching for Example \ref{example:contractedbasegraph}.}
    \label{fig:phasesminmatching}
\end{figure}

\begin{ex}\label{ex:minimalmatching}
Taking the base graph from Example \ref{example:contractedbasegraph}, and $\underline{d} = (1,1,2,1,1,1)$, the minimal matching $M_-(\underline{d})$ is shown in Figure \ref{fig:phasesminmatching}. The leftmost graphic shows the set $Z = e_{5,3} \cup e_{1,0}$, reflecting the singular arrows $5 \to 3$ and $1 \to 0$. The middle graphic represents then adding the edges in $M(G_1)$. The rightmost figure is the addition of the edges in $M(G_2)$ i.e. is the minimal matching for this choice of $(Q, \underline{d})$.
\end{ex}

Observe that the definition of the minimal matching $M_-(\underline{d})$ is well-defined. By Lemma \ref{lemma:d_vector_supp_connected}, we have that $G_1$ is connected; moreover, by Lemma \ref{lemma:d_vector_2}, we have that $G_2$ is connected. Hence, $M(G_1), M(G_2)$ are well-defined respectively.
\end{subsection}

\begin{subsection}{Poset of Mixed Dimer Configurations}
\label{sec:poset}
We now create the poset of mixed dimer configurations where each of the mixed dimer configurations appearing in this poset corresponds to a monomial that appears in the $F$-polynomial associated to $\underline{d}$. Namely, the minimal element of this poset corresponds to the minimal mixed dimer configuration defined in Section \ref{subsection: minimalmatching} and we define the poset relation as in Section 3.4 of \cite{prequel}. For more details and motivation, please refer to \cite{prequel}.

\begin{defn}\label{defn: posetrelation}
Fix a type $D_n$ quiver $Q$ and a base graph $G= G(Q)$. Let $\underline{d} \in \{0,1,2\}^n$ and let $D, D'$ be two mixed dimer configurations on $G$ with valence condition given by $\underline{d}$. We say $D$ covers $D'$, $D \lessdot D'$, if there exists a tile $i$ of $G$ such that $D'$ is obtained from $D$ by ``flipping" tile $i$ of $D$ i.e. exchanging highlighted edges black to white clockwise on tile $i$ in $D$ to white to black clockwise as shown in Figure \ref{fig:flip}. Implicit in this definition is that $d_i>0$.

\begin{figure}[H]
    \centering
    \includegraphics[scale=.25]{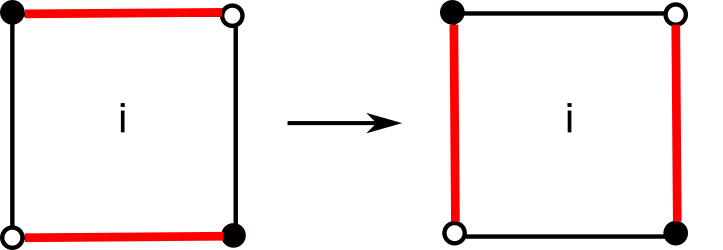}
    \caption{ }
    \label{fig:flip}
\end{figure}

\end{defn}

\begin{defn}\label{defn:largerposet}
Let $(\bar{P}, \leq)$ be the {\bf poset of mixed dimer configurations} that satisfy the valence condition and are reachable via a sequence of flips from $M_-(Q, \underline{d})$.  For two such mixed dimer configurations $D$ and $D'$, we say that $D \leq D'$ if there exists a sequence of allowable flips from $D$ to obtain $D'$.
\end{defn}

The poset $\bar{P}$ turns out to have more elements than the number of monomials in the $F$-polynomial associated to $\underline{d}$. As in \cite{prequel}, we put another condition on mixed dimer configurations that accurately reflects the cluster combinatorics by disallowing some mixed dimer configurations. We will let $(P, \leq) \subseteq (\bar{P}, \leq)$ be the subposet of mixed dimer configurations that satisfy the valence condition, are reachable via a sequence of allowable flips from $M_-$, and satisfy a condition known as being ``node monochromatic" we define in the next section.

\begin{subsubsection}{Node Monochromatic Mixed Dimer Configurations}
As in the acyclic case, we must define a special set of vertices of $G$ that we call ``nodes" to disallow certain mixed dimer configurations. These nodes allow us to disallow configurations that connect nodes of different color in order to correctly model the $F$-polynomial. In order to create paths in a mixed dimer configuration, we must have valence 2 on some vertices in $M_-(\underline{d})$. Hence, the definition of these nodes is only relevant when there is at least one 2 present in $\underline{d}$. 

Recall Vatne's classification of type $D_n$ quivers, described in Theorem \ref{thm:vatne}, which states that any type $D_n$ quiver consists of a type $D_n$ part and at least one type $A_m$ part. We first define nodes on the type $D_n$ part of the quivers. In types I, II and III, we define both \textcolor{red}{red} and \textcolor{blue}{blue} nodes; whereas in type IV, we just define \textcolor{blue}{blue} nodes. After this, we define \textcolor{olive}{green} nodes on the type $A_m$ part of the quiver, which applies to all types. 

\begin{rmk}
In the support of $\underline{d}$-vectors that contain a 2, types II and III degenerate to type I quivers when considering the induced subgraph on the quiver. In particular, the type $D_n$ parts of these quivers are acyclic on the support of $\underline{d}$. For details of how these types reduce to type I, refer to the classification of crossing vectors in these types found in Figures \ref{fig:2allfams} and \ref{fig:3allfams} in Appendix \ref{section:d-vectors}. So, in types I, II and III, nodes are defined in the same way and are very similar to the 6 nodes of three colors: \textcolor{red}{red}, \textcolor{blue}{blue} and \textcolor{olive}{green} as defined in \cite{prequel}. Namely, two pairs of these nodes are defined on the type $D_n$ part of the quiver, since tiles are disconnected in the induced base graph of the support of $\underline{d}$. 

However, in type IV quivers, the type $D_n$ part of the quiver that contains the central cycle is always fully supported in $\underline{d}$ when there is a 2 in the vector. So, in this case, we must define the nodes in a different way. We define 4 nodes of two colors: \textcolor{blue}{blue} and \textcolor{olive}{green} in this case because the type $D_n$ part of the induced subgraph is connected in this case. Hence, rather than having two pairs of nodes on the type $D_n$ part, we only have one pair of nodes in the type $D_n$ part. 
\end{rmk}

Let $\underline{d}$ be a crossing vector associated to a cluster variable $\underline{x}$ in $\mathcal{A}(Q)$. Let $Q_{\text{supp}}$ be the induced subquiver of $Q$ using vertices $i \in Q$ with $d_i > 0$ i.e. supported on $\underline{d}$. Suppose that $Q$ is type I, II or III. If $Q_{\text{supp}}$ contains an oriented cycle, then this cycle is necessarily in a type $A_m$ part of the model i.e. some $Q^{(\ell)}$ using Vatne's notation. By the surface model, the type $D_n$ part of the quiver in $Q_{\text{supp}}$ reduces to the case where we have a fork as in the classical Dynkin diagram. This allows us to define two pairs of nodes: two \textcolor{red}{red} nodes $\textcolor{red}{u, v}$ and two \textcolor{blue}{blue} nodes $\textcolor{blue}{w, x}$. Let $a$ and $b$ be the two forking vertices and let $c$ be the vertex that is connected to the rest of the type $A_m$ part of the quiver. Define the \textcolor{red}{red} nodes $\textcolor{red}{u,v}$ to be place on the two vertices of tile $a$ that are not shared with tile $c$. Define the \textcolor{blue}{blue} nodes $\textcolor{blue}{w,x}$ to be place on the two vertices of tile $b$ that are not shared with tile $c$. See Figure \ref{fig:tailnodes}. 

\begin{figure}[H]
    \centering
    \includegraphics[scale=.35]{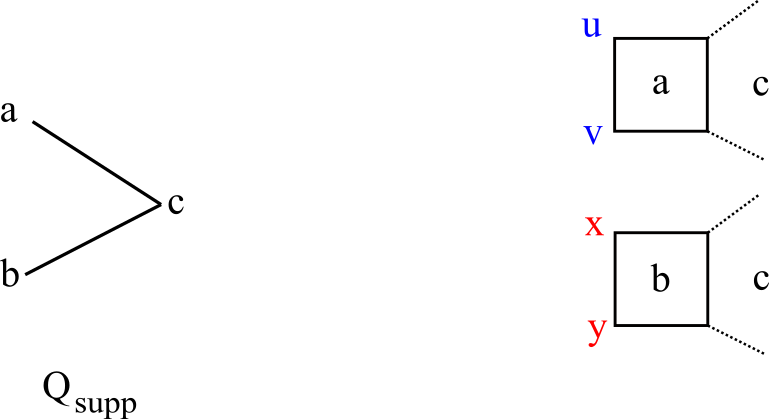}
    \caption{The \textcolor{blue}{blue} and \textcolor{red}{red} nodes on the tiles $a$ and $b$ in types I, II and III.}
    \label{fig:tailnodes}
\end{figure}

Now, we define the \textcolor{blue}{blue} nodes in type IV quivers. In the base graph, we create a $k$-star that corresponds to the vertex shared by all the tiles in the $k$-cycle. Define the vertex in the $k$-star to be the node $\textcolor{blue}{w}$. In the minimal matching, there is a unique vertex that is connected to $\textcolor{blue}{w}$ using the edges in $M_-$. Namely, the arrow $a \to b$ connecting the central $k$-cycle to the type $A_m$ spike must be a singular arrow - and we define $\textcolor{blue}{x}$ to be the vertex on tile $a$ adjacent to $\textcolor{blue}{w}$ but not shared by $b$. 

\begin{figure}[H]
    \centering
    \includegraphics[scale=.25]{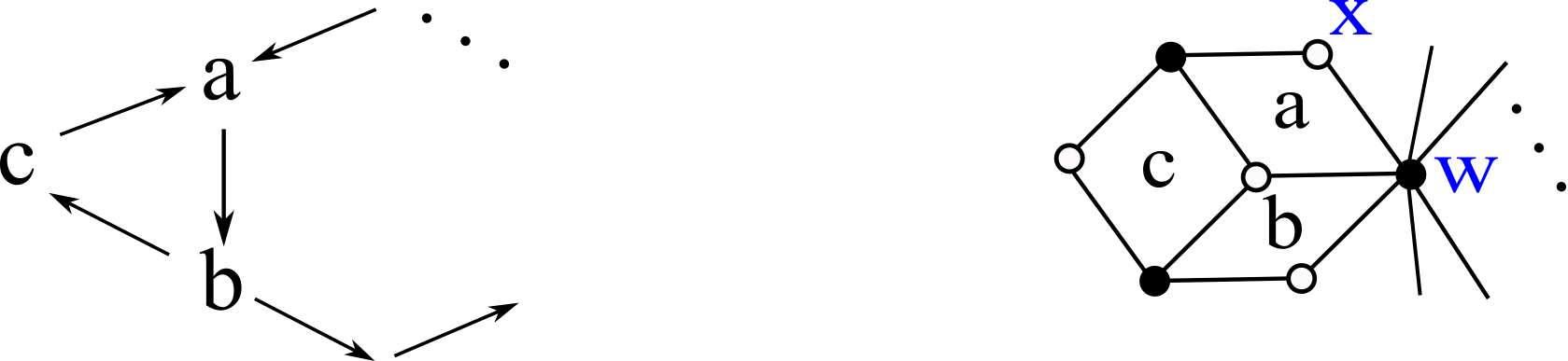}
    \caption{The \textcolor{blue}{blue} nodes on the tiles $a$ and $b$ in type IV.}
    \label{fig:typeIVnodesblue}
\end{figure}

Finally, to define the last pair of \textcolor{olive}{green} nodes, we focus on the type $A_m$ part of the quiver in any type. Let $i$ be the vertex in $Q_{\text{supp}}$ connected to the maximal number of vertices with $d_j = 0$ or $1$. If $i$ is not part of an oriented cycle, we define the \textcolor{olive}{green} nodes as in the acyclic case \cite{prequel}. If $i$ is part of an oriented cycle in $Q_{\text{supp}}$, then define the \textcolor{olive}{green} nodes $\textcolor{olive}{y,z}$ as in Figure \ref{fig:greennodes}.

\begin{figure}[H]
    \centering
    \includegraphics[scale=.25]{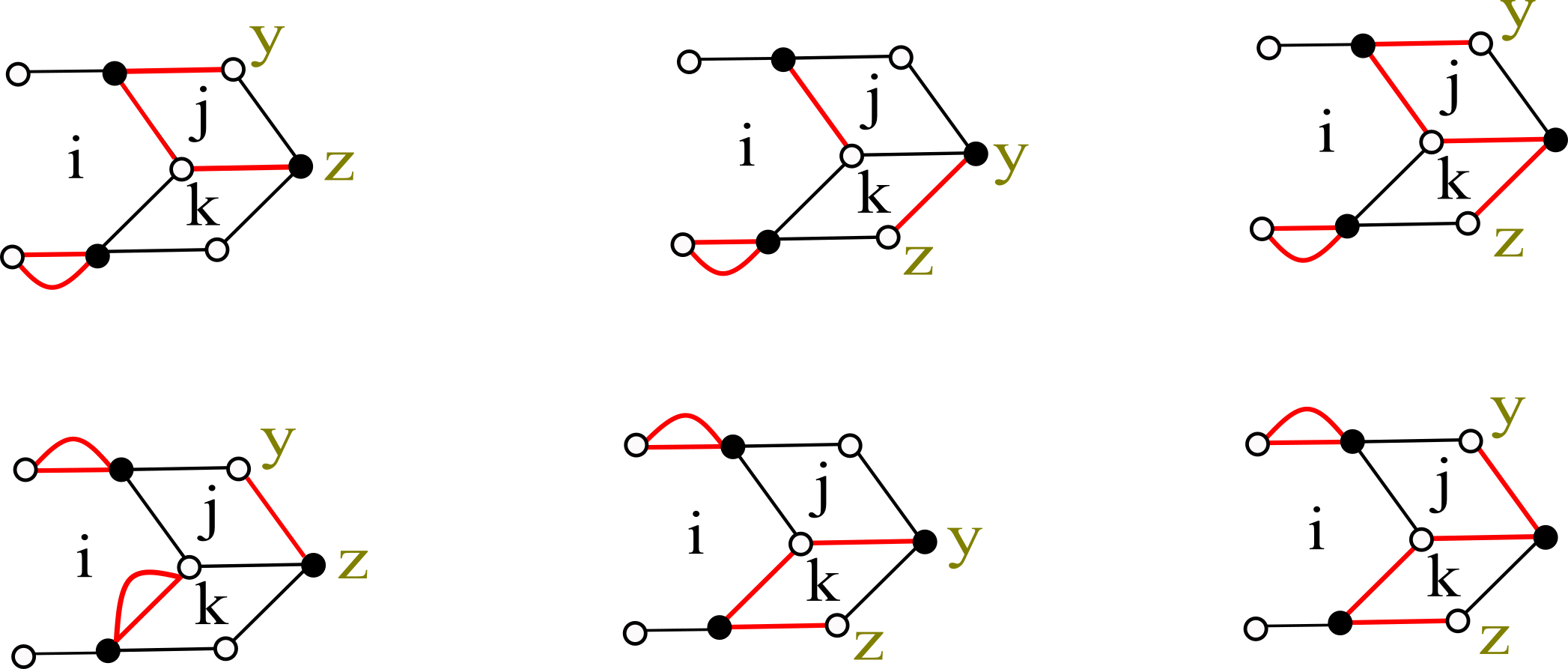}
    \caption{The \textcolor{olive}{green} nodes based on various tiles in or not in $Q_{\text{supp}}$. The leftmost column shows the case when $j \in Q_{\text{supp}}$ and $k \not\in Q_{\text{supp}}$; the center column shows the case when $j \not\in Q_{\text{supp}}$ and $k \in Q_{\text{supp}}$; and the rightmost column shows the case when $j,k \in Q_{\text{supp}}$.}
    \label{fig:greennodes}
\end{figure}

\begin{defn}\label{defn:nodesposet}
A mixed dimer configuration $M$ of $G$ is \textbf{node-monochromatic} if any path consisting of edges in $M$ between nodes connects nodes of the same color. If there exists a path consisting of edges in $M$ between nodes of different colors, we say $M$ is \textbf{node-polychromatic}.
We define $(P, \leq)$ to be the subposet of $(\bar{P}, \leq)$ consisting of mixed dimer configurations that satisfy the valence condition, are reachable via a sequence of allowable flips from $M_-$, and are node-monochromatic.
\end{defn}

\end{subsubsection}

\begin{prop}
\label{prop:minimalinP}
$M_-$ is in $P$, i.e. is a mixed dimer configuration on the base graph that satisfies the valence condition and is node-monochromatic.
\end{prop}

\begin{proof}
First note that $M_-$ trivially satisfies the condition that it is reachable by a sequence of allowable flips from $M_-$ by taking the empty sequence. Note that $M_-$ also satisfies the valence condition by construction. By definition of $M_-$, the sets $M_1, M_2$ will satisfy the valence condition and for any $k$-cycle fully supported in $\underline{d}$, the vertex corresponding to the $k$-star is matched by the edges in $Z$ corresponding to the unique singular arrow in that $k$-cycle. So, it suffices to show that $M_-$ is node-monochromatic by showing that there are no paths between nodes of different colors. If the associated quiver is acyclic, $M_-$ is node-monochromatic, see Proposition 3.4.1 of \cite{prequel}. So, we need to show $M_-$ is node-monochromatic if our induced subquiver with respect to $\underline{d}$ contains an oriented cycle. \allowbreak

\vspace{1em}

If the oriented cycle appears in the type $A_m$ part of our induced subquiver, then by definition of the green nodes in Figure \ref{fig:greennodes}, these green nodes must be connected in $M_-$. In particular, these sets of nodes cannot be connected to red or blue nodes as any such path would imply a vertex of degree 3 in $M_-$. If the oriented cycle appears in the type $D_n$ part of the subquiver, then this quiver must be type IV. In this case, the blue nodes are connected in $M_-$ by definition, which again implies that they cannot connect to the green nodes as any such path would imply a vertex of degree 3 in $M_-$. Therefore, $M_-$ must be node-monochromatic.
\end{proof}
\end{subsection}

In the acyclic case, refer to Example 3.4 and Figure 1 in \cite{prequel} to see an example of this poset $P$. We provide another example when $Q$ contains an oriented cycle before proceeding stating our main result in the following section. 

\begin{ex}\label{example:typeIVcyclic}
In Vatne type IV where the central cycle is a 3-cycle, consider the $D_6$ quiver shown in Figure \ref{fig:basegraphexample} with $\underline{d} = (1,1,2,1,1,1)$. The base graph and minimal matching with the nodes is given by:

\begin{center}
    \includegraphics[scale=.5]{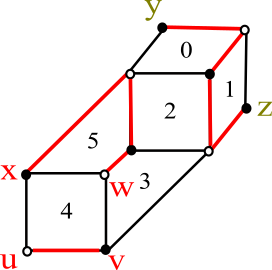}
\end{center}

\noindent and the poset $P$ of node monochromatic mixed dimer configurations is given in Figure \ref{fig:bigposetcyclic}.

\begin{figure}[H]
    \centering
    \includegraphics[scale=.335]{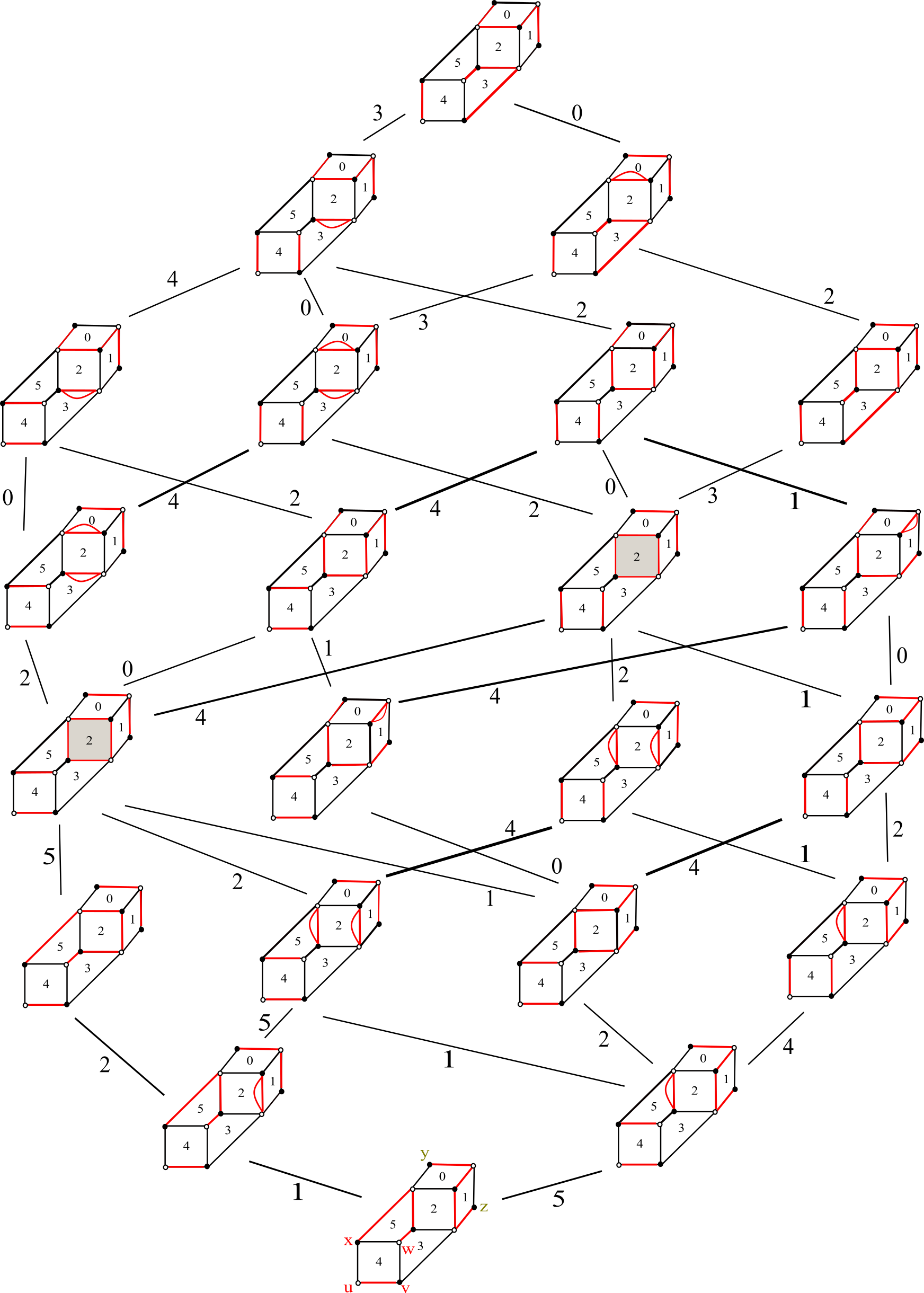}
    \caption{Poset of mixed dimer configurations in Example \ref{example:typeIVcyclic}. The tiles enclosed by cycles are shaded grey for emphasis.}
    \label{fig:bigposetcyclic}
\end{figure}

\end{ex}

Now, with these definitions and this running example, we have the ingredients to state our main result. Namely, we will show that the elements of this poset $P$ give the non-zero monomials in the $F$-polynomial associated to $\underline{d}$.
\end{section}

\begin{section}{Main Theorem} \label{section:maintheorem}
In this section, we connect the representation theory and combinatorics presented in the previous sections together in our main result. In particular, we give the generating function for cluster variables in terms of mixed dimer configurations indexed by certain dimension vectors of modules of the associated Jacobian algebra. 

\begin{thm}
\label{thm:main}
Given any quiver $Q$ of type $D_n$ and crossing vector $\underline{d}$, we let $F_{\underline{d}}$ denote the $F$-polynomial corresponding to the cluster variable with crossing vector $\underline{d}$.  This expression is based on the appropriate cluster algebra of type $D_n$ and assuming an initial seed defined by the choice of quiver $Q$ and the standard initial cluster of $\{x_1,x_2,\dots, x_n\}$. \allowbreak \vspace{1em}

Furthermore, let $M_- = M_-(Q,\underline{d})$ be the minimal matching, as defined in Section \ref{subsection: minimalmatching} and let $P$ be the poset of mixed dimer configurations that satisfy the valence condition, are reachable via a sequence of allowable flips from $M_-$, and satisfy the node monochromatic condition as defined in Definition \ref{defn:nodesposet}. Then the expansion of $F$-polynomial $F_{\underline{d}}$ can be expressed as a weighted multi-variate rank-generating function on the poset $(P, \leq)$ determined by $M_-$:

$$F_{\underline{d}} = \sum_{D \in P} 2^c u_{0}^{t_0} u_1^{t_1} \cdots u_{n-1}^{t_{n-1}},$$
where the sum is taken over mixed dimer configurations $D$ obtained by flipping tile $i$ $t_i$-times (keeping track of multiplicities) and $c$ is the number of cycles in $D$.
\end{thm}

\begin{ex}\label{ex:maintheoremexample}
Consider the quiver and $\underline{d}$-vector from Example \ref{example:typeIVcyclic}. Each element of the poset $P$ in Figure \ref{fig:bigposetcyclic} correspond to a monomial in the $F$-polynomial:
\begin{align*}
    F_{\underline{d}} &= 1+u_1+u_1u_5+u_4u_5+u_1u_2+u_2u_5+u_2u_4u_5+u_1u_4u_5+2u_1u_2u_5+u_0u_2u_5\\
    &+2u_1u_2u_4u_5+u_0u_2u_4u_5+u_1u_2^2u_5+u_0u_1u_2u_5+u_1u_2u_3u_4u_5+u_1u_2^2u_4u_5\\
    &+u_0u_1u_2u_4u_5+u_0u_1u_2^2u_5+u_1u_2^2u_3u_4u_5+u_0u_1u_2^2u_4u_5+u_0u_1u_2^2u_3u_4u_5.
\end{align*}

To highlight one of the monomials in the poset, notice the term $2u_1u_2u_5$ corresponds to traversing up the poset $P$ on the left by flipping tiles 1,2 and then 5. There is one cycle in the mixed dimer configuration enclosing the face 2 which gives the coefficient $2^1$.
\end{ex}

In order to prove Theorem \ref{thm:main}, we rely on representation theory. In the acyclic case \cite{prequel}, we proved this result by creating a bijection between mixed dimer configurations and vectors that parameterize subrepresentations of a fixed indecomposable quiver representation of dimension vector $\underline{d}$. In particular, we utilized a categorization of these vectors created by Tran \cite{tran}. The acyclic case is particularly nice because the indecomposable quiver representations are indexed by positive roots of the $D_n$ root system. As mentioned previously, when $Q$ contains an oriented cycle, positive roots no longer are in bijection with the cluster variables. We instead rely on the surface model with the crossing vectors cataloged in Appendix \ref{section:d-vectors}. \allowbreak

\vspace{1em}

To an arc $\gamma$ in a triangulated surface, one can associate an indecomposable Jacobian algebra module $M_\gamma$ whose dimension vector $\underline{d}_\gamma$ is given by enumerating the crossings of $\gamma$ with the arcs of the triangulation. With this, we mimic the work of Tran to parameterize $\underline{e}$ vectors that index sub-modules of $M_\gamma$ we call submodule-indexing vectors. We then prove Theorem \ref{thm:main} by creating a bijection between these submodule-indexing $\underline{e}$ vectors and mixed dimer configurations. To state this theorem, we define a few properties on arrows given by \cite{tran} in Definition 4.1. Define a partial order $\geq$ on $\zz^n$ via 

$$\underline{a} \geq \underline{a'} \text{ if } \underline{a} - \underline{a'} \in \zz_{\geq 0}^n.$$

\begin{defn}\label{defn:acceptablecritical}
Let $\underline{d} \in \{0,1,2\}^n$ and $\underline{e}$ such that $0 \leq e_i \leq d_i$ for each $1 \leq i \leq n$. An arrow $j \to k \in Q_1$ is \textbf{acceptable with respect to $(\underline{d}, \underline{e})$} if $e_j-e_k \leq [d_j-d_k]_+ := \max(d_j-d_k,0)$. An arrow is called \textbf{critical with respect to $(\underline{d}, \underline{e})$} if either 
$$(d_j,e_j) = (2,1) ~~\text{ and }~~(d_k,e_k) = (1,0)$$
or
$$(d_k,e_k) = (2,1) ~~\text{ and }~~(d_j,e_j) = (1,1)$$
Let $C = \{i \in Q_0~:~(d_i,e_i)=(2,1)\}$ and let $\nu(C)$ be the number of critical arrows that have a vertex in $C$.
\end{defn}

With these definitions, we are going to define a set of vectors $\underline{e}$ we call \textbf{submodule-indexing} that satisfy certain conditions. In \cite{tran}, analogous conditions gave criterion for indexing dimension vectors of subrepresentations of a given indecomposable acyclic quiver representation. Note that in the new cases of non-acyclic quiver, singular arrows can play a special role.

\begin{defn} \label{submod-indexing}
Let $\underline{d} \in \{0,1,2\}^n$. We say that a vector $\underline{e} \in \{0,1,2\}^n$ is \textbf{submodule-indexing with respect to $\underline{d}$} if \begin{enumerate}
    \item $0 \leq \underline{e} \leq \underline{d}$;
    \item Any arrow $j \to k \in Q_1$ that is not singular, is acceptable; and 
    \item $\nu(C) \leq 1$.
\end{enumerate} 
\end{defn}

\begin{rmk}\label{rmk:evectorlanguage}
We refer to condition 2 from Definition \ref{defn:acceptablecritical} as the acceptability condition and condition 3 as the criticality condition.
\end{rmk}

Now we are ready to state a bijection between mixed dimer configurations and these submodule-indexing vectors.

\begin{thm}\label{thm:bijection}
Let $T$ be an ideal triangulation of a once-punctured $n$-gon.  Let $(Q_T,W_T)$ be the associated quiver and the $J_T = \Bbbk Q_T/I$ be its Jacobian algebra. Let $M$ be an indecomposable $J_T$-module associated to an arc $\gamma$ with dimension vector $\underline{d}$. Then there exists a bijection between 
$$\{\text{mixed dimer configurations }D \text { in } P\} \stacklongleftrightarrow{\sim} 
\{\underline{e} \text{ submodule-indexing with respect to } \underline{d}\}$$
where $P$ is the poset defined in Definition \ref{defn:nodesposet} and is order-preserving where the set on the right has a natural partial ordering defined right above Definition \ref{defn:acceptablecritical}.
\end{thm}

We first work to prove Theorem \ref{thm:bijection} by demonstrating a map in both directions. The map taking mixed dimer configurations to submodule-indexing $\underline{e}$-vectors is given by taking the superimposition of the mixed dimer configuration with $M_-$ and deleting cycles created by this multigraph. The tiles that these cycles are formed on correspond to the entries in $\underline{e}$. The map taking a submodule-indexing vector $\underline{e}$ to a mixed dimer configuration is given by taking a sequence of ``weighted flips" from the minimal matching on tiles that are in $\supp(\underline{e})$. We first describe the ladder map by weighting the edges of the base graph $G$ via

$$w(e) := \#\{ \text{ edges on }e \text{ in } M_-\}$$
for all $e \in E(G)$.
     
\begin{defn}
\label{defn:weightedflip}
The weights of the edges are transformed via the following prescription after flipping a tile: 
\begin{center}
    \includegraphics[scale=.3]{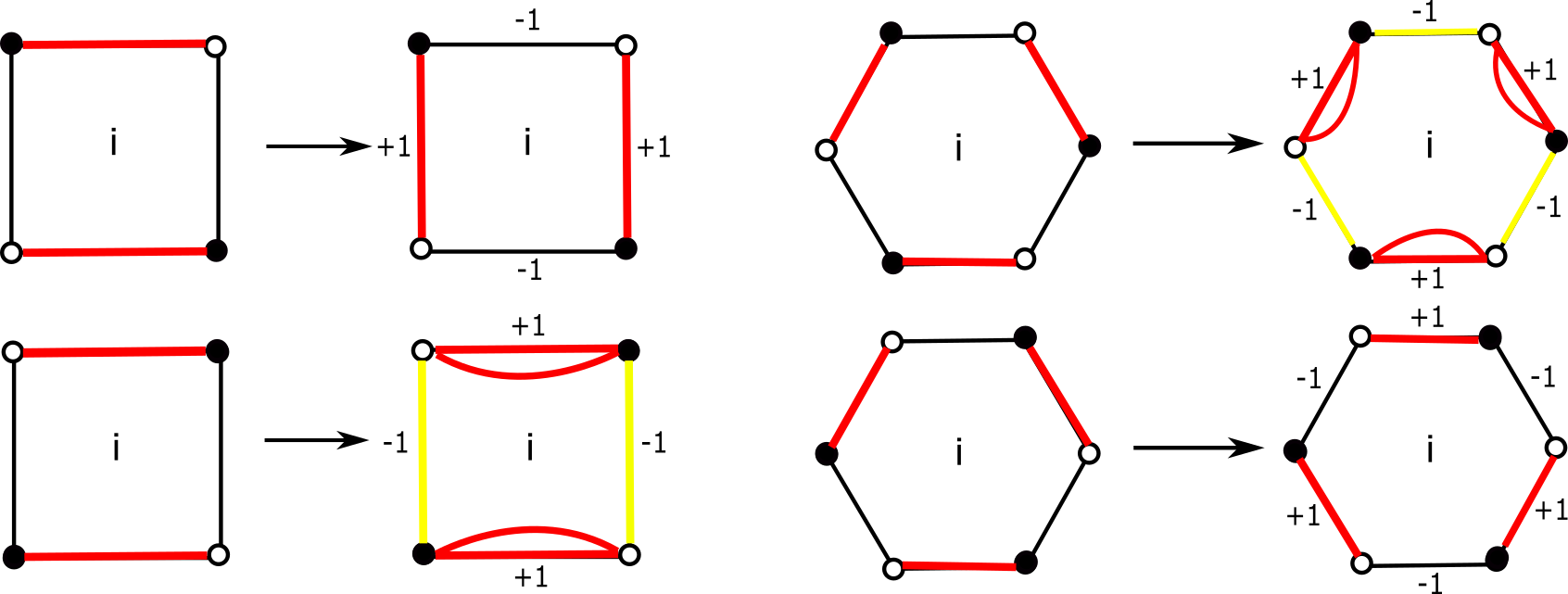}
\end{center}
This is known as a \textbf{weighted flip}.
\end{defn}
    
\begin{rmk}
Implicit in this definition is the fact that our weights will stay non-negative when we perform the flips given in Definition \ref{defn: posetrelation}. However, in the context of this direction of the bijection, we will allow for weighted flips at any tile; meaning that we now allow the weight of an edge to be negative. Following the conventions placed in \cite{prequel}, we emphasize the deficits which arise when we flip edges that are not in a given mixed dimer configuration, we distinguish edges of negative weight in yellow and call them ``antiedges." Allowing weights to be negative in intermediate steps of a sequence of flips simplifies the proof as we do not need to prove a sequence of allowable flips exists to obtain a mixed dimer configuration from a submodule-indexing vector $\underline{e}$. Namely, we just show that after flipping in any order, the mixed dimer configuration we obtain has all non-negative weighting and belongs to the poset $P$.
\end{rmk}
     
\begin{thm}
\label{thm:EtoD}
Let $\gamma$ be an arc superimposed on a triangulated once-punctured $n$-gon and let $Q$ be the quiver associated to this triangulation. Let $\underline{d}$ be the crossing vector of the arc $\gamma$. Let $G$ be the base graph constructed using the data of $Q$ and $\underline{d}$ as described in Definition \ref{defn:uncontractedbasegraph} and \ref{defn: basegraph}. Suppose that $\underline{e}$ is a submodule-indexing vector $\underline{e}$. Then there exists a unique way to produce a mixed dimer configuration $D$ in poset $P$ via the following procedure: 

\begin{enumerate}
\item Weight the edges of the base graph by $w(e)$ for $e \in M_-$. Take the positive entry $e_{i_1} > 0$ in $\underline{e}$ of minimal index, and flip tile $i_1$ $e_{i_1}$ number of times. Transform its edge weights as prescribed in Definition \ref{defn:weightedflip}. Let $D_1$ be the mixed dimer configuration obtained.
\item From $D_1$, take the next positive entry $e_{i_2} > 0$ in $\underline{e}$ of minimal index, and flip tile $i_2$ $e_{i_2}$ number of times. Again, transform its edge weights as prescribed in Definition \ref{defn:weightedflip} to arrive at the mixed dimer configuration $D_2$.
\item Iterate this process until we have exhausted all positive entries in $\underline{e}$. The resulting mixed dimer configuration $D$ will only have non-negative weights, i.e. no yellow edges will remain. Moreover, it will be an element of $P$. 
\end{enumerate}
\end{thm}

Before proving Theorem \ref{thm:EtoD}, let's first see an example of this algorithm.

\begin{ex}\label{ex:etoD}
Suppose $\underline{d} = (0,1,2,1,2,2,1,1,0)$ and the submodule-indexing vector $\underline{e} = (0,0,0,0,0,2,1,0,0)$. The minimal matching and steps of flipping tiles $5$ twice, then $6$ once is shown in Figure \ref{fig:etoD}.

\begin{figure}[H]
    \centering
    \includegraphics[width=\textwidth]{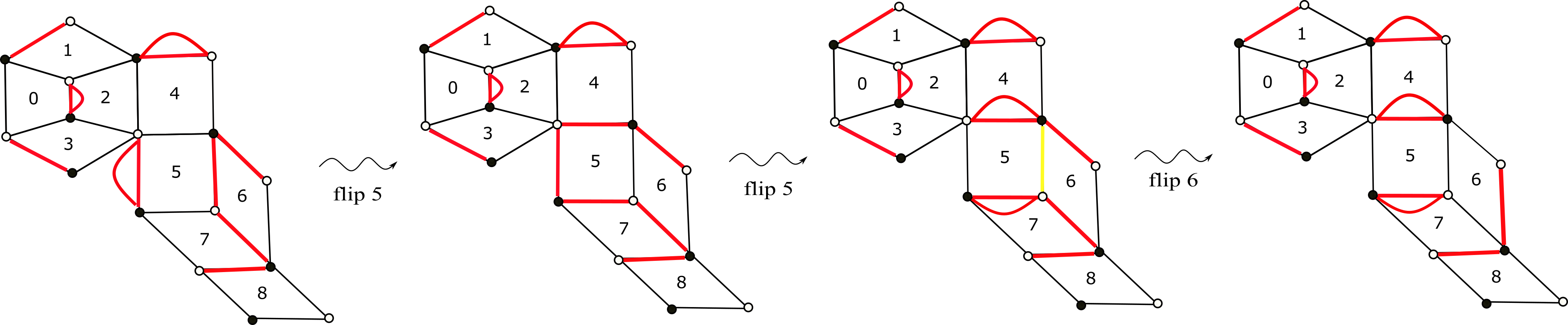}
    \caption{ }
    \label{fig:etoD}
\end{figure}

\end{ex}

In order to prove Theorem \ref{thm:EtoD} and eventually Theorem \ref{thm:bijection}, we need to create a dictionary between mixed dimer configurations and conditions on $\underline{e}$ vectors. We have the following sequence of lemmas that allows us to relate edges distinguished on a mixed dimer configuration with coordinates of a vector. These lemmas mimic arguments found in \cite{prequel} for the acyclic case, but we note the special features that arise when there are cycles in the quiver, and the special role of singular arrows in particular.

\begin{lemma} \label{lemma:MijNij}
Let $Q$ be a type $D_n$ quiver and let $\underline{d}$ be the dimension vector of an indecomposable Jacobian algebra module $M$. Let $\underline{e}$ such that $0 \leq \underline{e} \leq \underline{d}$ and let $D$ be the mixed dimer configuration obtained by flipping tile $k$ $e_k$ number of times from $M_-$. For any non-singular arrow $i \to j \in Q_1$, let $m_{i,j}$ be the number of edges distinguished in $D$ on the edge between tiles $i$ and $j$. Then
     
     $$m_{i,j} = \max(d_i-d_j,0) + (e_j-e_i) =: n_{i,j}.$$
\end{lemma}
     
\begin{proof}
We proceed by induction on $|\underline{e}| := \sum_{k=0}^{n-1} e_k$. When $|\underline{e}| = 0$, i.e. $\underline{e} = (0,0, \dots, 0)$, the associated mixed dimer configuration is $M_-$. By definition of $M_-$, we only distinguish internal edges if $G_1 \subsetneq G$ or $G_2 \subsetneq G$. If $d_i>0$ and $d_j=0$, then $i \in G_1$, but $j \notin G_1$. Since $i \to j$ is oriented black to white with respect to $i$, then $m_{i,j} = d_i = \max(d_i-0,0)$. If $j \to i$, then the edge straddling $i$ and $j$ is oriented white to black clockwise with respect to tile $i$. By definition of $M_-$, we will not distinguish that edge giving that $n_{j,i} = \max(d_j-d_i,0) = 0 = m_{j,i}$.\allowbreak \vspace{1em}

Suppose up to $k \in \nn$, our formula holds. Now, suppose $|\underline{e}| = k+1$, and choose $\tilde{\underline{e}}$ with $|\tilde{\underline{e}}|=k$, such that $\underline{e}$ can be obtained by adding $1$ to the $i^{\text{th}}$ entry of $\tilde{\underline{e}}$. In other words, the tile $i$ has been flipped from the associated mixed dimer configuration associated to $\tilde{\underline{e}}$ to obtain the mixed dimer configuration associated to $\underline{e}$.\allowbreak \vspace{1em}

If $i \to j$, then the edge straddling tiles $i$ and $j$ is oriented black to white clockwise on tile $i$. As tile $i$ has been flipped, $e_i = \tilde{e_i}-1$ and the number of edges distinguishing on the edges straddling $i$ and $j$ decreases by 1 from the dimer configuration associated to $\tilde{\underline{e}}$. Therefore,
     
$$n_{i,j} = \max(d_i-d_j,0) + (e_j-(e_i+1)) = \max(d_i-d_j,0) + (e_j-e_i) - 1.$$
     
Therefore, performing a flip at tile $i$ decreased both $n_{i,j}$ and $m_{i,j}$ by 1. If $j \to i$ in $Q$, then the edge straddling tiles $i$ and $j$ is oriented white to black clockwise on tile $i$. As tile $i$ has been flipped, $e_i = \tilde{e_i}+1$ and the number of edges distinguishing on the edges straddling $i$ and $j$ increases by 1 from the dimer configuration associated to $\tilde{\underline{e}}$. Therefore,

$$n_{j,i} = \max(d_j-d_i,0) + (e_i+1-e_j) = \max(d_i-d_j,0) + (e_i-e_j) + 1.$$
     
Therefore, performing a flip at tile $i$ increased both $n_{j,i}$ and $m_{j,i}$ by 1.
\end{proof}

\begin{lemma}\label{lemma:mijnijsingular} (Analogous to Lemma \ref{lemma:MijNij} for singular arrows). Following the notation of Lemma \ref{lemma:MijNij}, we can make an analogous statement for singular arrows $i \to j \in Q_1$. Namely, 

$$m_{i,j} - 1 = \max(d_i-d_j,0) + (e_j-e_i) =: n_{i,j}.$$
\end{lemma}

\begin{proof}
When an arrow $i \to j$ in singular, we add an extra distinguished edge straddling tiles $i$ and $j$ in $M_-$ that is not a boundary edge going black to white clockwise in $G_1$ or $G_2$. Hence, the formula from Lemma \ref{lemma:MijNij} can be modified by subtracting 1 to reflect this extra edge. 
\end{proof}

\begin{cor}
\label{cor:accnonneg}
The values $n_{i,j} := \max(d_i-d_j,0) + (e_j-e_i)$ all satisfy $n_{i,j} \geq 0$ if and only if all weights on interior edges on the mixed dimer configuration $D$ associated to $e$ are nonnegative.  In particular, all arrows $i \to j$ are acceptable with respect to $(\underline{d}, \underline{e})$ if and only if all weights on interior edges are nonnegative on the mixed dimer configuration $D$ associated to $e$.
\end{cor}

\begin{lemma}
\label{lemma:MinfinityN}
(Analogous formula to Lemma \ref{lemma:MijNij} with boundary edges). Let $Q$ be a type $D_n$ quiver and let $\underline{d}$ be the dimension vector of an indecomposable Jacobian algebra module $M$. Let $\underline{e}$ such that $0 \leq \underline{e} \leq \underline{d}$ and let $D$ be the mixed dimer configuration obtained by flipping tile $k$ $e_k$ number of times from $M_-$. Let the outer face of our graph $G$ be indexed by $\infty$. We assign an arrow to each of the boundary edges of $G$ with the convention that we ``see white on the right." To each of these boundary edges, assign a weight to the edge $\alpha$ on tile $i$ as follows:
     
     $$n_{i,\infty}(\alpha) = \max(d_i,0) - e_i ~~~\text{if }i \to \infty \text{ about } \alpha$$
     $$n_{\infty,i}(\alpha) = \max(-d_i,0) + e_i ~~~\text{if }\infty \to i \text{ about } \alpha$$
     
Let $m_{i,\infty}(\alpha)$ (respectively $m_{\infty,i}(\alpha)$) be the number of edges distinguished on $\alpha$ on tile $i$ in $D$ where $i \to \infty$ about $\alpha$ (respectively $\infty \to i$ about $\alpha$.) Then for any boundary edge $\alpha$ on tile $i$,
     
     $$n_{i,\infty}(\alpha) = m_{i,\infty}(\alpha) ~~\text{ and }~~ n_{\infty, i}(\alpha) = m_{\infty,i}(\alpha).$$
\end{lemma}

\begin{proof}
As in the proof of Lemma \ref{lemma:MijNij}, we proceed by induction on $|e|$. Suppose $|e|=0$, then the associated mixed dimer configuration is $M_-$. Let $\alpha$ be a boundary edge on a tile $i$. If $i \to \infty$ across the edge $\alpha$, then $m_{i,\infty}(\alpha) = \max(d_i,0) = d_i.$ Moreover, the edge $\alpha$ is oriented black to white clockwise giving that it is distinguished $d_i$-times in $M_-$. Therefore, $m_{i,\infty}(\alpha) = d_i = n_{i,\infty}(\alpha)$. If $\infty \to i$ across the edge $\alpha$, we have $m_{\infty,i}(\alpha) = \max(-d_i,0) = 0.$ Moreover, the edge $\alpha$ is oriented white to black clockwise giving that it is never distinguished in $M_-$. Therefore, $m_{i,\infty}(\alpha) = 0 = d_i = n_{i,\infty}(\alpha)$.\allowbreak \vspace{1em}

Suppose up to $k \in \nn$, our formula holds. Now, suppose $|\underline{e}| = k+1$, and choose $\tilde{\underline{e}}$ with $|\tilde{\underline{e}}|=k$, such that $\underline{e}$ can be obtained by adding $1$ to the $i^{\text{th}}$ entry of $\tilde{\underline{e}}$. In other words, the tile $i$ has been flipped from the associated mixed dimer configuration associated to $\tilde{\underline{e}}$ to obtain the mixed dimer configuration associated to $\underline{e}$. Suppose that $\alpha, \beta$ are boundary edges on tile $i$ where $\infty \to i$ across the edge $\alpha$ and $i \to \infty$ across the edge $\beta$. After flipping tile $i$, the $w(\beta)$ decreases by 1 and $w(\alpha)$ increases by 1. Since $\tilde{e_i} = e_i+1$, $m_{i,\infty}(\beta)$ is transformed by
$$m_{i,\infty}(\beta) = \max(d_i,0) - (e_i+1) = \max(d_i,0) -e_i-1,$$
and $m_{\infty,i}(\alpha)$ is transformed by
$$m_{\infty,i}(\alpha) = \max(-d_i,0) + e_i +1.$$

Hence, we have that $n_{i,\infty}, m_{i,\infty}$ and respectively $n_{\infty,i}, m_{\infty,i}$ are transformed in the same way after a flip at tile $i$.
\end{proof}

We need one more technical lemma regarding the graphical structure of a quiver $Q$ mutation-equivalent to a type $D_n$ Dynkin diagram and the possible subquivers arising as the support of a submodule-indexing vector $\underline{e}$.

\begin{lemma} \label{lem:special_j}
Suppose $Q$ is a quiver mutation-equivalent to an orientations of a type $D_n$ Dynkin diagram and $\underline{e}$ is a choice of submodule-indexing vector with respect to $\underline{d}$ for some $d$-vectors $\underline{d}$ that is not fully supported by only 1's on any cycle of $Q$.  Then $Q$ is guaranteed to contain a vertex $j$ with the property that $e_j > 0$ and such that for any arrow pointing $i \to j$, we have $e_j > e_i$.
\end{lemma}

\begin{proof}
If the given submodule-indexing vector $\underline{e}$ contains a $2$, then let $Q^{e}$ denote the subquiver of $Q$ containing exclusively the vertices $i$ such that $e_i = 2$.  By Lemma \ref{lemma:d_vector_2}, the subquiver $Q^{e}$ must be a tree and therefore must contain a source.  Letting $j$ denote the vertex of such a source, we see that vertex $j$ vacuously satisfies $e_j > e_i$ for any arrow pointing $i \to j$.

On the other hand, if the given submodule-indexing vector $\underline{e}$ contains no $2$'s, then we let $Q^{e}$ denote the subquiver of $Q$ given by the support of $\underline{e}$, i.e. containing exclusively the vertices $i$ such that $e_i = 1$.  By assumption, the subquiver $Q^{e}$ again must be a tree and therefore must contain a source, and we can choose vertex $j$ just as above. 
\end{proof}

With these lemmas, we are now ready to prove Theorem \ref{thm:EtoD}.

\begin{proof}
Suppose $\underline{e}$ is submodule-indexing and let $D$ be the mixed dimer configuration produced as prescribed in the statement of Theorem \ref{thm:EtoD}. In order to show that $D \in P$, we must show that $D$ is a mixed dimer configuration reachable via a sequence of flips from $M_-$ and is node-monochromatic. We first show that $D$ is reachable via a sequence of flips from $M_-$ by induction on $|e|$. For the base case, we agree to associate the submodule-indexing vector $\underline{e} = (0,0,\dots,0)$ to the minimal matching $M_-$ which is reachable by an empty sequence of flips from $M_-$. \allowbreak  \vspace{1em}

Suppose that for any nonzero submodule-indexing vector $\underline{e}$ with $|e|=k$, the associated mixed dimer configuration $D$ is reachable via a sequence of flips from $M_-$. Now suppose that $|e| = k+1$ and that we have added 1 to the $i^{\text{th}}$ entry of $\underline{e}$. As $\underline{e}$ is submodule-indexing, any arrow involving vertex $i$ must be acceptable with respect to $(\underline{d}, \underline{e})$ i.e. $e_i - e_j \leq \max(d_i-d_j,0)$. By Lemma \ref{lemma:MijNij}, we have that $n_{i,j} = m_{i,j}$ which implies that the number of edges straddling tiles $i$ and $j$ must be non-negative. Moreover, by Lemma \ref{lemma:MinfinityN}, since $0 \leq \underline{e} \leq \underline{d}$, we have that for any boundary edge $\alpha$ on tile $i$, $n_{i,\infty}(\alpha) = \max(d_i,0) - e_i = d_i - e_i \geq 0$ giving that $m_{i,\infty}(\alpha)$, the number of edges on $\alpha$ in $D$ is non-negative. Similarly, $n_{\infty,i}(\alpha) = \max(-d_i,0) + e_i = 0+e_i \geq 0$ giving that $m_{\infty,i}(\alpha)$, the number of edges on $\alpha$ in $D$ is non-negative. If $i \to j$ is a singular arrow, the number of edges straddling $i$ and $j$ will have only increased by 1, so $m_{i,j} \geq 0$ as well. Hence, the resulting mixed dimer configuration $D$ has edges with all non-negative weights. 
\allowbreak \vspace{1em}

We now show that there exists some order in which we can read the entries of $\underline{e}$ which will yield a mixed dimer configuration with non-negative weights at each flip along the way in the sequence.  Even though we constructed $\underline{e}$ inductively in a way such that the last entry increased by one is the $i^{\text{th}}$ entry, it is not necessarily the case that the corresponding mixed dimer configuration $D$ is reachable by a flip sequence ending in a flip of tile $i$. \allowbreak \vspace{1em}

Instead, here we invoke Lemma \ref{lem:special_j} with $j$ denoting the label of the vertex whose existence is posited by the lemma. If $\underline{d}$ is fully supported by 1's on a cycle of $Q$, then there exists some singular arrow $a \to b$ with $d_a=d_b=1$. In this case, let $j$ be the vertex $b$. We let $\underline{e'}$ be the result of subtracting the $j^{\text{th}}$ unit vector from $\underline{e}$. Then, for any non-singular arrow $i \to j$, we have $n_{i,j}' = \max(d_i-d_j,0) + (e_j-1-e_i)$ must be nonnegative. If the arrow is the reverse orientation $j \to i$, then $n_{j,i}' = \max(d_j-d_i,0) + (e_i - e_j+1)$ is also nonnegative. Moreover, on the boundary, we also have $n_{i, \infty}' = d_i - e_i$ and $n_{\infty, i}' = e_i$ are nonnegative. In the case of a fully supported cycle with only 1's and singular arrow $a \to j$, by Lemma \ref{lemma:mijnijsingular}, we have that $m_{a,j} - 1 = \max(d_a-d_j,0) + e_j-e_a$. Since $e_a \leq 1$, $e_j \leq 1$, their difference is at most -1 implying that $n_{a,j} = e_j-e_a+1$ is nonnegative. \allowbreak  \vspace{1em}

Therefore $\underline{e'}$ corresponds to a mixed dimer configuration $D'$.  Since $|e'|=k$, by the inductive hypothesis, $D'$ is reachable by a sequence of allowable flips from $M_-$. This gives that $D$ itself is reachable from $M_-$ by a sequence of allowable flips, where we tack on a flip of tile $j$.  This last flip is allowable since both $D'$ and $D$ are mixed dimer configurations, i.e. with nonnegative weights on edges, and the difference $\underline{e} - \underline{e'}$ is the $j^{\text{th}}$ unit vector.\allowbreak  \vspace{1em}

We now show $D$ is a node monochromatic mixed dimer configuration. We again proceed by induction on $|e|$. By Proposition \ref{prop:minimalinP}, we have that $M_-$ is node-monochromatic, establishing the base case. So suppose that up to $k \in \nn$, we have that when $|e| = k$, the resulting mixed dimer configuration $D$ is node-monochromatic. Now take $\underline{e}$ with $|e| = k+1$ where we have added 1 to the $i^{\text{th}}$ entry of $\underline{e}$. We may assume that $D$ was obtained via a sequence of $k+1$ allowable flips from $M_-$. We claim that the only way we could have produced a path connecting nodes of different colors is if there is more than one critical arrow in $Q$ with respect to $(\underline{d}, \underline{e})$.\allowbreak \vspace{1em}

If $Q^{\text{supp}}$ is acyclic, then we employ the proof of Theorem 4.2.2 from \cite{prequel}. Moreover, in type IV where we have full support on the central $m$-cycle, we may also employ the argument in the acyclic case when the type $A_m$ part of the quiver has no oriented cycles. This is because such a path connecting \textcolor{blue}{blue} to \textcolor{olive}{green} nodes can only occur on an acyclic subquiver of $Q^{\text{supp}}$. Namely, a path only occurs in a mixed dimer configuration when there are valence 2 tiles and must end after a consecutive tile has valence 1. In any $\underline{d}$-vector in type IV with a 2, the central $k$-cycle is fully supported where the $\underline{d}$-vector entries are all 1. Hence, a path connecting the type $D_n$ part of the quiver to the type $A_m$ part of the quiver will only use the tile with the \textcolor{blue}{blue} nodes and the consecutive tile whose $\underline{d}$-vector entry is 2 and no other tiles on the $m$-cycle - reducing to the acyclic case. \allowbreak
\vspace{1em}

So it suffices to show that $D$ is node-monochromatic when $Q^{\text{supp}}$ contains an oriented cycle on the type $A_m$ part of the quiver. First note that in types I, II and III, the $\textcolor{red}{red}$ or $\textcolor{blue}{blue}$ nodes cannot be connected without creating two critical arrows. Let $p,q,r$ be the forking vertices in the type $D_n$ part of $Q^{\text{supp}}$. Then, for any orientation of the fork, paths between $\textcolor{red}{red}$ and $\textcolor{blue}{blue}$ nodes create two critical arrows as shown in Figure \ref{fig:nopathsbwredandblue}.\allowbreak
\vspace{1em}

\begin{figure}
    \centering
    \includegraphics[scale=.22]{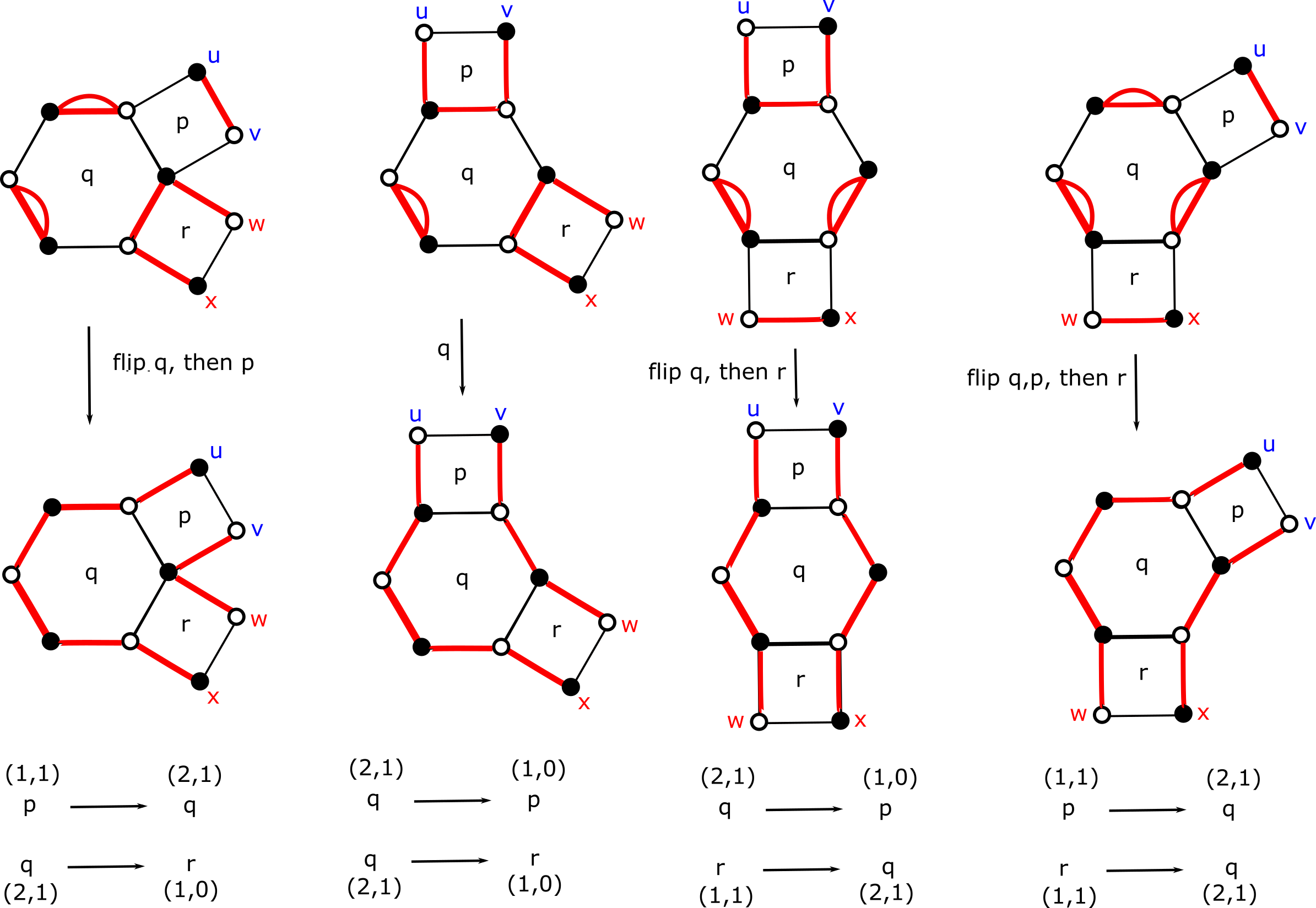}
    \caption{No paths between $\textcolor{red}{red}$ and $\textcolor{blue}{blue}$ nodes can occur when $\underline{e}$ satisfies the criticality condition.}
    \label{fig:nopathsbwredandblue}
\end{figure}

Suppose that the cycle is on the type $A_m$ part of the quiver and we aim to show that the $\textcolor{olive}{green}$ nodes cannot attach to either the $\textcolor{red}{red}$ or $\textcolor{blue}{blue}$ nodes without having more than one critical arrow. Note that it suffices to show that there is no connection between the $\textcolor{blue}{blue}$ nodes and $\textcolor{olive}{green}$ nodes as in types I, II and III, the $\textcolor{blue}{blue}$ nodes and $\textcolor{red}{red}$ nodes are symmetric and type IV has no \textcolor{red}{red} nodes. Suppose that tile $a$ has the two $\textcolor{blue}{blue}$ nodes $\textcolor{blue}{w,x}$ and let $b$ be the vertex connecting the type $D_n$ part of the quiver to the type $A_m$ part of the quiver. Suppose that the 3-cycle that has full support occurs at the set vertices $i,j,k$ read in cyclic order in the type $A_m$ part of the quiver. In order to connect the $\textcolor{olive}{green}$ nodes to the $\textcolor{blue}{blue}$ nodes, we must have flipped tiles connecting the 3-cycle $i,j,k$ with the vertices in the type $A_m$ part of the quiver connecting to tile $b$, then tile $b$. For example, if $b$ is directly attached to tile $i$, then we first flip tile $i$, then $b$ as shown in Figure \ref{fig:mubmui}:

\begin{figure}[h]
    \centering
    \includegraphics[scale=.3]{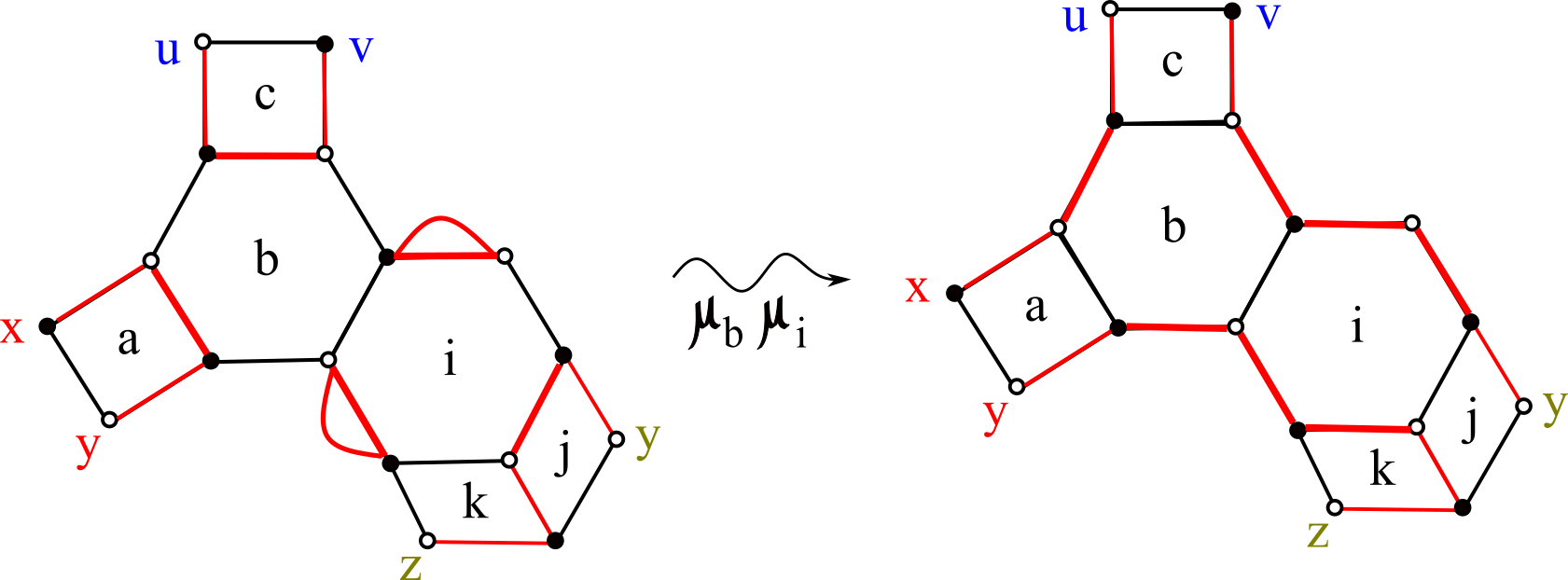}
    \caption{ }
    \label{fig:mubmui}
\end{figure}

With this, we see that this implied that the quiver must have had more than one critical arrow with respect to $(\underline{d}, \underline{e})$ as the arrow $b \to a$ has $(d_b,e_b) = (2,1)$ and $(d_a,e_a) = (1,0)$ and the arrow $i \to j$ has $(d_i,e_i) = (2,1)$ and $(d_j,e_j) = (1,0)$. Similarly, if $a \to b$, we would flip tiles $b$, then $a$ giving that the arrow $a \to b$ is critical with $(d_b,e_b) = (2,1)$ and $(d_a,e_a) = (1,1)$ with $i \to j$ has $(d_i,e_i) = (2,1)$ and $(d_j,e_j) = (1,0)$ still critical. The same argument holds if the 3-cycle $i,j,k$ was not directly adjacent to $b$ using a flip sequence along the vertices connecting $b$ to the 3-cycle $i,j,k$. \allowbreak \vspace{1em}

Therefore, even when there is a cycle in $Q^{\supp}$, we see that node monochromatic paths imply that the quiver must have had more than one critical arrow with respect to $(\underline{d}, \underline{e})$. Hence, the dimer $D$ associated to a submodule-indexing vector must be node monochromatic implying that $D \in P$ as desired.
\end{proof}

To complete the proof of Theorem \ref{thm:bijection}, we now describe the other side of the bijection i.e. the map taking mixed dimer configurations to submodule-indexing $\underline{e}$-vectors.

\begin{thm}
\label{thm:Dtoe}
Let $\gamma$ be an arc superimposed on a triangulated once-punctured $n$-gon and let $Q$ be the quiver associated to this triangulation. Let $\underline{d}$ be the crossing vector of the arc $\gamma$. Let $G$ be the base graph constructed using the data of $Q$ and $\underline{d}$ as described in Definition \ref{defn: basegraph}. Suppose $D$ is a mixed dimer configuration in $P$ i.e. is node monochromatic. There exists a unique way to produce a submodule-indexing vector $\underline{e}$. The process is given by the following procedure: 
     
     \begin{enumerate}
         \item Superimpose $D$ with $M_-$ on the base graph $G$ to obtain the multigraph $D \sqcup M_-$. In this superimposition, if $D$ and $M_-$ have any edge $e \in D \cap M_-$ in common on $G$, delete one copy of $e$ from $D \sqcup M_-$. Call the resulting multigraph $M_1$.
         \item Using the edges $M_1$, create a cycle of maximal length $\ell_1 > 2$, call it $C_1$. For all faces $i$ enclosed by $C_1$, add $+1$ to $v_i$ in $\underline{v}$ and delete $C_1$ from $M_1$. Call the resulting multigraph $M_2 = M_1 \setminus C_1$.
         \item Examine $M_2$ and if there are any cycles of length $\ell_2 > 2$, find a cycle of maximal length and call it $C_2$. For all faces $i$ enclosed by $C_2$, add +1 to $v_i$ and delete $C_2$ from $M_2$.
         \item Iterate this process of deleting cycles of of largest length and adding 1's to the vector $\underline{v}$ until nothing is left besides 2-cycles. The resultant vector $\underline{v}$ corresponds to $\underline{e}$.
     \end{enumerate}
\end{thm}

Before proving Theorem \ref{thm:Dtoe}, we compute an example of this side of the bijection.

\begin{ex}\label{ex:Dtoe}
In this example, let $\underline{d} = (0,1,2,1,2,2,1,1,0)$ and consider the base graph shown in Figure \ref{fig:Dtoe}. We consider a mixed dimer configuration $D$ and we show that its associated $\underline{e}$-vector is given by $\underline{e} = (0,0,0,0,0,2,1,0,0)$. When we superimpose $D$ with $M_-$ and delete any paired sets of edges in their intersection, we obtain $M_1$. In this case, $C_1$ is the cycle around the tiles labeled 5 and 6. This indicates that $\underline{v} = (0,0,0,0,0,1,1,0,0)$ at this step. After we delete the edges used to make this cycle, we obtain $M_1 \setminus C_1$ and see a cycle around the tile 5. This gives that $\underline{v} = (0,0,0,0,0,2,1,0,0)$ at this step and we see are left with an empty disjoint union of two cycles after this deletion. Hence, $\underline{v} = (0,0,0,0,0,2,1,0,0) = \underline{e}$ is the submodule-indexing vector associated to $D$ which is illustrated in Figure \ref{fig:Dtoesteps}.
\end{ex}

\begin{figure}[H]
    \centering
     \includegraphics[scale=.22]{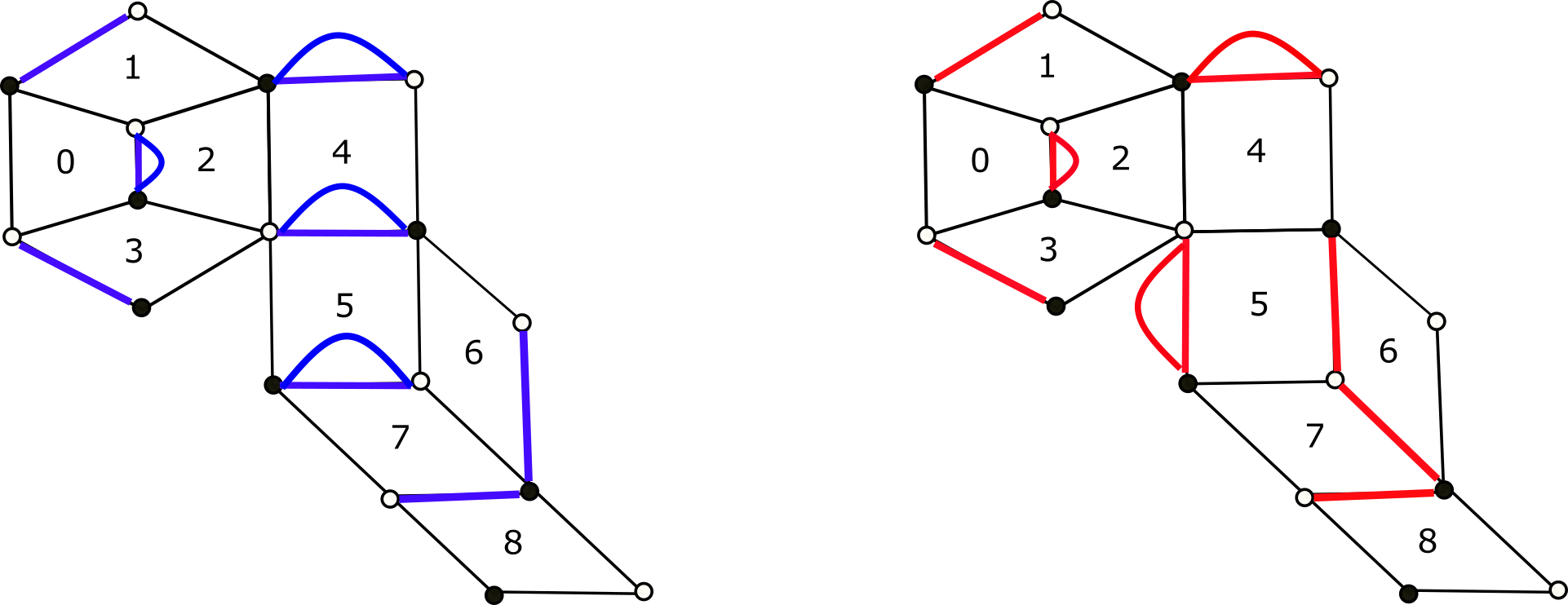}
    \caption{On the left, we have the mixed dimer configuration $D$ in blue and on the right, we have the minimal matching $M_-$ in red.}
    \label{fig:Dtoe}
\end{figure}
\begin{figure}[H]
    \centering
    \includegraphics[scale=.22]{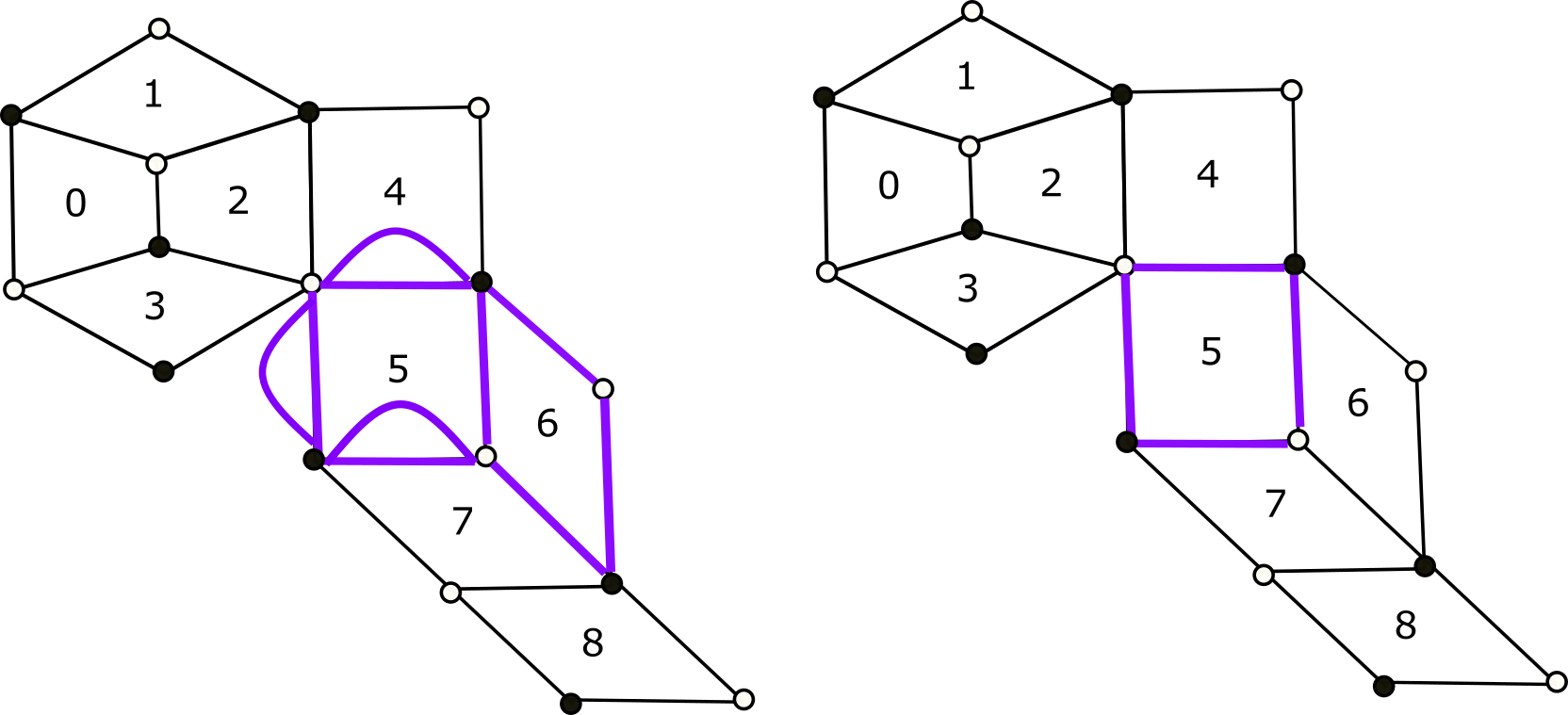}
    \caption{On the left, we have the superimposition of $D$ with $M_-$ after deleting any pairs of edges in the intersection, $M_1$. After deleting the cycle around tiles 5 and 6, we obtain $M_2$ illustrated on the right.}
    \label{fig:Dtoesteps}
\end{figure}

Now we prove Theorem \ref{thm:Dtoe}.

\begin{proof}
Suppose $D$ is a non-minimal mixed dimer configuration in the poset $P$ and initialize $\underline{v} = (0,0,\dots,0)$. To begin the algorithm, we must find a cycle of maximal length strictly larger than 2. In the superimposition of $D \sqcup M_-$ we call $M_1$, we obtain that each vertex that was previously valence 2 is now valence 4 in $M_1$ and each vertex that was valence 1 is now valence 2 in $M_1$. Because $D \neq M_-$, there must exist some cycle of length at least 4 in their superimposition as $D$ differs from the minimal matching by at least one flip. Hence, there exists some cycle of maximal length $\ell_1 \geq 4$.\allowbreak

\vspace{1em}

We claim that after each deletion of a maximal length cycle, the vector $\underline{v}$ is a submodule-indexing vector with respect to $\underline{d}$. To do this, we proceed by induction on $n$, the number of deletions of cycles required so that $D \sqcup M_-$ is a (possible empty) disjoint union of 2-cycles. \allowbreak  \vspace{1em}

When $n=0$, we agree to associate $M_-$ to $\underline{e} = (0,0,\dots,0)$ which is submodule-indexing. Suppose that $D'$ is a mixed dimer configuration that requires $i-1$ iterations of our algorithm to obtain the associated resultant vector $\underline{v}'$. Suppose further that this vector $\underline{v}'$ is submodule-indexing. Now suppose that $D$ is a mixed dimer configuration that requires one more application of our algorithm than $D'$ i.e. $D$ takes $i$ iterations and the resulting vector is called $\underline{v}$. Namely, we have $\underline{v}$ is obtained from adding 1's to $\underline{v'}$ to the tiles enclosed by cycle $C_i$. After the $i^{\text{th}}$ iteration of the algorithm, we have $0 \leq v_k \leq d_k$ for every $k$, is satisfied as the only way a cycle can enclose tile $j$ in the superimposition of $D$ and $M_-$ is if $D$ differs from $M_-$ at a tile $j$. This precisely happens when a flip occurred at tile $j$ and implicit in an occurrence of a flip is that $d_j>0$. In $\underline{v}-\underline{v'}$, the only new 1's occur if $C_i$ enclosed those corresponding tiles, $v_{j} \leq d_j$ for all $j$ enclosed by $C_i$. This gives that $0 \leq \underline{v} \leq \underline{d}$.\allowbreak  \vspace{1em}

To show that $\underline{v}$ is acceptable and satisfies the boundedness condition on the number of critical arrows, we need to induct on the number of tiles enclosed by $C_i$. We first show that $\underline{v}$ is acceptable. If $C_i$ enclosed a single tile, call it $j$, then to verify that the acceptability condition is satisfied, it suffices to check that any non-singular arrows involving vertex $j$. Note that $v_{j} = 1$ or $2$ because $j$ could have been enclosed in a previous cycle in an earlier iteration. Suppose that $j$ is the tail of an arrow i.e. we have some arrow $k \to j \in Q_1$. If $v_k \leq 1$, then acceptability is satisfied as $v_k-v_j \leq 0$. If $v_k = 2$, then $d_k=2$. Moreover, acceptability is satisfied as long as $v_j \neq 1$ and $d_j = 2$. However, by Lemma \ref{lemma:MijNij}, this implies that $m_{k,j} < 0$ contradicting Corollary \ref{cor:accnonneg}. Now suppose $j$ is the head of an arrow i.e. there is an arrow $j \to k \in Q_1$. Then if $v_k=2$, then acceptability is satisfied as $v_j-v_k \leq 0$. If $0 \leq v_k \leq 1$, then the only case that $j \to k$ fails acceptability is when $v_k=v_j-1$ and $d_j=d_k$. However, by Lemma \ref{lemma:MijNij}, this implies that $m_{k,j} < 0$ contradicting Corollary \ref{cor:accnonneg}. \allowbreak \vspace{1em}

Now, assume that if $C_i$ enclosed $k$ cycles, then the resulting $\underline{v}$ is submodule-indexing. Now, suppose that $C_i$ encloses $k+1$ tiles, $t_1, \dots, t_{k+1}$. We aim to show that $\underline{v} + \underline{e_{t_{k+1}}} = \underline{y}$ is submodule-indexing. It suffices to show that any non-singular arrow with vertex $t_{k+1}$ satisfies the acceptability condition. Since, $t_{k+1}$ is enclosed by $C_i$ and may have been enclosed by another cycle in a previous iteration, we have that $y_{t_{k+1}} \geq 1$. Suppose that $t_{k+1}$ is the tail of an arrow i.e. we have some arrow $j \to t_{k+1} \in Q_1$.\allowbreak
\vspace{1em}

If $y_j \leq 1$, then acceptability is satisfied as $y_j-y_{t_{k+1}} \leq 0$. If $y_j = 2$, then $d_k=2$. Moreover, acceptability is satisfied as long as $y_k \neq 1$ and $d_k = 2$. However, by Lemma \ref{lemma:MijNij}, this implies that $m_{j,t_{k+1}} < 0$ contradicting Corollary \ref{cor:accnonneg}. Now suppose $t_{k+1}$ is the head of an arrow i.e. there is an arrow $t_{k+1} \to j \in Q_1$. Then if $y_j=2$, then acceptability is satisfied as $y_{t_{k+1}}-y_j \leq 0$. If $0 \leq y_j \leq 1$, then the only case that $t_{k+1} \to j$ fails acceptability is when $y_j=y_{t_{k+1}}-1$ and $d_j=d_k$. However, by Lemma \ref{lemma:MijNij}, this implies that $m_{k,j} < 0$ contradicting Corollary \ref{cor:accnonneg}. \allowbreak \vspace{1em}

Therefore, we have shown that if $C_i$ encloses $k+1$ cycles, the resulting vector $\underline{y}$ is submodule-indexing. By induction, we have that after the $i^{\text{th}}$ step of our algorithm, $\underline{v}$ is submodule-indexing. Therefore, any vector obtained via this algorithm must satisfy the acceptability condition. Hence, the resultant $\underline{e}$ must be acceptable. \allowbreak  \vspace{1em}

Now, we show that $\underline{v}$ satisfies the criticality condition. In order to do this, we again induct on the number of steps needed to complete our algorithm as well as induct on the number of tiles enclosed by cycle $C_i$ at each step. Suppose that $C_1$ encloses a unique tile $j$ giving that $\underline{v} = \underline{e_j}$, the unit vector with a 1 in the $j^{\text{th}}$ position. Note that if $d_j \neq 2$, no critical arrows can be formed as $C := \{i \in Q ~:~ (d_i,e_i) = (2,1)\} = \emptyset$. So, suppose that $d_j = 2$. Then, as $(d_j,v_{j}) = (2,1)$, $j \in S$ and the only way that the criticality condition could have failed is if this created two critical arrows i.e. the quiver and associated pair $(\underline{d}, \underline{e})$ is
     
$$k \longleftarrow j \longrightarrow \ell$$
$$(1,0) \leftarrow (2,1) \rightarrow (1,0)$$
\noindent where we must have $v_{k}=0=v_{\ell}$ as the only tile enclosed by a cycle is $j$. Note that since the subgraph $G_2$ associated to all tiles with $\underline{d}$ entry 2 is connected which means that $j$ must be the unique tile in $G_2$. Moreover, since both critical arrows would need to have source $j$, it is impossible that $j,k, \ell$ form a 3-cycle in the sense of the quiver. Hence, we are reduced locally to the acyclic case and can import the argument for the base case found in Theorem 4.2.3 of \cite{prequel}. \allowbreak \vspace{1em}

Now, suppose that if $C_1$ encloses $k$ tiles, the resulting vector $\underline{v}$ satisfies the criticality condition. We now aim to show that if $C_1$ encloses $k+1$ tiles, then the resulting vector $\underline{w} := \underline{v} + \underline{e_j}$ satisfies the criticality condition. Note that by our inductive assumption, $(\underline{v}, \underline{d})$ has at most one critical arrow. Suppose that $(\underline{w}, \underline{d})$ created 0 critical arrows. Since $\underline{w}$ only differs from $\underline{v}$ in the $j^{\text{th}}$ entry, it suffices to only check any arrows involving $j$. Moreover, we need to show that two critical arrows cannot be created by adding $+1$ to $\underline{v}$ to obtain $\underline{w}$. The only ways that this could happen is if $j$ is the unique vertex in $C$, i.e. $d_j$ is the unique 2 in $\underline{d}$ and the quiver and associated pair $(\underline{d}, \underline{e})$ has one of the following local orientations:

\begin{align*}
p~~ \longleftarrow ~~&j~~ \longleftarrow ~~q \tag{i}\\ 
(1,0) \leftarrow (2&,1) \leftarrow (1,1)\\
p~~ \longrightarrow ~~&j~~ \longrightarrow ~~q \tag{ii}\\
(1,1) \rightarrow (2&,1) \rightarrow (1,0)\\
p~~ \longleftarrow ~~&j~~ \longrightarrow ~~q \tag{iii}\\
(1,0) \leftarrow (2&,1) \rightarrow (1,0)\\
p~~ \longrightarrow ~~&j~~ \longleftarrow ~~q \tag{iv}\\
(1,1) \rightarrow (2&,1) \leftarrow (1,1)
\end{align*}
        
Note that case (iv) will not be possible as if $w_p = w_q = 1$, this means that both $p$ and $q$ were enclosed by $C_1$ in which case, since $j$ is in between these vertices, this could not happen without having enclosed $j$ as well by the connectedness of the tiles enclosed by $C_1$.\allowbreak  \vspace{1em}

Moreover, the orientation in case (iii) occurred in the base case. Namely, since $C_1$ enclosed a connected set of tiles, $p$ and $q$ cannot both be terminal vertices. Therefore, it suffices to analyze cases (i) and (ii). Since these cases are symmetric, we focus on case (i). \allowbreak \vspace{1em}

Note that if $j,p,q$ are not in a cycle in the quiver, then we are reduced to the acyclic case. Therefore, we rely on the proof of Theorem 4.2.3 in \cite{prequel}. If $j,p,q$ indeed are in a cycle in the quiver, then it is either a 3-cycle in the type $D_n$ part of the type II the spike of a type IV surface, the 4-cycle in the type III surface or is in the type $A_m$ part of any type $D_n$ quiver. Note that, $j,p,q$ cannot be in a larger central cycle in the type IV surface, if there is a 2 in the $\underline{d}$-vector, it occurs at the attaching spike of the quiver rather than in the central cycle.\allowbreak
\vspace{1em}

If $j,p,q$ forms a 3-cycle in the type $D_n$ part of the type II surface or a 4-cycle in the type $D_n$ part of the type III surface, then the nodes \textcolor{red}{$u,v$} and \textcolor{blue}{$w,x$} are on tiles $p,q$ and this degenerates into the acyclic case. If the quiver is type IV, then the \textcolor{blue}{blue} nodes are on tile $p$. Namely, the node \textcolor{blue}{$w$} is on the white vertex on the edge straddling $p$ and $q$. In order for $w_q=1$ and $w_p=0$, there must have existed a previous cycle enclosing tile $q$ and not tile $p$. However, by definition of $M_-$, tile $q$ would have only have been enclosed by a previous cycle if $j$ also was -- as the only edge enumerated in $M_-$ on tile $q$ is the straddling $p,q$ representing the singular arrow $p \to q$. Hence, this case is also impossible. Similarly, if $j,p,q$ forms a 3-cycle in the type $A_m$ part of the quiver. In order to flip the tile $q$, we must flip the tile $p$ giving that $w_q=1$ and $w_p=0$ is impossible. Therefore, $\underline{w}$ satisfies the criticality condition if $(\underline{v}, \underline{d})$ created 0 critical arrows. \allowbreak \vspace{1em}

If $(\underline{v}, \underline{d})$ created one critical arrow either $k \to \ell$ with $(d_k,v_k) = (2,1)$ and $(d_{\ell}, v_{\ell}) = (1,0)$ or $k' \rightarrow \ell'$ with $(d_{k'},v_{k'}) = (1,1)$ and $(d_{\ell'}, v_{\ell'}) = (2,1)$, the only way for the $\underline{w}$ to create another critical arrow is if we added 1 to the source of an arrow whose sink is in $C$. Namely, $w_j=1$ because if $w_j=2$, then no more critical arrows could be created. Therefore, $\underline{w}$  satisfies the criticality condition if $(\underline{v}, \underline{d})$ created one critical arrow. Thus, $\underline{v}$ is submodule-indexing as desired.
\end{proof}

Now that we have demonstrated the map in both directions, we show that the maps described in Theorem \ref{thm:EtoD} and Theorem \ref{thm:Dtoe} are inverses of one another to complete the proof of Theorem \ref{thm:bijection}.

\begin{proof}

    By definition of the algorithm in Theorem \ref{thm:Dtoe}, we see that the number of steps in the algorithm exactly equals $|e|$, where $\underline{e}$ is the resulting vector associated to the dimer $D$. Therefore, we will induct on $|e|$, i.e. the number of steps needed to perform the algorithm in Theorem \ref{thm:Dtoe}, to show that the maps described in Theorems \ref{thm:Dtoe} and \ref{thm:EtoD} are inverses of one another. \allowbreak
    \vspace{1em}
     
     First suppose that $|e| = 0$. Then, by both the algorithm for Theorem \ref{thm:Dtoe} and for Theorem \ref{thm:EtoD}, we have that $\underline{e}$ is associated to $M_-$ and vice versa. Suppose that some $m \in \nn$, when $|e| = k \leq m$, i.e. when we perform the first $k$ steps of the algorithm in Theorem \ref{thm:Dtoe}, the maps are inverses of one another. Let $\underline{e}'$ be such that $|e'| = m$. Then define $\underline{e}$ to be the resulting vector from adding 1 to the $i^{\text{th}}$ entry of $\underline{e}'$, so $|e| = m+1$. \allowbreak
     \vspace{1em}
     
     Applying Theorem \ref{thm:EtoD} to $\underline{e}$ and $\underline{e}'$, let $D'$ be the mixed dimer configuration corresponding to $\underline{e}'$ and $D$ be the mixed dimer configuration associated to $\underline{e}$. Note that as $|e| = m+1$, we had to perform $m+1$ flips from $M_-$ to obtain $D$ and by assumption, we had to perform one more flip at tile $i$ to obtain $D$ from $D'$. \allowbreak
     \vspace{1em}
     
     We aim to show that using the algorithm in Theorem \ref{thm:Dtoe}, the vector $\underline{e}$ is the vector associated to $D$. If we take the superimposition of $D \sqcup D'$, note that will consist of all 2-cycles and exactly one 4-cycle enclosing tile $i$. This implies that $D \sqcup M_-$ will also have at least one cycle containing $i$. Moreover, $D \sqcup M_-$ will have exactly one more cycle containing $i$ than that of $D' \sqcup M_-$. Hence, the corresponding vector obtained by performing the algorithm in Theorem \ref{thm:Dtoe} must be $\underline{e}' + \underline{e_i}$ which is exactly $\underline{e}$.\allowbreak
     \vspace{1em}
     
     Next we show the reverse direction. Suppose that given a mixed dimer configuration $D$, using the algorithm in Theorem \ref{thm:EtoD}, we obtain $\underline{e}$ such that $|e|=m+1$. Let $D'$ be the mixed dimer configuration associated to $\underline{e}'$, where $\underline{e}'$ is the result of subtracting 1 from the $i^{\text{th}}$ entry of $\underline{e}$. By the algorithm in Theorem \ref{thm:Dtoe}, we must have enclosed the tile $i$ in dimer $D$ $e_i$ number of times and $(e_i-1)$ number of times in dimer $D'$. By our inductive assumption, let $\underline{e}'$ be the vector associated to $D'$. Note that the order in which we flip tiles in the algorithm in Theorem \ref{thm:EtoD} does not matter, so to obtain the mixed dimer configuration $D$ from $D'$, it suffices to flip the tile $i$. Therefore, we see that the mixed dimer configuration associated to $\underline{e}$ prescribed by the algorithm in Theorem \ref{thm:EtoD} is exactly the mixed dimer configuration $D$ we began with. Therefore, we conclude that these maps are indeed inverses of each other.

\end{proof}

\begin{ex}
Examples \ref{ex:etoD} and \ref{ex:Dtoe} demonstrate each side of the bijection are inverses of one another.
\end{ex}

Now that we have established the connection between mixed dimer configurations and submodule-indexing vectors, we provide the connection between submodule-indexing vectors and representation theory. Ultimately, this gives the connection to the $F$-polynomial through representation theory. The following theorem is adapted from \cite{dwz} and \cite{tran}.

\begin{thm}\label{connectiontoreptheory}
Let $T$ be an ideal triangulation of a once-punctured $n$-gon and $J_T$ be its associated Jacobian algebra. Let $M$ be an indecomposable $J_T$-module associated to an arc $\gamma$ with dimension vector $\underline{d}$. Let $Gr_{\underline{e}}(M)$ be the variety of all submodules of $M$ with dimension vector $\underline{e}$ known as the quiver Grassmannian. Let $\chi$ denote the Euler characteristic. Then
$$F_{\underline{d}} = \sum_{\underline{e}} \chi(Gr_{\underline{e}}(M)) \prod_{i=1}^n u_i^{e_i}$$

where the sum ranges over $\underline{e} \in \zz^n$ with $0 \leq \underline{e} \leq \underline{d}$. 
\end{thm}

The following lemma, proven in \cite{tran}, helps to determine when there is a subrepresentation of a given dimension. 

\begin{lemma}\label{lemma-of-connectiontoreptheory}
Let $M'$ and $M''$ be vector spaces of dimensions $d'$ and $d''$, respectively, and $\phi: M' \rightarrow M''$ be a linear map of maximal possible rank min$(d', d'')$. Let $e'$ and $e''$ be two integers such that such that $0 \leq e' \leq d'$ and $0 \leq e'' \leq d''$. Then the following conditions are equivalent:
\begin{enumerate}
    \item There exist subspaces $N' \subseteq M'$ and $N'' \subseteq M''$ such that dim$(N') = e'$, dim$(N'') = e''$, and $\phi(N') \subseteq N''$.
    \item $e' - e'' \leq [d' - d'']_+$.
\end{enumerate}

\end{lemma}

For a representation of quiver $Q = (Q_0, Q_1)$ with dimension vector $\underline{d}$, we determine when there is a subrepresentation with of a given dimension vector $\underline{e}$. We do this by adapting an argument in \cite{tran} relying on Lemma \ref{lemma-of-connectiontoreptheory} and Theorem \ref{connectiontoreptheory}.

\begin{thm} \label{thm:4.2Tran}
For vectors $\underline{d}$ and $\underline{e} = (e_1, \dots, e_n)$, the coefficient of the monomial $u_1^{e_1} u_2^{e_2} \cdots u_n^{e_n}$ in $F_{\underline{d}}$ is nonzero if and only if
\begin{enumerate}
    \item $0 \leq \underline{e} \leq \underline{d}$,
    \item all arrows in $Q$ are acceptable, and
    \item $\nu(S) \leq 1$ for all connected components $S$ of $C$.
\end{enumerate}
If all of the conditions above are satisfied, then the coefficient of $u_1^{e_1} u_2^{e_2} \cdots u_n^{e_n}$ is $2^c$, where $c$ is the number of connected components $S$ such that $\nu(S) = 0$. 
\end{thm}

\begin{proof}
We may restrict attention to $\underline{e} \in \mathbb{Z}^n$ which satisfy the first two conditions of Definition \ref{submod-indexing}. Otherwise, ignoring singular arrows, Theorem \ref{connectiontoreptheory} and Lemma  \ref{lemma-of-connectiontoreptheory} imply that the coefficient of $u^{\underline{e}}$ in $F_{\underline{d}}$ is 0. If $\underline{d} \in \{ 0,1 \}^n$, then Theorem \ref{lemma-of-connectiontoreptheory} follows from Theorem 7.4 of \cite{tran} and \cite{dwz} in the acyclic case. The non-acyclic case reduces to the acyclic proof because the support of quiver will be acyclic when $\underline{d} \in \{ 0,1 \}^n$.\allowbreak
\vspace{1em}

If $d_i \geq 2$ for some $i$, then we show that $\underline{e}$ indexes a subrepresentation if and only if $\underline{e} \leq \underline{d}$ satisfies the criticality condition. Throughout this proof, we only consider arrows in $Q_1$ which are not singular. Note that $Q_1$ can only contain singular arrows when there are some $d_i \geq 2$. By the surface model, and more specifically using Lemma \ref{lemma:d_vector_2}, we can index $\underline{d}$ such that
\begin{align*}
\underline{d} = \underline{e}_p +  \cdots + \underline{e}_{r-1} + 2 \underline{e}_r + \cdots + 2 \underline{e}_{m-2} + \underline{e}_{m-1} + \underline{e}_m,
\end{align*}
for some $1 \leq p < r \leq m-2$, where $\underline{e}_i$ is the standard basis vector with $i^\text{th}$ entry 1 and all other entries 0.\allowbreak
\vspace{1em}

An indecomposable representation $M = (M_i)$ with dimension vector $\underline{d}$ can be selected in which we insist that the maps between $M_i$ and $M_{j}$ where $d_i = d_j$ are identity maps. We denote maps in the representation as $\varphi_{i,j}$ where $i$ denotes the index of the source and $j$ the index of the taret of the map.\allowbreak
\vspace{1em}

To compute $\chi(Gr_{\underline{e}}(M))$, we will construct all possible subrepresentations with dimension vector $\underline{e}$. When a representation $N = (N_i)$ has a dimension vector with $\underline{e} \leq \underline{d}$, the condition that each $N_i$ is a subspace of $M_i$ is satisfied. Therefore, to check that $N$ is a subrepresentation of $M$, it suffices to check that 
\begin{equation}
    \varphi_{i,j}(N_i) \subseteq N_j \text{ if } i \rightarrow j \text{ , and  } \varphi_{j,i}(N_j) \subseteq N_i \text{ if } j \rightarrow i
\end{equation} 
\noindent for all $i$ and $j$ adjacent in $Q$.\allowbreak 
\vspace{1em}
 
Recall from Definition \ref{defn:acceptablecritical} that $C = \{ i \in Q_0 : (d_i, e_i) = (2,1) \}$. Then for any $i \notin C$, there is only one possible subspace of $M_i$ of dimension $e_i$. For a connected component $S$ of $C$, if $i,j$ are vertices in $S$ and $i \rightarrow j$, then condition (1) above together with the fact that $\underline{e} \leq \underline{d}$ imply that $N_i \cong N_j$. Therefore, when a subspace of dimension 1 is chosen for one vertex of the component $S$, all vertices in that component must be assigned the same subspace. \allowbreak
\vspace{1em}

If $p \leq i,j \leq n$ with $i \rightarrow j \in Q_1$ are such that at least one of $i,j$ are not in $C$ and greater than or equal to $r$ and less than or equal to $m-2$, then it is straightforward to check that $\varphi_{i,j}(N_i) \subseteq N_j$. For example, if $r \leq i \leq m-2$ and $i \notin C$, then by definition of $C$, since $d_i = 2$, we have that $e_i = 0$ or $e_i = 2$. If $e_i = 0$, then $N_i = 0$ so $\varphi_{i,j}(N_i) = 0$. If $e_i = 2$, since $i \rightarrow j$ is acceptable, by Lemma \ref{lemma-of-connectiontoreptheory}, we have that $e_j = d_j$, so $N_j = M_j$, which contains $\varphi_{i,j}(N_i)$. Therefore, it only remains to show that the property (1) holds for $(i,j) = (r-1, r)$ if $r \in C$, and that the property holds for $(i,j) = (m-2, m-1), (m-2,m)$ if $m-2 \in C$. \allowbreak
\vspace{1em}

We will consider which one-dimensional subspaces can be assigned to each component. To do this, we construct three distinct one-dimensional subspaces of the 2-dimensional subspace $M_{m-2}$. For one-dimensional $M_j$ and 2-dimensional $M_k$ with  $j \rightarrow k \in Q_1$, let $V_{\{j,k\}} = Im(\varphi_{j,k})$. In the case that $k \rightarrow j \in Q_1$, let $V_{\{j,k\}} = Ker(\varphi_{k,j})$. Then $V_{\{r-1,r\}}$, $V_{\{m, m-2\}}$, and $V_{\{m-1, m-2\} }$ are three distinct one-dimensional subspaces of $M_{m-2}$.\allowbreak \vspace{1em}

We call a pair of vertices $(i,j)$ a critical pair when the arrow between them in $Q_1$ is a critical arrow. Note that when $(i,j) = (r-1, r)$ is a critical pair, property (1) is satisfied for $N_{r-1}$ and $N_r$ if and only if $N_r = V_{\{r-1, r\}}$.  When $(i,j) = (m-2, m)$ is a critical pair, property (1) is satisfied if and only if $N_{m-2} = V_{\{m, m-2\}}$. Analogously, when $(i,j) = (m-2, m-1)$, (1) is satisfied if and only if $N_{m-2} = V_{\{m-1, m-2\}}$.\allowbreak
\vspace{1em}

Let $S$ again be a connected component of $C$. Recall that once a subspace of dimension 1 is chosen for one vertex of $S$, all vertices must of $S$ must be assigned the same subspace. We consider cases for the number of critical arrows in $S$.  If $\nu(S) = 0$, then any one-dimensional subspace of $\Bbbk^2$ may be assigned to all of the vertices of $S$. This corresponds to $\mathbb{P}^1$, so $\chi(\mathbb{P}^1) = 1$. Next suppose that $\nu(S) = 1$. If $(m-2, m)$ is a critical pair and $m-2 \in S$, then the chosen subspace for $S$ must be $V_m$. Similarly, if $(m-2, m-1)$ is a critical pair and $m-2, \in S$, then the chosen subspace assigned to the vertices of $S$ must be $V_{m-1}$. Since there is no choice of subspace, the Euler characteristic is 1. Finally, if $\nu(S) \geq 2$, there is no one-dimensional subspace which can be assigned to the vertices of $S$ which will satisfy condition (1). This implies that the Euler characteristic is 0. 
\end{proof}

With this, the only thing that remains to show to prove Theorem \ref{thm:main} is to verify that the number of cycles on a mixed dimer configuration $D$ associated to a submodule-indexing vector $\underline{e}$ is given by the number of connected components $S$ of $C$ such that with $\nu(S) = 0$. 

\begin{proof}
First note that we cannot have a cycle on a tile that is a dimer configuration, i.e. we need $d_i = 2$ in order for $i$ to potentially be enclosed by a cycle. Also, note that there are no cycles in $M_-$ by the convention of the black and white coloring in the definition of $M_-$, so in order for tile $i$ to be enclosed by a cycle, we must have that $e_i \geq 1$. This tells us that the set of all tiles that can potentially be enclosed by a cycle are contained within $S = \{i \in Q~:~ (d_i,e_i) = (2,1)\}$.\allowbreak
\vspace{1em}
     
Fix $D$ associated to $(\underline{d}, \underline{e})$. Let $C$ be a connected component of $S$. Suppose that $\nu(C) \neq 0$, that is, there is some $c \in C$ such that $c$ is the vertex of a critical arrow. Namely, this $c$ must be connected to a tile outside of $S$ by definition of the $d,e$ coordinates of the other vertex of the critical arrow. Let $c'$ be the other vertex of the critical arrow involving $c$. Note that if $c \to c'$ about the edge $\alpha$, then $(d_{c'}, e_{c'}) = (1,0)$ and we have that $n_{c,c'} = \max(d_c-d_{c'},0) + (e_{c'}-e_c) = 1-1=0$. If $c' \to c$ about the edge $\alpha$, then $(d_{c'}, e_{c'}) = (1,1)$ and we have that $n_{c',c} = \max(d_c'-d_{c},0) + (e_{c}-e_{c'}) = 0$. Therefore, by Lemma \ref{lemma:MijNij}, the edge $\alpha$ cannot have an edge from $D$ which tells us that no cycle in $D$ can enclose tile $c$. This tells us that if a cycle encloses any tile or collection of tiles in $C$, then $\nu(C) = 0$.\allowbreak
\vspace{1em}
 
Now, suppose that $C$ is a connected component of $S$ with $\nu(C) = 0$, i.e. no critical arrows involve any vertex $c$ from $C$. Suppose that $C$ is comprised of the tiles $i, \dots, m$ where $i$ is the minimal tile in $C$ and $i'$ is the maximal tile in $C$ with respect to the indexing in $Q$. We aim to show that there exists a cycle in $D$ enclosing all of $C$. Note that for $i \leq j \leq i'-1$, we have that $(d_j,e_j) = (2,1)$, so there are no edges in $D$ are on the edge straddling tiles $j$ and $j+1$ by Lemma \ref{lemma:MijNij} since $n_{j,j+1} = 0$. This implies that if we can show that all the boundary edges of $C$ are in $D$, that there is a cycle enclosing all of $C$ in $D$. \allowbreak
\vspace{1em}

Note that as $e_j =1$ for all $i \leq j \leq i'$, we have that each of the tiles in $C$ have been flipped exactly once from $M_-$ using our bijection of adding 1's to the $\underline{e}$-vector coinciding with flipping the corresponding tile from Theorem \ref{thm:Dtoe}. Note that if a boundary edge $\alpha$ on tile $t$ is oriented black to white clockwise with respect to $G_2$, then by definition of $M_-$, $\alpha$ must have two edges distinguished in $M_-$. After one flip at tile $t$, $w(\alpha)$ decreases by 1. This gives that this edge now appears distinguished as a single edge in $D$. On the other hand, if $\alpha$ is oriented white to black clockwise with respect to $G_2$, then by definition of $M_-$, $\alpha$ has no edges distinguished in $M_-$. After performing a flip on tile $t$, we have that $w(\alpha)$ increases by 1. This gives that $\alpha$ now has exactly one edge distinguished in $D$. Hence, all the boundary edges on $C$ have exactly one edge on them in $D$.\allowbreak
\vspace{1em}

Now, we need to show that the cycle closes up on the interior edges of $G$, i.e. if $i > 0$ and/or $i' < m-1$, we need to verify that the edges straddling tiles $i,i-1$ and $i',i'+1$ have one edge distinguished in $D$. Note that as $i,i' \in C$ with $\nu(C) = 0$, no arrow involving $i$ or $i'$ is critical. By the structure of $\underline{d} \in \Phi_+$, we have that $d_{i-1} \geq 1$ and $d_{m+1} \geq 1$. Seeking a contradiction, suppose that $n_{i, i-1} = 0$. Then, $\max(d_i-d_{i-1},0) = e_i-e_{i-1}$, i.e. $\max(2-d_{i-1},0) = 1-e_{i-1}$. Since $e_{i-1},d_{i-1} \in \{0,1,2\}$, this only occurs if either $d_{i-1} = 1$ giving that $e_{i-1} = 0$ or if $d_{i-1} = 2$ giving that $e_{i-1} = 1$. In the first case, this gives that the arrow $i \to i-1$ is critical which contradicts that $\nu(C) = 0$. In the second case, $i-1 \in C$ which contradicts the minimality of $i$. Hence, $n_{i,i-1} >0$ and by Lemma \ref{thm:Dtoe}, this means there must be an edge straddling tiles $i, i-1$ in $D$. Note that this same argument holds for showing that there must be an edge in $D$ straddling tiles $i',i'+1$. Therefore, we see that one cycle is formed around all of $C$ when $\nu(C) = 0$. Hence, the number of cycles in $D$ is exactly the number of connected components $C$ of $S$ with $\nu(C) = 0$. 
\end{proof}

Therefore, we have shown that the $F$-polynomial is indeed given by the mixed dimer configuration generating function given in Theorem \ref{thm:main}.

\begin{subsection}{$g$-Vectors}
\label{subsection:gvectors}
As detailed in \cite{ca4}, see Definition \ref{defn:fpoly} above, a cluster variable can be reconstructed from an $F$-polynomial, but only if it is also accompanied by the extra data of a $g$-vector.
In our prequel to this paper \cite{prequel}, we were able to define a weighting scheme on the edges of the base graph and then show that the $g$-vector can be realized using weight of the minimal mixed dimer configuration in the case of an acyclic quiver of type $D_n$. We conjecture that a similar weighting scheme exists in the non-acyclic case, which would lead to a similar interpretation for $g$-vectors, in order to have a full dimer theoretic interpretation of the Laurent expansions for these cluster variables.

\begin{conj}
The $g$-vector associated to $Q$ and $\underline{d}$-vector, denoted $\underline{g} = \underline{g}(Q, \underline{d})$, is given by 

$$\underline{g} = \deg\left(\frac{\text{wt}(M_-)}{\underline{x}^{\underline{d}}}\right)$$

\noindent where $\deg(x_0^{a_0}x_1^{a_1} \cdots x_{n-1}^{a_{n-1}}) := (a_0,a_1, \dots, a_{n-1})$ and $\text{wt}(M_-)$ is some weighting scheme given on the base graph.
\end{conj}

Based on preliminary computations, it appears that this weighting scheme must be different than the one in \cite{prequel} to account for singular arrows. We also conjecture that the edge weights may have to be allowed to be more complicated terms than just a single $x_i$, which differs from the acyclic case. The proof of our result for $g$-vectors in the acyclic case also relies on a classification of $g$-vectors found in \cite{tran}, that we would have to extend to the cyclic case.
\end{subsection}
\end{section}

\begin{appendix}
\begin{section}{Classification of $\underline{d}$-vectors} \label{section:d-vectors}
In this appendix, we classify all crossing vectors, $\underline{d}$-vectors, in all non-acyclic type $D_n$ cluster algebras. To accomplish this, we rely on the surface model for type $D_n$ cluster algebras - a once-punctured disk with $n$ marked points. We split our work into four types following Vatne's classification of type $D_n$ quivers \cite{vatne}. In each type, we rely on the computing the crossing vector associated to an arc $\gamma$ in a triangulation of the once-punctured disk with $n$ marked points by keeping track of the arcs that $\gamma$ crosses in the given triangulation \cite{fst}.\allowbreak
\vspace{1em}

If $\gamma$ crosses any arc in a triangulation twice, we streamline the Vatne categorization to make our catalog more concise. To this end, we describe the precise situations where this can occur. Any arc of the triangulation that is crossed twice by $\gamma$ must be a peripheral arc i.e. not incident to the puncture. Moreover, if we follow the path of $\gamma$, out of all the peripheral arcs that are crossed twice, there is a unique peripheral arc that is closest to the puncture, up to isotopy.  Without loss of generality, we let $c$ denote this arc. This arc splits the triangulation into two parts. On one side of $c$, the sub-triangulation is that of a unpunctured polygon i.e. has corresponding sub-quiver $Q'$ is of type $A_m$ with vertex $c$ is one of its endpoints.  On the other side of $c$, the sub-triangulation is that of a punctured bi-gon, with arcs $a$ and $b$ incident to the puncture such that the only peripheral arc that $\gamma$ crosses is $c$ (following Vatne's labeling of vertices in Figure \ref{fig:vatnetypes} that we repeat here for convenience). Arc $\gamma$ must cross arcs $a$ and $b$ along its path, and hence the corresponding sub-quiver consists of a fork of $a$ and $b$ attached to $c$, with any of the four orientations on the edges $a - c$ and $b - c$ depending on the four possible triangulations of the bigon.\allowbreak
\vspace{1em}

\begin{figure}[H]
    \centering
    \includegraphics[scale=.25]{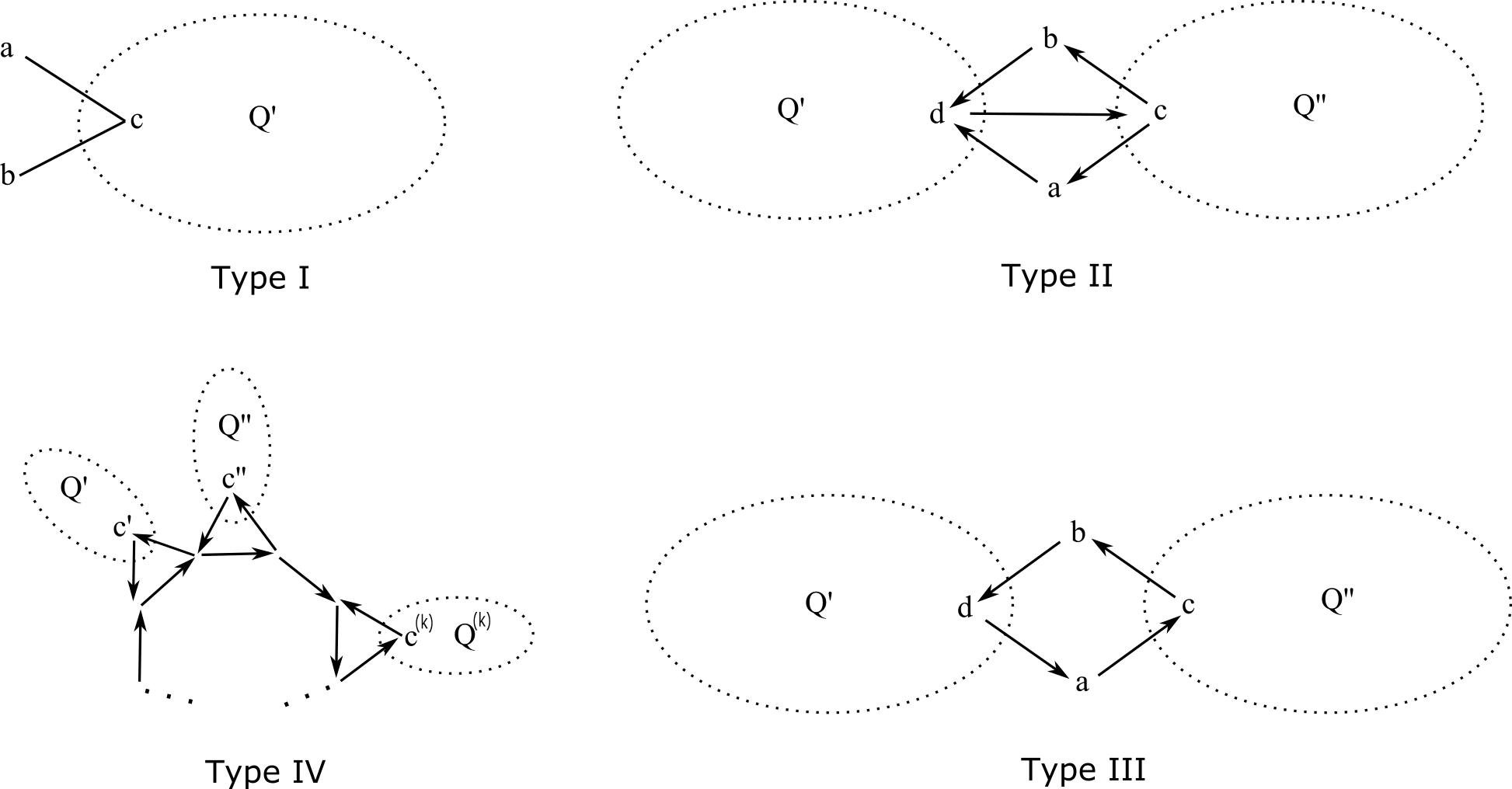}
    \caption{ }
\end{figure}

With these observations above, the types II and III surfaces degenerate to the type I surface. Moreover, arc $\gamma$ must cross at most three other arcs, each of which corresponds to a vertex of the quiver. Either the vertices $a,c,d$ and not $b$ or the vertices $b,c,d$ and not $a$, which implies the type $D_n$ parts of both of these surfaces degenerate to the fork in the type $D_n$ part of the type I surface. \allowbreak
\vspace{1em}

Furthermore, we can streamline these types of surfaces even more so. Observe that if an arc $\gamma$ is supported on some subset of $a,b,c$, but not $d$ or some subset of $a,b,d$, but not $c$, we can treat the arcs $d,c$ respectively as boundary arcs and the arc can be considered a type I arc.  \allowbreak
\vspace{1em}

The last reduction we make reduces case work for the type A parts of the surfaces. Namely, we rely on sectioning off the triangulations of our once-punctured disk into its type $A_m$ parts. In our catalog, we focus on the general structure of the $\underline{d}$-vectors and leave space for all potential type $A_m$ crossing vectors that correspond to triangulations of polygons.\allowbreak
\vspace{1em}

With those reductions, we begin by cataloging the type I surface. There are nuances between the orientation of the arrows between $a,b,c$ with what the surface model looks like; however, there are six families of arcs that are produced in any type I case. 

\begin{lemma}\label{lemma:typeIarcs}
Consider any orientation of the quiver corresponding to the type I surface. Suppose that first three coordinates of the $\underline{d}$-vector are given by $a,b,c$ respectively. The fourth coordinate is the first coordinate in the type $A_m$ part of the surface. Then, there are six families:

$$(0,0,0, \text{type }A)~~(0,0,1, \text{type }A)~~(1,1,1, \text{type }A)$$
$$(0,1,1, \text{type }A)~~(1,0,1, \text{type }A)~~(1,1,2, \text{type }A)$$
\end{lemma}

In order to see the proof of this lemma, we systemically go through the possibilities for arcs when the arrows between $a$ and $c$ and $b$ and $c$ are pointed in the same direction and when they are pointing in opposite directions. 

\begin{subsubsection}{Type IA}\label{subsection: type1a}
Let type IA be the orientation of the type I quiver where $a \to c$ and $c \to b$. We begin by categorizing the families of arcs that do not cross any arc in a triangulation twice. Namely, we analyze the possibilities for the support of $a,b,c$.\\

\noindent \textbf{Family I.} Any arc that does not cross the type $D_n$ part of the surface i.e. the arcs $a,b,c$ will be of the form:  $\underline{d} = (0,0,0, \text{type }A_m)$.\\

\noindent \textbf{Family II.} Arcs that emanate from the marked point $y$ and wrap clockwise around the puncture: $\underline{d} = (0,0,1, \text{type }A)$.\\

\noindent \textbf{Family III.} Arcs that emanate from the marked point $x$ and wrap counterclockwise around the puncture: $\underline{d} = (1,1,1, \text{type }A)$.\\

\noindent \textbf{Family IV.} Arcs going into the puncture, tagged $\underline{d}_{\Bowtie}$ and untagged $\underline{d}$ are given by: $\underline{d}_{\Bowtie} = (0,1,1, \text{type }A)$ and $\underline{d} = (1,0,1, \text{type }A)$.\\

These are all the families of arcs that cross existing arcs in the triangulation at most once, pictured in Figure \ref{fig:1Acrossonce}. 

\begin{figure}[H]
    \centering
    \includegraphics[width=\textwidth]{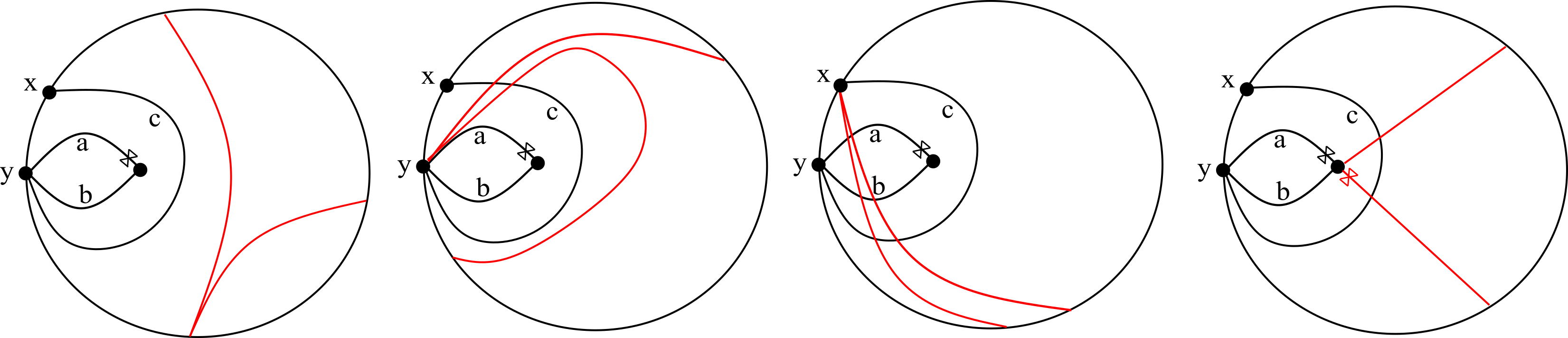}
    \caption{Examples of arcs in families for surface IA: I, II, III and IV from left to right drawn in \textcolor{red}{red}.}
    \label{fig:1Acrossonce}
\end{figure}

We now consider when arcs can cross multiple times i.e. where 2's can appear in the $\underline{d}$-vectors. Up to isotopy, the only type $D_n$ arc that can be crossed twice is $c$; namely, $a$ and $b$ cannot be crossed twice. Observe that arcs whose endpoints only involve marked points in the ``type $D_n$" part i.e. endpoints of arcs $a,b,c$ cannot cross any arc twice, so at least one of the endpoints must be involved the ``type $A_m$" part i.e. a marked point in the triangulation of the $n$-gon. In particular, any such arc must wrap counterclockwise around the puncture. This gives the last family of arcs in Type 1A: $\underline{d} = (1,1,2, \overrightarrow{\text{type }A_\alpha} + \overrightarrow{\text{type }A_\beta})$, where we split up our arc into two pieces: $\alpha$, an arc in the $n$-gon involving the crossings that appear before the first crossing of $c$ and $\beta$, an arc in the $n$-gon involving the crossings that appear after the second crossing of $c$. See Figure \ref{fig:1Acrosstwice}.

\begin{figure}[H]
    \centering
    \includegraphics[scale=.22]{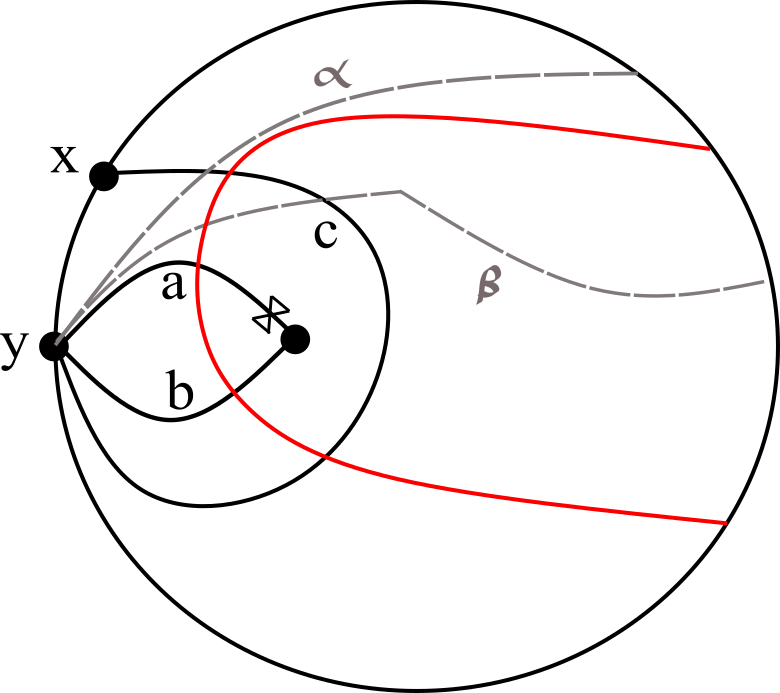}
    \caption{Example of arc in that cross twice in surface type IA drawn in \textcolor{red}{red}, where $\alpha, \beta$ components are drawn dashed in gray.}
    \label{fig:1Acrosstwice}
\end{figure}

\begin{rmk}
We note that for crossing vectors $\underline{d}$ of this format, a quiver theoretic formula for F-polynomials was explored by unpublished work of the second author and Lauren Williams where the F-polynomial was conjecturally written as a difference. The contributions being subtracted therein appear to correspond to the node polychromatic mixed dimer configurations that appear in the larger poset $\overline{P}$ defined in Definition \ref{defn:largerposet} but not poset $P$ from Definition \ref{defn:nodesposet}.
\end{rmk}

\end{subsubsection}

\begin{subsubsection}{Type IB}\label{subsection: type1c}
Let type IB be the orientation of the type I quiver where $a \to c$ and $b \to c$. We begin by categorizing the families of arcs that do not cross any arc in a triangulation twice. Namely, we analyze the possibilities for the support of $a,b,c$.\\ 

\noindent \textbf{Family I.} Any arc that does not cross the type $D_n$ part of the surface i.e. the arcs $a,b,c$ will be of the form: $\underline{d} = (0,0,0, \text{type }A)$.\\

\noindent \textbf{Family II.} Arcs that emanate from the marked point $y$ and wrap clockwise around the puncture: $\underline{d} = (0,1,1, \text{type }A)$.\\

\noindent \textbf{Family III.} Arcs that emanate from the marked point $x$ and wrap counterclockwise around the puncture: $\underline{d} = (1,0,1, \text{type }A)$.\\

\noindent \textbf{Family IV.} Arcs going into the puncture, tagged $\underline{d}_{\Bowtie}$ and untagged $\underline{d}$ that involved at least one endpoint in the ``type $A_m$" part of the surface: $\underline{d}_{\Bowtie} = (0,1,1, \text{type }A)$ and $\underline{d} = (1,0,1, \text{type }A)$.\\

We also have two additional arcs that are the tagged versions of the arcs $a,b$ which give two additional $\underline{d}$-vectors: $\underline{d}_{a\Bowtie} = (0,1,0,\dots, 0)$ and $\underline{d}_{b\Bowtie} = (1,0,\dots, 0)$.\\

These are all the families of arcs that cross existing arcs in the triangulation at most once, pictured in Figure \ref{fig:1Bcrossonce}. 

\begin{figure}[H]
    \centering
    \includegraphics[width=\textwidth]{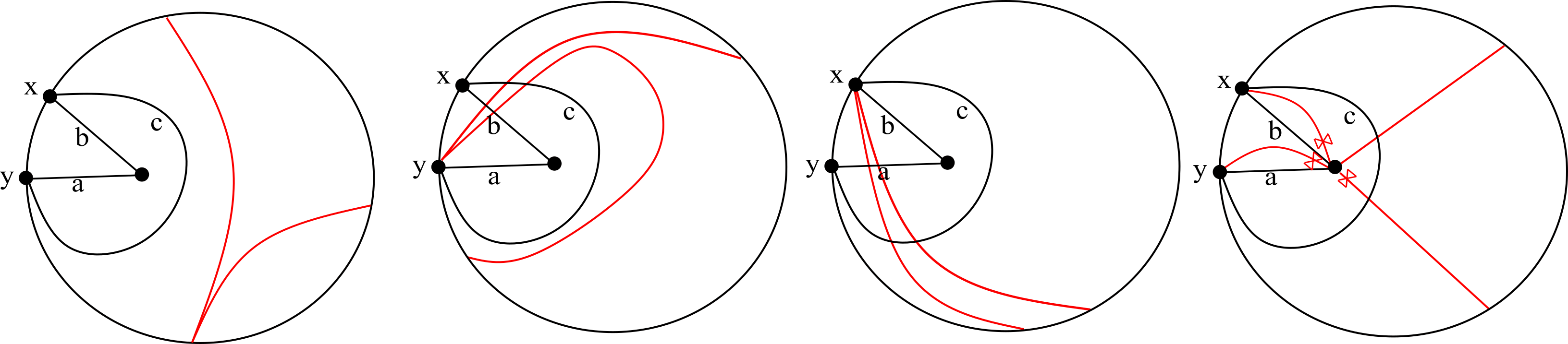}
    \caption{Examples of arcs in families for surface IB: I, II, III and IV from left to right drawn in \textcolor{red}{red}.}
    \label{fig:1Bcrossonce}
\end{figure}

We now consider when arcs can cross multiple times i.e. where 2's can appear in the $\underline{d}$-vectors. As in Type IA, we see again that the only family $\underline{d}$-vector in this case is: $\underline{d} = (1,1,2, \overrightarrow{\text{type }A_\alpha} + \overrightarrow{\text{type }A_\beta})$, where $\alpha, \beta$ are defined as in Type 1A. See Figure \ref{fig:1Bcrosstwice}.

\begin{figure}[H]
    \centering
    \includegraphics[scale=.18]{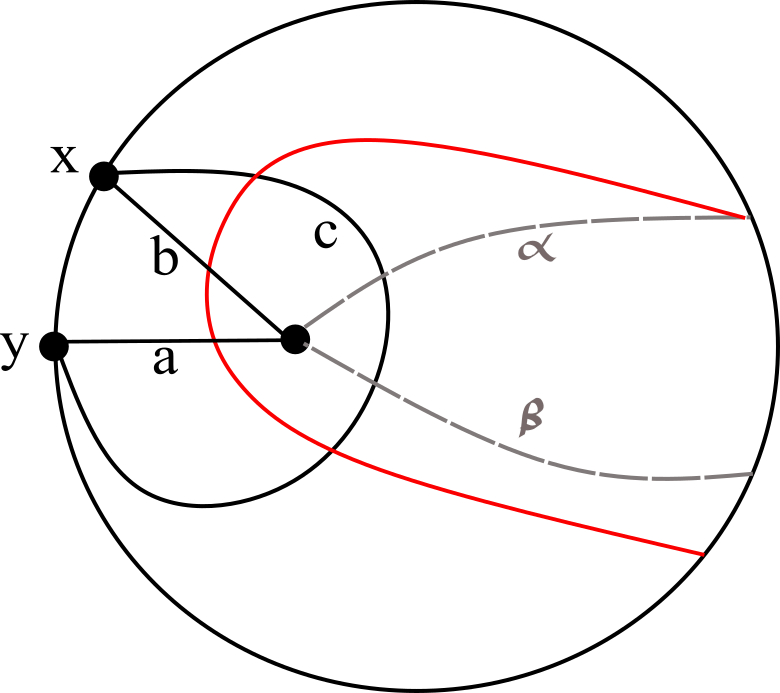}
    \caption{Example of arc in that cross twice in surface type IA drawn in \textcolor{red}{red}, where $\alpha, \beta$ components are drawn dashed in gray.}
    \label{fig:1Bcrosstwice}
\end{figure}
\end{subsubsection}

\begin{subsubsection}{Type II}
The surface model in type II has two potential type $A_m$ parts corresponding to the quivers $Q', Q''$ i.e. two triangulations of polygons $P', P''$. We begin by cataloging the families of arcs that do not cross any arc in a triangulation twice. We index the $\underline{d}$-vectors as follows: the first four coordinates in order are $a,b,c,d$, the next set of coordinates are indexed by arcs in the polygon $P'$ and the last set of coordinates are indexed by arcs in the polygon $P''$. \allowbreak
\vspace{1em}
Recall any arc that crosses an arc in the triangulation more than once or any arc that is supported on some subset of $a,b,c$, but not $d$ or some subset of $a,b,d$, but not $c$, are previously cataloged in the type I surface. So, it suffices to catalog any arcs in the type II surfaces that are supported on both $c$ and $d$. It turns out that there are only two families of arcs that are supported on both $c$ and $d$ shown in Figure \ref{fig:2allfams}. \allowbreak
\vspace{1em}

\noindent \textbf{Family I.} Arcs that start in one polygon and end in the other, wrapping clockwise around the puncture. These arcs cannot start or end at the marked points $x$ or $y$: \\ $\underline{d} = (1,1,1,1, \text{type }A , \text{type }A)$.\\

\noindent \textbf{Family II.} Arcs that start in one polygon and end in the other, wrapping counterclockwise around the puncture. These arcs cannot start or end at the marked points $x$ or $y$: $\underline{d} = (0,0,1,1, \text{type }A , \text{type }A)$.

\begin{figure}[H]
    \centering
    \includegraphics[scale=.2]{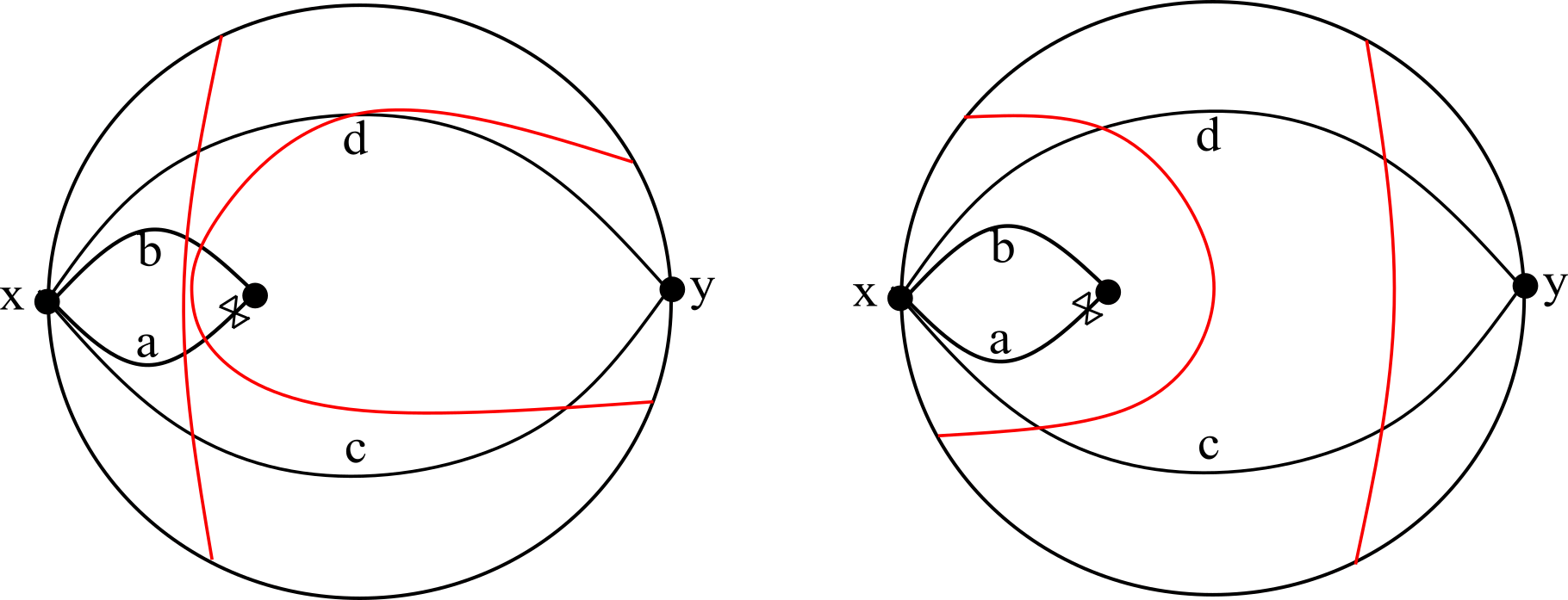}
    \caption{Examples of arcs in families I and II from left to right in the type II surface drawn in \textcolor{red}{red}.}
    \label{fig:2allfams}
\end{figure}
\end{subsubsection}

\begin{subsubsection}{Type III}
Similarly to type II, the surface model in type III has two potential type $A_m$ parts corresponding to the quivers $Q', Q''$ i.e. two triangulations of polygons $P', P''$. We begin by cataloging the families of arcs that do not cross any arc in a triangulation twice. We index the $\underline{d}$-vectors as follows: the first four coordinates in order are $a,b,c,d$, the next set of coordinates are indexed by arcs in the polygon $P'$ and the last set of coordinates are indexed by arcs in the polygon $P''$. \\

We index the $\underline{d}$-vectors as follows: the first four coordinates in order are $a,b,c,d$, the next set of coordinates are indexed by arcs in the polygon $P'$ and the last set of coordinates are indexed by arcs in the polygon $P''$. Recall that any arc that crosses an arc in the triangulation twice has already previously been cataloged in type I. Additionally, if an arc is supported on some subset of $a,b,c$, but not $d$ or some subset of $a,b,d$, but not $c$, it has also has been previously cataloged in type I. So, it suffices to catalog any arcs in the type III surfaces that are supported on both $c$ and $d$. It turns out that there are only two families of arcs that are supported on both $c$ and $d$ shown in Figure \ref{fig:3allfams}. \allowbreak
\vspace{1em}

\noindent \textbf{Family I.} Arcs that start in one polygon and end in the other, wrapping clockwise around the puncture. These arcs cannot start or end at the marked points $x$ or $y$: \\ $\underline{d} = (1,0,1,1, \text{type }A , \text{type }A)$.\\

\noindent \textbf{Family II.} Arcs that start in one polygon and end in the other, wrapping counterclockwise around the puncture. These arcs cannot start or end at the marked points $x$ or $y$: $\underline{d} = (0,1,1,1, \text{type }A , \text{type }A)$.

\begin{figure}[H]
    \centering
    \includegraphics[scale=.2]{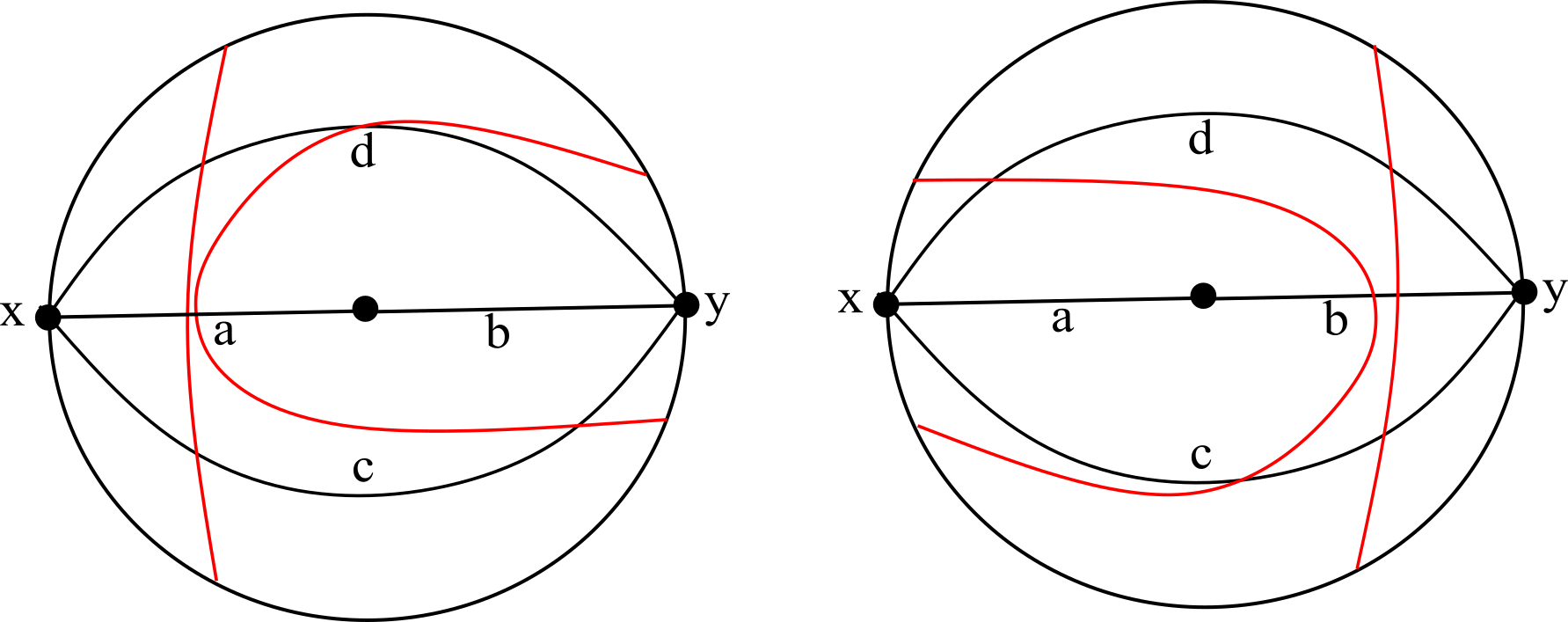}
    \caption{Examples of arcs in families I and II from left to right in the type III surface drawn in \textcolor{red}{red}.}
    \label{fig:3allfams}
\end{figure}
\end{subsubsection}

\begin{subsection}{Type IV}
The surface model in type IV has consists of a central $n$-cycle for $n \geq 3$ and can have $k \leq n$ ``spikes" i.e. $k$ attached triangulated polygons labeled $P_1, P_2, \dots, P_k$. We refer to the inner punctured disk that corresponds to the central $n$-cycle as $D$. For notational convenience, we index the $\underline{d}$-vectors based on their support and only specify where the non-zero entries occur. \\

In order to catalog arcs in a type IV surface, we need to develop more notation. Let the attached polygons of the surface be labeled in  clockwise cyclic order $P_1, P_2, \dots, P_k$. Let $a_1, a_2, \dots, a_k$ be the arcs that are shared by the punctured disk $D$ and a polygon $P_i$ - where $a_i$ is the label of the arc between $D$ and $P_i$. Let the triangles in $D$ be labeled such that in clockwise cyclic order, the triangle in $D$ using arc $a_i$ read $b_i, a_i, b_{i+1}$. Call the vertices that are the start and end of each $a_i$ ``gluing vertices."  To demonstrate this notation, see Figure \ref{fig:petalnotation}. With this notation, we are ready to catalog more families of arcs.  We begin by cataloging the families of arcs that do not cross any arc in a triangulation twice.\\

\begin{figure}[H]
    \centering
    \includegraphics[scale=.25]{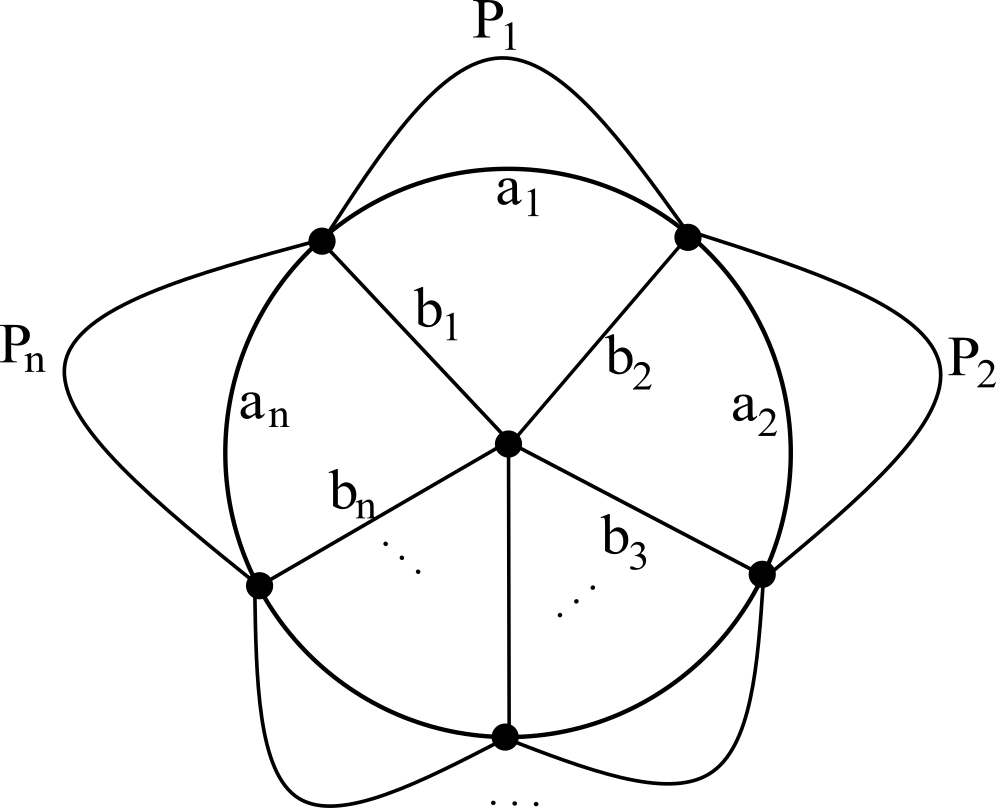}
    \caption{ }
    \label{fig:petalnotation}
\end{figure}

\noindent \textbf{Family I.} Arcs that are completely contained in a single polygon $P$: see the leftmost figure in the top row of Figure \ref{fig:4IthroughVII}.

$$\underline{d} = \begin{cases}
\text{type }A \text{ vector } &\text{if in }  P\\
0 &\text{otherwise}
\end{cases}$$

\noindent \textbf{Family II.} Arcs that start and end in different polygons $P_i, P_j$ where $i<j$ that wrap clockwise around the puncture. In addition, disallow arcs that start of end at a gluing vertex. See the second from the left figure in the top row of Figure \ref{fig:4IthroughVII}. For coordinates $\alpha$ representing an arc in the sub-triangulation of $D$, we have

$$d_{\alpha} = \begin{cases}
1 &\text{if } \alpha = b_{i+1}, b_{i+2}, \dots, b_j~\text{or } \alpha = a_i \text{ or } a_j\\
0 &\text{otherwise}
\end{cases}$$

\noindent For coordinates $\beta$ representing an arc in the sub-triangulation of one of the petal polygons, we have
$$d_{\beta} = \begin{cases}
\text{type } A \text{ vector } &\text{if } \beta \text{ in } P_i \text{ or } P_j\\
0 &\text{otherwise}
\end{cases}$$

\noindent \textbf{Family III.} Arcs that start and end in different polygons $P_i, P_j$ where $i<j$ that wrap counterclockwise around the puncture. In addition, disallow arcs that start of end at a gluing vertex. See the third from the left figure in the top row of Figure \ref{fig:4IthroughVII}. For coordinates $\alpha$ representing an arc in the sub-triangulation of $D$, we have

$$d_{\alpha} = \begin{cases}
1 &\text{if } \alpha = b_{1}, \dots, b_i, b_{j+1}, \dots b_k ~\text{or } \alpha = a_i \text{ or } a_j\\
0 &\text{otherwise}
\end{cases}$$
\noindent For coordinates $\beta$ representing an arc in the sub-triangulation of one of the petal polygons, we have
$$d_{\beta} = \begin{cases}
\text{type } A \text{ vector } &\text{if } \beta \text{ in } P_i \text{ or } P_j\\
0 &\text{otherwise}
\end{cases}$$

\noindent \textbf{Family IV.} Arcs that begin at a gluing vertex on $P_i$ and wrap counterclockwise around the puncture ending in a non-gluing vertex in $P_j$. See the rightmost figure in the top row of Figure \ref{fig:4IthroughVII}. For coordinates $\alpha$ representing an arc in the sub-triangulation of $D$, we have

$$d_{\alpha} = \begin{cases}
1 &\text{if } \alpha = b_1, \dots, b_{i-1},b_{j+1}, \dots, b_k~\text{or } \alpha = a_j \\
0 &\text{otherwise}
\end{cases}$$
\noindent For coordinates $\beta$ representing an arc in the sub-triangulation of one of the petal polygons, we have
$$d_{\beta} = \begin{cases}
\text{type } A \text{ vector } &\text{if } \beta \text{ in } P_j\\
0 &\text{otherwise}
\end{cases}$$
\noindent If such an arc began at the other gluing vertex on $P_i$, the support of $\underline{d}$ includes a 1 on the index corresponding to the arc $b_i$.\\

\noindent \textbf{Family V.} Arcs that begin at a gluing vertex on $P_i$ and wrap clockwise around the puncture ending in a non-gluing vertex in $P_j$. See the leftmost figure in the bottom row of Figure \ref{fig:4IthroughVII}. For coordinates $\alpha$ representing an arc in the sub-triangulation of $D$, we have

$$d_{\alpha} = \begin{cases}
1 &\text{if } \alpha = b_{i+1}, \dots, b_{j}~\text{or } \alpha = a_j \\
0 &\text{otherwise}
\end{cases}$$
\noindent For coordinates $\beta$ representing an arc in the sub-triangulation of one of the petal polygons, we have
$$d_{\beta} = \begin{cases}
\text{type } A \text{ vector } &\text{if } \beta \text{ in } P_j\\
0 &\text{otherwise}
\end{cases}$$
\noindent If such an arc began at the other gluing vertex on $P_i$, then $b_{i+1}=0$ rather than 1. \\

\noindent \textbf{Family VI.} Arcs that begin at the gluing vertex between $P_i$ and $P_{i-1}$ on $P_i$ and traverse clockwise around the puncture ending in $P_i$. See the center figure on the bottom row of Figure \ref{fig:4IthroughVII}.

$$d_{\alpha} = \begin{cases}
1 &\text{if } \alpha \neq b_{i+1}, \text{ or } \alpha = a_i \\
0 &\text{otherwise}
\end{cases} \hspace{3em} d_{\beta} =\begin{cases}
\overrightarrow{\text{type }A} &\text{if } \beta \text{ in } P_i\\
0 &\text{otherwise}
\end{cases}$$

\noindent \textbf{Family VII.} Similarly to Family VI, arcs that begin at the gluing vertex between $P_i$ and $P_{i+1}$ on $P_i$ and traverse counterclockwise around the puncture ending in $P_i$. See the rightmost figure on the bottom row of Figure \ref{fig:4IthroughVII}.

$$d_{\alpha} = \begin{cases}
1 &\text{if } \alpha \neq b_i, \text{ or } \alpha = a_i \\
0 &\text{otherwise}
\end{cases} \hspace{3em} d_{\beta} =\begin{cases}
\overrightarrow{\text{type }A} &\text{if } \beta \text{ in } P_i\\
0 &\text{otherwise}
\end{cases}$$

\begin{rmk}
Families I through VII are all shown in Figure \ref{fig:4IthroughVII}.
\end{rmk}

\begin{figure}[H]
    \centering
    \includegraphics[width=\textwidth]{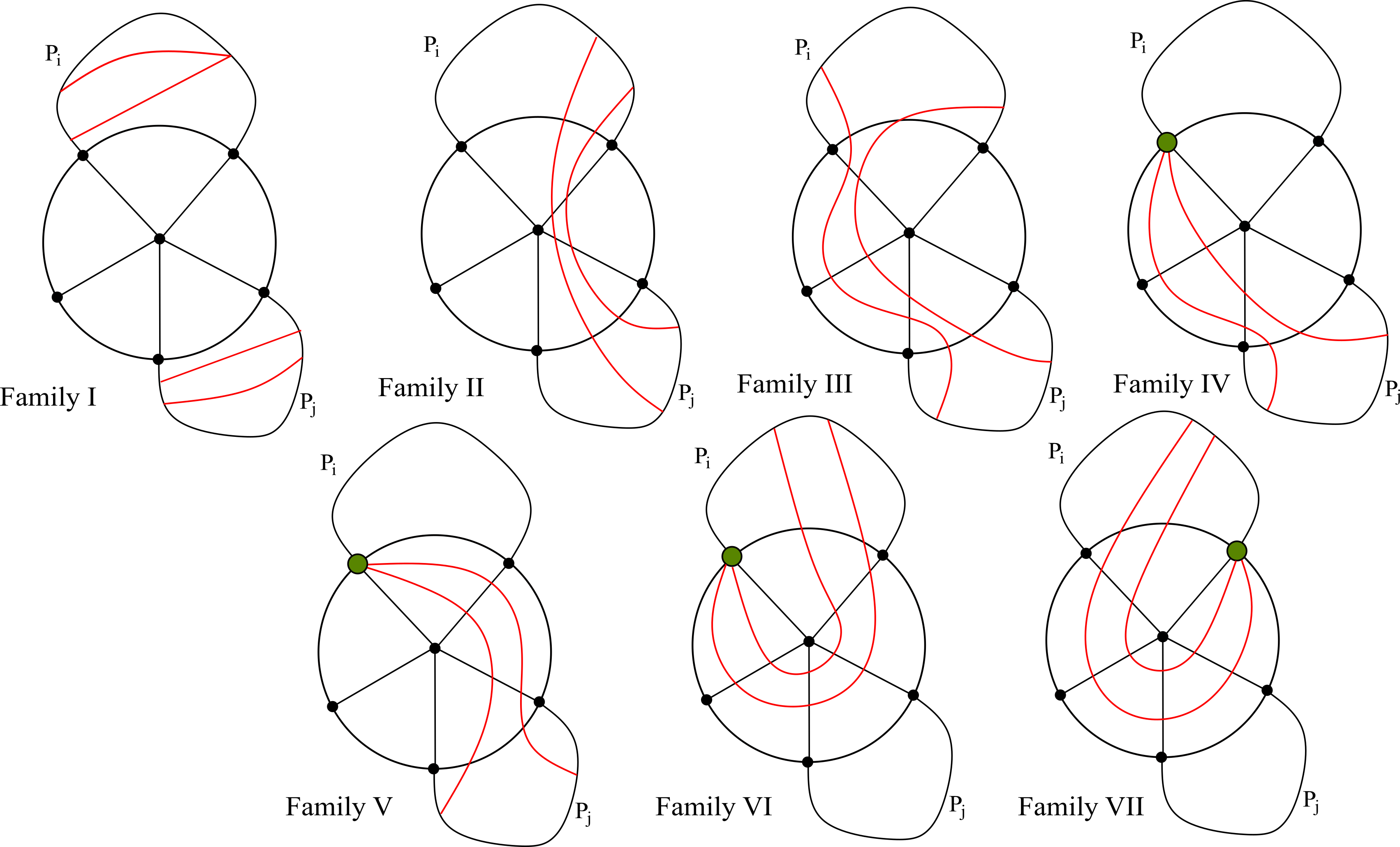}
    \caption{Families I through VII shown in order left to right in the first row and left to right in the second row. Examples of arcs in these families are drawn in \textcolor{red}{red} and the gluing vertices are shown in \textcolor{olive}{green}.}
    \label{fig:4IthroughVII}
\end{figure}

\noindent \textbf{Family VIII.} Arcs and their tagged versions that are coming into the puncture from a non-gluing vertex from polygon $P_i$. See the leftmost figure in Figure \ref{fig:4VIIIthroughend}.

$$d_{k} = \begin{cases}
1 &\text{if } k = a_i\\
\overrightarrow{\text{type }A} &\text{if } k \text{ in } P_i\\
0 &\text{otherwise}
\end{cases} \hspace{1.5em}; \hspace{1.5em} d_{\Bowtie_k} =\begin{cases}
1 &\text{if } k = b_1, \dots, b_n, \text{ or } k = a_i \\
\overrightarrow{\text{type }A} &\text{if } k \text{ in } P_i\\
0 &\text{otherwise}
\end{cases}$$

\noindent \textbf{Family IX.} Tagged arcs coming into the puncture from gluing vertices. See the second from the left figure in Figure \ref{fig:4VIIIthroughend}.

$$d_{\Bowtie_k} =\begin{cases}
1 &\text{if } k \neq b_i \text{ in the wheel} \\
0 &\text{otherwise}
\end{cases}$$

\noindent \textbf{Family X.} Arcs that start and end at gluing vertices must stay inside $D$ up to isotopy. See the third from the left figure in Figure \ref{fig:4VIIIthroughend}. This gives for some $\ell,m \in \{1,2, \dots, k\}$,

$$d_{\alpha} = \begin{cases}
1 &\text{if } \alpha = b_{\ell}, \dots, b_{m}\\
0 &\text{otherwise}
\end{cases}$$

We can now consider when arcs cross arcs in the triangulation more than once. The only way this can happen is if the arcs starts and ends in the same polygon. Notably, neither the start or end point can be a gluing vertex. This yields one final family of arcs:\\

\noindent \textbf{Family XI.} Arcs that begin in polygon $P_i$ wrap around the puncture then return to $P_i$. See the rightmost figure in Figure \ref{fig:4VIIIthroughend}. For coordinates $\alpha$ representing an arc in the sub-triangulation of $D$, we have

$$d_{\alpha} = \begin{cases}
1 &\text{if } \alpha =b_1, \dots, b_k \\
2 &\text{if } \alpha = a_i\\
0 &\text{otherwise}
\end{cases}$$
\noindent For coordinates $\beta$ representing an arc in the sub-triangulation of $P_i$, we have
$$d_{\beta} =\begin{cases}
\overrightarrow{\text{type }A} &\text{if } \beta \text{ in } P_i\\
0 &\text{otherwise}
\end{cases}$$

\begin{figure}[H]
    \centering
    \includegraphics[width=\textwidth]{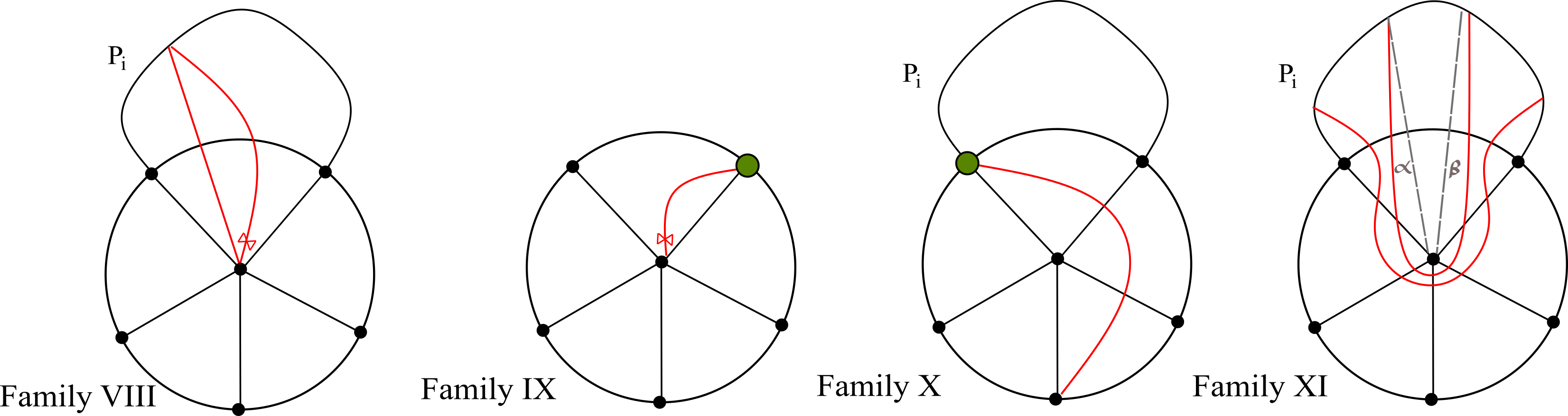}
    \caption{Families VIII through XI shown in order left to right. Examples of arcs in these families are drawn in \textcolor{red}{red} and the gluing vertices are shown in \textcolor{olive}{green}.}
    \label{fig:4VIIIthroughend}
\end{figure}

\end{subsection}
\end{section}

\begin{section}{Proofs of Lemmas}\label{section:topologylemmas}

Now that we have categorized all of the possible $\underline{d}-$vectors, we first prove the lemmas used in Definition \ref{defn: minimalmatching} as well as make some observations about how the structure of these vectors streamlines some of our case work. \allowbreak
\vspace{1em}

\noindent \textbf{Lemma 3.1.} Suppose $\gamma$ is some arc not in $T$ and let $\underline{d} = \text{cross}(\gamma)$. $Q^{\supp(\underline{d})}$, the induced subquiver using vertices $i \in Q_0$ with $d_i > 0$, is connected.\allowbreak
\vspace{1em}

\noindent \textbf{Lemma 3.2.} Suppose $\gamma$ is some arc not in $T$ such that there exists some arc $\tau \in T$ that $\gamma$ crosses twice. Let $\underline{d} = \text{cross}(\gamma)$. $Q^{\supp_2(\underline{d})}$, the induced subquiver using vertices $i \in Q_0$ with $d_i = 2$, is a connected tree.

\begin{proof}
We aim to prove Lemma \ref{lemma:d_vector_supp_connected}. Suppose that $\underline{d} = \text{cross}(\gamma) = (\gamma_{t_1}, \gamma_{t_2}, \dots, \gamma_{t_\ell})$ and suppose that $Q^{\text{supp}}$ has two connected components $R,R'$. Without loss of generality, suppose that the vertices of the quiver $R$ correspond to the section of $\gamma$ on the interval $(t_1,t_p)$ and the vertices of the quiver $R'$ correspond to the section of $\gamma$ on the interval $(t_{p'},t_\ell)$ for $1 \leq p < p' \leq \ell$. As $R \cap R' = \emptyset$, $r \neq r'$ implying that $p \neq p'$. If $r,r'$ are two sides of the same triangle, there is an arrow in $Q$ connecting these two vertices. Since $r,r' \in Q^{\supp}$, this would imply $Q^{\supp}$ must be connected in this case. If $r,r'$ are not in any triangle together, then there exists a sequences of arcs in $T$ corresponding to $\gamma_{t_p} < \gamma_{t_{p+1}} < \dots < \gamma_{t_{p'}}$. The continuity of the image of $\gamma$ guarantees there is a sequence of arrows in $Q$ connecting $r$ to $r'$ implying that $Q^{\supp}$ is connected. 
\end{proof}

\begin{proof}
We aim to prove Lemma \ref{lemma:d_vector_2}. Suppose $\gamma$ is an arc that crosses some arc in the triangulation twice. Without loss of generality, $\gamma$ is type IA or B Family V or type IV Family XI (note that type II and III degenerate to the type IV case). In any case, let $\tau_c$ be the arc in the triangulation that connects the type $D_n$ part of the surface to the type $A_m$ part of the surface that has both of $\gamma$'s endpoints. By the classification of $\underline{d}$-vectors, we know that $\tau_c$ must be crossed twice. If $\gamma$ crosses no other arcs twice, then we are done and $Q^{\supp_2}$ is a connected tree with a single vertex. If $\gamma$ crosses another arc in the triangulation twice, then the segments of the arc $\alpha$ and $\beta$ must both cross this arc. Namely, by the symmetry of $\alpha, \beta$ such an arc must be attached by a sequence of of arcs $\tau_1, \dots, \tau_r$ such that some $\tau_i$ and $\tau_c$ are in the same triangle. Therefore, the vertices $1, \dots, r, c$ must be connected. \allowbreak
\vspace{1em}
We now argue that $Q^{\supp_2}$ cannot have any oriented cycles. Note that in order for $Q^{\supp_2}$ to have a cycle in the type $A_m$ part of the quiver, $\gamma$ must cross all arcs of an internal triangle twice. However, it is impossible up to isotopy for an arc to cross more than one arc of an internal twice. In the type $D_n$ part of the surface, we saw that in types II and III, a 2 can only appear on one of vertices in $\{a,b,c,d\}$. In type IV, the type $D_n$ part of the surface that contains a cycle is the wheel triangulation. Up to isotopy, there is no way to cross the spokes of the wheel twice which implies that we cannot have 2's form a cycle on this part of the surface. Hence, the subquiver $Q^{\supp_2}$ is 2's form a connected tree as desired.
\end{proof}
\end{section}
\end{appendix}

\newpage

\bibliographystyle{alpha}
\bibliography{bibliography}
\end{document}